\documentclass[a4paper, reqno, 10pt]{amsart}

\usepackage{a4wide}
\usepackage{verbatim}
\usepackage[dvips]{color}
\usepackage[usenames,dvipsnames]{xcolor}
\usepackage{amsmath}
\usepackage{amssymb}
\usepackage{amsfonts}
\usepackage{dsfont}
\usepackage[active]{srcltx}
\usepackage[colorlinks,pdfpagelabels,pdfstartview = FitH,bookmarksopen = true,bookmarksnumbered = true,linkcolor = NavyBlue,plainpages = false,hypertexnames = false,citecolor = OliveGreen] {hyperref}
\usepackage[italian, textsize=tiny, color=BlueGreen, backgroundcolor=BlueGreen, linecolor=BlueGreen]{todonotes}
\usepackage{mathtools}
\usepackage{cite}
\usepackage{empheq}
\usepackage{textcomp}
\newcommand\blfootnote[1]{%
	\begingroup
	\renewcommand\thefootnote{}\footnote{#1}%
	\addtocounter{footnote}{-1}%
	\endgroup
}

\allowdisplaybreaks

\title[]{Global gradient estimates for the mixed local and nonlocal problems with measurable nonlinearities}

\author[Byun]{Sun-Sig Byun}
\address{Department of Mathematical Sciences and Research Institute of Mathematics,
	Seoul National University, Seoul 08826, Korea}
\email{byun@snu.ac.kr}

\author[Kumar]{Deepak Kumar}
\address{Research Institute of Mathematics,
	Seoul National University, Seoul 08826, Korea}
\email{deepak.kr0894@gmail.com}

\author[Lee]{Ho-Sik Lee}
\address{Fakult\"at für Mathematik, Universität Bielefeld, 33615 Bielefeld, Germany}
\email{ho-sik.lee@uni-bielefeld.de}

\makeatletter
\@namedef{subjclassname@2020}{\textup{2020} Mathematics Subject Classification}
\makeatother

\newtheorem{theorem}{Theorem}[section]

\newtheorem{proposition}[theorem]{Proposition}
\newtheorem{lemma}[theorem]{Lemma}

\newtheorem{corollary}[theorem]{Corollary}
\theoremstyle{definition}
\newtheorem{definition}[theorem]{Definition}
\newtheorem{remark}[theorem]{Remark}

\numberwithin{equation}{section}

\def\eqn#1$$#2$${\begin{equation}\label#1#2\end{equation}}
\def\charfn_#1{{\raise1.2pt\hbox{$\chi_{\kern-1pt\lower3pt\hbox{{$\scriptstyle#1$}}}$}}}

\makeatletter
\newcommand{\pushright}[1]{\ifmeasuring@#1\else\omit\hfill$\displaystyle#1$\fi\ignorespaces}
\newcommand{\pushleft}[1]{\ifmeasuring@#1\else\omit$\displaystyle#1$\hfill\fi\ignorespaces}
\makeatother

\def\ep{\varepsilon}

\newcommand{\vp}{\varphi}

\def\dist{\operatorname{dist}}

\newcommand{\divo}{\textnormal{div}}

\newcommand{\data}{\texttt{data}}

\def\er{\mathbb R}
\newcommand{\ern}{\mathbb{R}^n}

\def\loc{{\operatorname{loc}}}

\newcommand{\supp}{{\rm supp}}

\def\mean#1{\mathchoice%
	{\mathop{\kern 0.2em\vrule width 0.6em height 0.69678ex depth -0.58065ex
			\kern -0.8em \intop}\nolimits_{\kern -0.4em#1}}%
	{\mathop{\kern 0.1em\vrule width 0.5em height 0.69678ex depth -0.60387ex
			\kern -0.6em \intop}\nolimits_{#1}}%
	{\mathop{\kern 0.1em\vrule width 0.5em height 0.69678ex
			depth -0.60387ex
			\kern -0.6em \intop}\nolimits_{#1}}%
	{\mathop{\kern 0.1em\vrule width 0.5em height 0.69678ex depth -0.60387ex
			\kern -0.6em \intop}\nolimits_{#1}}}

\def\vintslides_#1{\mathchoice%
	{\mathop{\kern 0.1em\vrule width 0.5em height 0.697ex depth -0.581ex
			\kern -0.6em \intop}\nolimits_{\kern -0.4em#1}}%
	{\mathop{\kern 0.1em\vrule width 0.3em height 0.697ex depth -0.604ex
			\kern -0.4em \intop}\nolimits_{#1}}%
	{\mathop{\kern 0.1em\vrule width 0.3em height 0.697ex depth -0.604ex
			\kern -0.4em \intop}\nolimits_{#1}}%
	{\mathop{\kern 0.1em\vrule width 0.3em height 0.697ex depth -0.604ex
			\kern -0.4em \intop}\nolimits_{#1}}}

\newcommand{\aveint}[2]{\mathchoice%
	{\mathop{\kern 0.2em\vrule width 0.6em height 0.69678ex depth -0.58065ex
			\kern -0.8em \intop}\nolimits_{\kern -0.45em#1}^{#2}}%
	{\mathop{\kern 0.1em\vrule width 0.5em height 0.69678ex depth -0.60387ex
			\kern -0.6em \intop}\nolimits_{#1}^{#2}}%
	{\mathop{\kern 0.1em\vrule width 0.5em height 0.69678ex depth -0.60387ex
			\kern -0.6em \intop}\nolimits_{#1}^{#2}}%
	{\mathop{\kern 0.1em\vrule width 0.5em height 0.69678ex depth -0.60387ex
			\kern -0.6em \intop}\nolimits_{#1}^{#2}}}

\newtoks\by
\newtoks\paper
\newtoks\book
\newtoks\jour
\newtoks\yr
\newtoks\pages
\newtoks\vol
\newtoks\publ

\def\ota{{\hbox{\bf ???}}}
\def\cLear{\by=\ota\paper=\ota\book=\ota\jour=\ota\yr=\ota
	\pages=\ota\vol=\ota\publ=\ota}
\def\endpaper{\the\by, \textit{\the\paper},
	{\the\jour} \textbf{\the\vol} (\the\yr), \the\pages.\cLear}
\def\endbook{\the\by, \textit{\the\book},
	\the\publ, \the\yr.\cLear}
\def\endpap{\the\by, \textit{\the\paper}, \the\jour.\cLear}
\def\endproc{\the\by, \textit{\the\paper}, \the\book, \the\publ,
	\the\yr, \the\pages.\cLear}

\begin{document}
	\maketitle
	\begin{abstract}
A non-homogeneous mixed local and nonlocal problem in divergence form is investigated for the validity of the global Calder\'on-Zygmund estimate for the weak solution to the Dirichlet problem of a nonlinear elliptic equation. We establish an optimal Calder\'on-Zygmund theory by finding not only a minimal regularity requirement on the mixed local and nonlocal operators but also a lower level of geometric assumption on the boundary of the domain for the global gradient estimate. More precisely, assuming that the nonlinearity of the local operator, whose prototype is the classical $(-\Delta_p)^1$-Laplace operator with $1<p<\infty$, is measurable in one variable and has a small BMO assumption for the other variables, while the singular kernel associated with the nonlocal $(-\Delta_q)^s$-Laplace operator with $0<s<1<q<\infty$ is merely measurable, and that the boundary of the domain is sufficiently flat in Reifenberg sense, we prove that the gradient of the solution is as integrable as that of the associated non-homogeneous term in the divergence form.
  
	\end{abstract}

\blfootnote{S.-S. Byun was supported by NRF-2022R1A2C1009312, D. Kumar was supported by NRF-2021R1A4A1027378 and H.-S. Lee
was supported by the Deutsche Forschungsgemeinschaft (DFG, German Research Foundation) through (GRK 2235/2
2021 - 282638148) at Bielefeld University.}

\section{Introduction and Main Results}
In the present paper, we study some regularity properties of weak solutions to the mixed local and nonlocal problem whose prototype is given by
\begin{align}\label{eq:model}
	-\divo(|Du|^{p-2}Du)+(-\Delta_q)^su=-\divo(|F|^{p-2}F)\quad\text{in}\quad\Omega,
\end{align}
where $\Omega\subset\ern$ ($n\geq 2$) is a bounded domain and the operator $(-\Delta_q)^s$ is defined by
\begin{align*}
	(-\Delta_q)^s u(x)=P.V.\int_{\ern}\dfrac{|u(x)-u(y)|^{q-2}(u(x)-u(y))}{|x-y|^{n+sq}}\,dy
\end{align*}
with constants $s,p$ and $q$ satisfying
\begin{align}\label{eq:s1pq}
	0<s<1<p, q<\infty
\end{align}
and 
\begin{align}\label{eq:sq<p}
	p>sq.
\end{align}
Here, $F:\Omega\rightarrow\ern$ is a given vector field such that $|F|\in L^p(\Omega)$. The equation \eqref{eq:model} is the Euler-Lagrange equation of the following variational problem:
\begin{align*}
W^{1,1}(\Omega)\ni u\mapsto\int_{\Omega}[|Du|^p-|F|^{p-2}F\cdot Du]\,dx+\int_{\ern}\int_{\ern}\dfrac{|u(x)-u(y)|^{q}}{|x-y|^{n+sq}}\,dx\,dy.
\end{align*}

Mixed local and nonlocal problems have been comprehensively studied in recent years due to their extensive applications in the real world. For instance, these kinds of operators model diffusion patterns with different scales and also appears in the theory of
optimal searching strategies, biomathematics, animal foraging and so on; see for instance, \cite{Dip-Nem-logst} and references therein.  Another interesting application comes from the superposition of two stochastic processes with different scales such as a classical random walk and a L\'evy process, we refer to \cite{Dip-Val-phy} for further details on this. From the mathematical analysis point of view, the mixed local and nonlocal operator can be regarded as the limiting case of the strongly non-homogeneous nonlocal operators,  $(-\Delta_p)^{s_1}+(-\Delta_q)^{s_2}$ with $0<s_1,s_2<1$, as $s_1\to 1$.\par

\subsection{State of the Art}
The $p$-Laplace problem is a specific example of the following $p$-growth problem:
\begin{align}\label{eq:pgrowth}
-\divo(A(x,Du))=-\divo(|F|^{p-2}F)\quad\text{in }\Omega,
\end{align}
where $F:\Omega\rightarrow\ern$ and $A(\cdot,\cdot):\Omega\times\ern\rightarrow\ern$ is a Carath\'{e}odory vector field with $|\partial_zA(x,z)|\simeq |z|^{p-2}$. In the case of \eqref{eq:pgrowth}, $|F|\in L^{\gamma}$ does not imply $|Du|\in L^{\gamma}$, for every $\gamma>p$, if there is no regularity assumption (with respect to the $x$-variable) on the nonlinearity $A$ (cf. \cite{M0}). In this regard, the small BMO assumption is usually considered in the literature, see for instance \cite{D2,KZ,CP,BW0}, but later the weak small BMO assumption as in Definition \ref{def:BMO.l} is recognized as a minimal condition. Moreover, the latter condition allows a big jump in one direction of the variable in $\Omega$, and so our problem is deeply related to the transmission problem, see \cite{BX,ERS}. When $p=2$, the gradient estimate is proved with such a minimal assumption in \cite{BW1,DK,KK2} for the linear case and in \cite{BK} for the linear growth case. In \cite{KimY1}, a Moser type iteration for the regularized problem of \eqref{eq:pgrowth} and an approximation argument are used to consider the $p$-growth case. Recently, the gradient estimate is proved under the setting of generalized $p$-growth problem in \cite{BL1}.\par

The nonlocal nonlinear operators, particularly of the kind $(-\Delta_q)^s$, have been studied extensively in recent years, especially concerning the regularity perspective of solutions involving these operators. 
For instance, Di Castro et al. \cite{DKP} and Iannizzotto et al. \cite{IMS} proved local H\"older continuity results for Dirichlet problems involving the fractional $q$-Laplacian (or the minimizers of corresponding energy functionals). Subsequently, Brasco et al. \cite{BLS} obtained higher H\"older continuity results for local weak solutions. Interested readers may refer to \cite{BKO,BOS,chaker,GKS,GKS2,BKK-PubM,defil-pala} for continuity results of nonlocal problems involving non-standard growth operators. Furthermore, Kuusi et al. \cite{KMS} obtained a self-improving property for weak solutions. Recently, Byun and Kim in \cite{BK-CZ-pfrac}, and Diening and Nowak \cite{DiNo-CZ}, independently, obtained Calder\'{o}n-Zygmund type estimates for nonlocal problems of $p$-growth and irregular kernel coefficients. 
In this direction, we also mention the works \cite{FMSY,MSY,Now-AAG,caff-silv} for some related results on nonlocal problems with certain assumptions on the kernel coefficients.

The linear mixed operators; that is, $p=q=2$ in \eqref{eq:model}, have a rich literature in terms of the regularity theory, see for instance, the notable works of Dipierro, Valdinoci et al. \cite{Biag-math-eng,Biag-cpde,Biag-prse} and references therein.  However, the nonlinear case is relatively new and only a few results are known in this case. Particularly, Garain and Kinnunen \cite{GK1} established the H\"older continuity results by using the De Giorgi-Nash-Moser technique, for $p=q\neq 2$. Recently, Garain and Lindgren \cite{Garain-cvpde} obtained higher H\"older regularity for weak solutions of \eqref{eq:model} with $p=q\geq 2$; that is, the super-quadratic case, by using the techniques of \cite{BLS}. See also \cite{nakamura,shang-zhang-par,byun-song} and references therein for some other regularity results involving the mixed local-nonlocal type operator appearing as in \eqref{eq:model} with $p=q$. As far as the case $p\neq q$ is concerned, we mention the very recent work of De Filippis and Mingione \cite{defilip-ming-mixed}, where local gradient regularity is established for minimizers of functional whose Euler-Lagrange equation involves the mixed local-nonlocal type operator. Specifically, the gradient of a minimizer $u$ of the functional 
\begin{align*}
    \int_{\Omega}[F(Du)-fu]\,dx+\int_{\ern}\int_{\ern}\Phi(u(x)-u(y))K(x,y)\,dx\,dy,
\end{align*}
where $F(z)\simeq |z|^p$, $\Phi(t)\simeq |t|^q$ and $K(x,y)\simeq |x-y|^{-n-sq}$ is a singular kernel, is H\"older continuous if $p>sq$ and $f\in L^\gamma_{\loc}(\Omega)$ with $\gamma>n$. Additionally, they also proved that when the domain $\Omega$ is of type $C^{1,\alpha}$, for some $\alpha\in (0,1)$ and the exterior data is globally Lipschitz continuous (or more generally it satisfies a certain type of global higher order fractional Sobolev regularity), then the minimizer $u$ is globally almost Lipschitz regular. We also refer \cite{BLS2} for regularity results in case of $p\neq q$ for mixed local and nonlocal double phase problems.\par



\subsection{Overview of Analytic Tools}
In this paper, our aim is to obtain the global behavior of the gradient of weak solutions to mixed local-nonlocal problems, which to the best of our knowledge, was not known earlier. To this end, we consider a more general operator than the one consisting of the sum of the $p$-Laplacian and the $s$-fractional $q$-Laplacian involving a divergence type reaction term, see \eqref{eq:inter} with structure conditions \eqref{eq:L} and \eqref{eq:A}. Here, we stress that similar to the purely local problems, the nonlinearity of the leading operator is simply a measurable function in one variable and satisfies a small BMO assumption on the rest of the variables. However, for the nonlocal term, the kernel coefficient is merely assumed to be measurable. To prove the desired global Calder\'{o}n-Zygmund type estimates, the condition on the boundary of the domain also plays an important role. Here, as a minimal condition, we assume that the boundary of the domain is Reifenberg flat set (see Definition \ref{def:reif}). Notably, each Lipschitz domain with small Lipschitz constant is Reifenberg flat, but the domain whose boundary satisfies Reifenberg flat condition may not be a Lipschitz domain and such a domain is allowed to have fractal boundaries. For a further study on Reifenberg flat domain, we refer to \cite{KT,LMS}. We point out that in the local elliptic equation, many related articles for global gradient estimates appear as in \cite{BD,BW0,C2,MP,P1}, but in mixed problems there is no such a result under the setting of rough domains even like Lipschitz domains.\par 

 In order to achieve the aforementioned gradient estimates, we rely on the maximal function-free technique as developed in \cite{AM1} and later used in \cite{col-ming-JFA}. The method  heavily depends upon the use of a suitable comparison argument which requires a global self-improving property (small higher integrability of the gradient) of the weak solution to the original problem. 
We provide a detailed proof of this in Section 3 as the rigorous analysis done here helps obtain comparison estimates in subsequent sections. In this regard, a main difficulty arises from the appearance of a nonlocal tail term. More precisely, the tail term interacts with the exit-time (or stopping) cylinders and it also requires technical cares of the solution near the boundary. To deal with this, we further split the tail term appropriately into finitely many quantities over annuli-type cylinders of size comparable to the exit-time radius which contain information on the gradients of the solution and the exterior data. 
It is worth mentioning that our main results complement and extend the global almost Lipschitz regularity obtained in \cite{defilip-ming-mixed} to the case of less regular domains and more general nonlinearities, namely, we assume the domain to satisfy Reifenberg flatness condition and the nonlinearity to have a small BMO seminorm (see Remark \ref{rem:def-ming-mix} for details).

\subsection{Assumptions and Main Results} 
In this subsection, we present our main problems and corresponding results. We further distinguish into the case of interior estimates for local weak solutions and the case of boundary estimates for the weak solution to Dirichlet problems.  
Throughout the article, for $p,q$ satisfying \eqref{eq:s1pq} and \eqref{eq:sq<p}, we set 
\begin{align}\label{eq:X.l}
\mathbb{X}(\Omega):=\begin{cases}
W^{1,p}(\Omega) \quad\quad\quad\quad\quad\,\mbox{if }q\leq p,\\
W^{1,p}(\Omega)\cap L^\infty(\Omega) \quad\mbox{if }q> p.
\end{cases}
\end{align}
We say that $u\in\mathbb{X}_{\loc}(\Omega)$, if $u\in\mathbb{X}(\tilde{\Omega})$ for all $\tilde{\Omega}\Subset\Omega$. 
 Next, we also set 
	\begin{align}\label{eq:data}
		\data:=\{n,s,p,q,\Lambda\} \quad\mbox{and}\quad \mathcal{M}(v;U):=\begin{cases}
	\|v\|_{L^\infty(U)} \quad\,\,\mbox{ if }q>p,\\  
	1 \quad\quad\quad\quad\quad\mbox{ if } q\leq p,
	\end{cases} \end{align}
 where $U\subset\ern$ and  $v\in \mathbb{X}(U)$.
\subsubsection{\textbf{Interior case}}
We consider the following general problem in this subsection:
\begin{align}\label{eq:inter}
	-\divo(A(x,Du))+\mathcal{L}u=-\divo(|F|^{p-2}F)\quad\text{in}\quad\Omega,
\end{align}
with
\begin{align}\label{eq:L}
	\mathcal{L}u(x):=P.V. \int_{\ern} |u(x)-u(y)|^{q-2}(u(x)-u(y))K(x,y)\,dy.
\end{align}
Here, the Carath\'{e}odory vector field $A(\cdot,\cdot):\Omega\times\ern\rightarrow\ern$ is $C^1(\ern\setminus\{0\})$-regular  in the second variable and satisfies the following condition:
\begin{align}\label{eq:A}
	\begin{cases}
		|A(x,z)||z|+|\partial_z A(x,z)||z|^2\leq \Lambda|z|^{p}\\
		\Lambda^{-1}|z|^{p-2}|\xi|^2\leq \left<\partial_zA(x,z)\xi,\xi\right>
	\end{cases}
\end{align}
for a.e. $x\in\Omega$ and all $z,\xi\in\ern$, and the singular kernel $K(\cdot,\cdot):\ern\times\ern\rightarrow\er$ satisfies
\begin{align}\label{eq:K}
	\Lambda^{-1}\leq K(x,y) |x-y|^{n+sq}\leq \Lambda \quad\mbox{for all }x\neq y\in\ern,
\end{align}
with a constant $\Lambda\geq 1$.
We can assume that $K(x,y)$ is symmetric in the sense that $K(x,y)=K(y,x)$, since otherwise we can replace $K(x,y)$ with $K^*(x,y)=\frac{K(x,y)+K(y,x)}{2}$.\par  
For given $|F|\in L^p_{\loc}(\Omega)$, a function $u\in W^{1,p}_{\loc}(\Omega)\cap W^{s,q}_{\loc}(\Omega)\cap L^{q-1}_{sq}(\ern)$ is said to be a local weak solution of \eqref{eq:inter} if
\begin{align*}
	\begin{split}
		&\int_{\Omega}A(x,Du)\cdot D\vp\,dx+\int_{\ern}\int_{\ern}|u(x)-u(y)|^{q-2}(u(x)-u(y))(\vp(x)-\vp(y))K(x,y)\,dx\,dy\\
		&\quad=\int_{\Omega}|F|^{p-2}F\cdot D\vp\,dx
	\end{split}
\end{align*}
holds for all $\vp\in W^{1,p}(\Omega)\cap W^{s,q}(\Omega)$ with compact support contained in $\Omega$. Here, the fractional Sobolev space $W^{s,q}(\Omega)$ and the tail space $L^{q-1}_{sq}(\ern)$ will be defined later in Section 2.

We next introduce the weak small BMO assumption on the vector field $A$. For this,  we write the $n$-dimensional cylinder as $Q_r(y)=(y_1-r,y_1+r)\times B'_r(y')$, where  $y\in(y_1,y')\in\er\times\er^{n-1}$ and $B'_r(y')$ is the ball in $\er^{n-1}$ with center $y'$ and radius $r>0$ (see Section 2 for precise definitions).
\begin{definition}\label{def:BMO.l}
 For given parameters $\delta\in(0,1)$ and $R>0$, we say that \textit{$A$ is $(\delta,R)$-vanishing of codimension 1} if
	\begin{align}\label{eq:BMO.l}
		\sup_{0<r\leq R}\sup_{y\in\Omega}\mean{Q_r(y)}\kappa(A,Q_r(y))(x)\,dx\leq\delta,
	\end{align}
 where the function $\kappa(A,Q_r(y))(\cdot):Q_r(y)\rightarrow\er$ is defined as
\begin{align}\label{eq:kappa}
	\kappa(A,Q_r(y))(x)=\sup_{z\in\ern\setminus\{0\}}\dfrac{|A(x_1,x',z)-\bar{A}_{B'_r(y')}(x_1,z)|}{|z|^{p-1}}
\end{align}
with $\bar{A}_{B'_r(y')}(\cdot,\cdot):(y_1-r,y_1+r)\times\ern\rightarrow\ern$ given by
\begin{align*}
	\bar{A}_{B'_r(y')}(x_1,z)=\mean{B'_r(y')}A(x_1,x',z)\,dx'.
\end{align*}
\end{definition}

\begin{remark}
With the aid of \eqref{eq:A} and \eqref{eq:kappa}, we notice that
	\begin{align}\label{eq:kappa.2L}
		|\kappa(A,Q_r(y))(x)|\leq 2\Lambda\quad\text{for a.e.  }x\in Q_r(y)
	\end{align}
 holds without assuming any vanishing condition on $A$. Moreover, a typical example for $A(x,z)$ satisfying \eqref{eq:BMO.l} is the following:
	\begin{align*}
		A(x,z)=a_1(x_1)a_2(x')|z|^{p-2}z,
	\end{align*}
	where $a_1(\cdot):\er\rightarrow\er$ and $a_2(\cdot):\er^{n-1}\rightarrow\er$ with $\Lambda^{-1}\leq a_i(\cdot)\leq \Lambda$ for $i=1,2$, and $[a_2]_{\text{BMO}(\er^{n-1})}\leq\delta/\sqrt{\Lambda}$. 
\end{remark}

Having $\mathbb{X}$ given by \eqref{eq:X.l}, and $\data$ and $\mathcal{M}$ as in \eqref{eq:data} in mind, we now state the main theorem in the interior case. 

\begin{theorem}\label{thm:l}
 Let $u\in \mathbb{X}_{\loc}(\Omega)\cap  L^{q-1}_{sq}(\ern)$ be a local weak solution of \eqref{eq:inter} under the assumptions \eqref{eq:s1pq}, \eqref{eq:sq<p}, \eqref{eq:A} and \eqref{eq:K}. Suppose that $|F|\in L^{\gamma}_{\loc}(\Omega)$ for $\gamma>p$. Then, for every $\widetilde{\Omega}\Subset\Omega$, there exist constants $\delta$ and $r_*\in (0,1)$ both depending only on $\data$, $\mathcal{M}(u;\widetilde{\Omega})$ and $\gamma$ such that if $A$ is $(\delta,r_*)$-vanishing of codimension 1, then we have $|Du|\in L^{\gamma}(\widetilde{\Omega})$. Moreover, there exists a constant $c=c(\data, \mathcal{M}(u;\widetilde{\Omega}),\gamma)$ such that 
	\begin{align*}
		\left(\mean{Q_{R}(x_0)}|Du|^{\gamma}\,dx\right)^\frac{1}{\gamma}&\leq c \left(\mean{Q_{2R}(x_0)}|Du|^p\,dx\right)^\frac{1}{p}
		+c\left(\mean{Q_{2R}(x_0)}(|F|^{\gamma}+1)\,dx\right)^\frac{1}{\gamma}\nonumber\\
		&\quad+c\left(R^{1-\varrho}\int_{\ern\setminus Q_{2R}(x_0)}\dfrac{|u(y)-(u)_{2R}|^{q-1}}{|y-x_0|^{n+sq}}\,dy\right)^\frac{1}{m-1}
	\end{align*}
 holds for all $\varrho\in(0,m-qs)$ and $Q_{2R}(x_0)\subset\widetilde{\Omega}$, where $m:=\min\{p,q\}$ and $R\in (0,r_*/2)$.
\end{theorem}

Note that Lemma \ref{lem:embed2}, $\mathbb{X}_{\loc}(\Omega)\hookrightarrow W^{s,q}_{\loc}(\Omega)$ holds for $s,p$ and $q$ satisfying \eqref{eq:s1pq} and \eqref{eq:sq<p}.


\subsubsection{\textbf{Boundary case}}
With  $|F|\in L^p(\Omega)$ and $g\in W^{1,p}(\Omega)\cap W^{s,q}(\Omega')\cap L^{q-1}_{sq}(\ern)$ for some bounded open sets $\Omega\Subset\Omega'\subset\ern$, we next consider the following Dirichlet problem:
\begin{equation}\label{eq:bdary}
	\left\{ \begin{array}{llr}
		-\divo(A(x,Du))+&\mathcal{L}u =-\divo(|F|^{p-2}F)\quad\text{in}\quad\Omega,\\
		& \ \ u=g\quad\quad\quad\quad\quad\quad\quad\text{in}\quad \ern\setminus\Omega,
	\end{array}
	\right. 
\end{equation}
where  the Carath\'{e}odory vector field $A(\cdot,\cdot)$ satisfies \eqref{eq:A} and $\mathcal{L}$ is defined in \eqref{eq:L} with the singular kernel $K(\cdot,\cdot)$ satisfying \eqref{eq:K}. For fixed $\Omega$ and $\Omega'$ as before,  we define the function space
\begin{align*}
	\begin{split}
		&\mathbb{X}_{g}(\Omega;\Omega')\\
		&:=\left\{v:\ern\rightarrow\er: v\in W^{1,p}(\Omega)\cap W^{s,q}(\Omega')\cap L^{q-1}_{sq}(\ern), v-g\in W^{1,p}_0(\Omega), v=g\,\,\text{a.e. in}\,\,\ern\setminus\Omega\right\}.
	\end{split}
\end{align*}
We also write $\mathbb{X}_0(\Omega;\Omega')$ as the set $\mathbb{X}_{g}(\Omega;\Omega')$ with $g\equiv 0$. Note that 
$\mathbb{X}_{{g}}(\Omega;\Omega')$ is non-empty since $g\in \mathbb{X}_{{g}}(\Omega;\Omega')$.\par 
We say that $u\in\mathbb{X}_{g}(\Omega;\Omega')$ is a weak solution of \eqref{eq:bdary}, if the following holds
\begin{align*}
	\begin{split}
		&\int_{\Omega}A(x,Du)\cdot D\vp\,dx+\int_{\ern}\int_{\ern}|u(x)-u(y)|^{q-2}(u(x)-u(y))(\vp(x)-\vp(y))K(x,y)\,dx\,dy\\
		&\quad=\int_{\Omega}|F|^{p-2}F\cdot D\vp\,dx
	\end{split}
\end{align*}
for all $\vp\in \mathbb{X}_{0}(\Omega;\Omega')$.

Then we state Reifenberg flatness condition on the domain as below.
\begin{definition}\label{def:reif}
 Let two parameters $\delta\in(0,\frac{1}{2^{n+1}})$ and $R>0$ be given. A bounded open set $\Omega$ is said to be $(\delta,R)$-Reifenberg flat if for each $x_0\in\partial\Omega$ and $r\in(0,R]$, there is a new coordinate system $\{x_1,\dots,x_n\}$ (which may depend on $r$ and $x_0$) such that $x_0=0$ in this new coordinate system and we have
\begin{align*}
		Q_r(0)\cap\{x_1>\delta r\}\subset \Omega\cap Q_r(0)\subset Q_r(0)\cap\{x_1>-\delta r\}.
	\end{align*}	
\end{definition}

Employing the above definition on $\Omega$ and Definition \ref{def:BMO.l} on the nonlinearity $A$, we introduce the following optimal smallness assumption on $A$ and $\Omega$.

\begin{definition}\label{def:BMO.b}
 For given parameters $\delta\in(0,\frac{1}{2^{n+1}})$ and $R>0$, we say that \textit{$(A,\Omega)$ is $(\delta,R)$-vanishing of codimension 1} if for any point $y\in\Omega$ and $r\in (0,R]$ with
    \begin{align*}
      \dist(y,\partial\Omega) =\min_{x_0\in\partial\Omega}\dist(y,x_0)>\sqrt{2}r,	
     \end{align*}
there exists a coordinate system depending on $y$ and $r$, whose variables are still denoted by $x =(x_1, x')\in\er\times\er^{n-1}$, so that in this coordinate system $y$ is the origin and there holds
	\begin{align*}
		\mean{Q_{r}(y)}\kappa(A,Q_r(y))(x)\,dx\leq\delta.
	\end{align*} 
 On the other hand, for any point $y\in\Omega$ and $r\in (0,R]$ with 
	\begin{align*}
		\dist(y,\partial\Omega)=\dist(y,x_0)\leq\sqrt{2}r
	\end{align*}
	for some $x_0\in\partial\Omega$, there exists a new coordinate system depending on $y$ and $r$, whose variables are still denoted by $x=(x_1,x')\in\er\times\er^{n-1}$, such that in this coordinate system $x_0$ is the origin, 
 \begin{align*}
		Q_r(0)\cap\{x_1>\delta r\}\subset \Omega\cap Q_r(0)\subset Q_r(0)\cap\{x_1>-\delta r\}
	\end{align*}
 and
	\begin{align*}
		\mean{Q_r(x_0)}\kappa(A,Q_r(x_0))(x)\,dx\leq\delta
	\end{align*}
	hold.
\end{definition}


Then in conjunction with $\mathbb{X}$, $\data$ and $\mathcal{M}$ given by \eqref{eq:X.l} and \eqref{eq:data}, we are ready to state the main theorem in the boundary case.

\begin{theorem}\label{thm:b}
Let $u\in \mathbb{X}(\Omega)\cap\mathbb{X}_{g}(\Omega;\Omega')$ be a weak solution to problem \eqref{eq:bdary} under the assumptions \eqref{eq:s1pq}, \eqref{eq:sq<p}, \eqref{eq:A}, \eqref{eq:K}, and $g\in \mathbb{X}(\Omega')$. Assume that for $\gamma>p$, $|F|\in L^{\gamma}(\Omega)$ and $|Dg|\in L^\gamma(\Omega')$. Then there exist constants $\delta$ and $r_*$ both depending only on $\data$, $\mathcal{M}(u;\Omega)$, $\mathcal{M}(g;\Omega')$ and $\gamma$ such that if $(A,\Omega)$ is $(\delta,r_*)$-vanishing of codimension 1, then we have $|Du|\in L^{\gamma}(\Omega_{R/2}(x_0))$, for any $x_0\in \overline{\Omega}$ and $R\in (0,r_*/2)$ satisfying $Q_{R}(x_0)\Subset\Omega'$. Moreover, there exists a constant $c=c(\data,\mathcal{M}(u;\Omega),\mathcal{M}(g;\Omega'),\gamma)$ such that
	\begin{align*}
		\left(\mean{\Omega_\frac{R}{2}(x_0)}|Du|^{\gamma}\,dx\right)^\frac{1}{\gamma}
		&\leq c \left(\mean{\Omega_{R}(x_0)}|Du|^{p}\,dx\right)^{\frac{1}{p}}+c\left(\mean{\Omega_R(x_0)}(|F|^\gamma+1)\,dx\right)^\frac{1}{\gamma}
			\nonumber\\
		& +c\left(\mean{Q_{R}(x_0)}|Dg|^{\gamma}\,dx\right)^\frac{1}{\gamma}
		+c \left(R^{1-\varrho}\int_{\ern\setminus Q_{R}(x_0)}\dfrac{|u(y)-(u)_{R}|^{q-1}}{|y-x_0|^{n+sq}}\,dy\right)^\frac{1}{m-1}
		\end{align*}	
holds for all $\varrho\in(0,m-qs)$ with $m:=\min\{p,q\}$.
\end{theorem}

by Proposition \ref{prop:h.i.b} below, if $\gamma$ is close to $p$, then we do not need to assume any vanishing condition on the vector field $A$.

\begin{remark}
 We remark that our preceding main theorem in the boundary case remains valid under the local assumption on the exterior data $g$ with respect to the domain $\Omega'$, namely $g\in\mathbb{X}_{\loc}(\Omega')$ and $|Dg|\in L^{\gamma}_{\loc}(\Omega')$, for $\gamma>p$. Since we can find $\Omega\Subset\Omega''\Subset\Omega'$, so that the aforementioned assumptions are valid globally in $\Omega''$, we choose to label $\Omega''$ as $\Omega'$ itself to avoid introducing several domains.
\end{remark}

\begin{remark}\label{rem:def-ming-mix}
We point out that our global gradient estimate of Theorem \ref{thm:b} together with Theorem \ref{thm:l} complements the global almost Lipschitz regularity of \cite[Theorem 4]{defilip-ming-mixed} in the case of less regular domains and more general nonlinearities. For instance, let us assume that $g\in W_{\loc}^{1,\infty}(\Omega')$, for some $\Omega'\Supset\Omega$ with $\Omega$ being a bounded Lipschitz domain. Then for given $f\in L^n(\Omega)$, let $u\in \mathbb{X}(\Omega)\cap\mathbb{X}_{g}(\Omega;\Omega')$ be the weak solution to the exterior data problem \eqref{eq:bdary} with right-hand side $f$ in place of $-{\rm div}(|F|^{p-2}F)$. By the standard regularity (and existence) theory, there exists $F:\Omega\rightarrow\ern$ satisfying $-{\rm div} (|F|^{p-2}F)=f$ in $\Omega$ and $|F|\in L^{\gamma}(\Omega)$ for any $\gamma>1$. Thus from Theorem \ref{thm:l} and Theorem \ref{thm:b}, $|Du|\in L^\gamma(\Omega)$ for each $\gamma>p$ provided $(A,\Omega)$ is $(\delta,R)$-vanishing of codimension 1 with sufficiently small $\delta\in(0,1)$ and $R>0$ depending on $\gamma$. Moreover, by \cite[Proposition 2.1]{defilip-ming-mixed}, we have $u\in L^\infty(\Omega)$. Consequently, we get that $u\in C^{\alpha}(\Omega)$ for each $\alpha<1$ by the Sobolev embedding theorem, provided $\delta$ and $R$ are sufficiently small depending on $\alpha$.
\end{remark}


The rest of the paper is organized as follows. In Section 2, we provide some additional basic notations, function spaces and preliminaries which will be used in the subsequent part of the paper. We also prove the existence result in this section. In Section 3, we prove a small higher integrability result for the gradient of weak solutions. Section 4 is devoted to showing the local estimates on the gradient which is Theorem \ref{thm:l}. Finally, in Section 5, we prove the boundary gradient estimate Theorem \ref{thm:b}.

\section{Preliminaries}
The notation $B_r(x_0)$ means the open ball in $\ern$ centered at $x_0\in\ern$ with radius $r>0$. We sometimes abbreviate $B_r(x_0)$ with $B_r$ when its center is not important in the context. Furthermore, for a point $y\in(y_1,y')\in\er\times\er^{n-1}$ and a radius $r>0$, we write
	\begin{itemize}
		\item $B'_r(y')=\{x'=(x_2,\dots,x_n)\in\er^{n-1}:|x'-y'|<r\}$.
		\item $Q_r(y)=(y_1-r,y_1+r)\times B'_r(y')$.
		\item $Q^+_r(0)=Q_r(0)\cap\{x:x_1>0\}$, $Q^+_r(y)=Q^+_r(0)+y$.
		\item $\Omega_{r}(y)=Q_{r}(y)\cap\Omega$, $\partial_{w}\Omega_{r}(y)=\partial\Omega\cap Q_{r}(y)$.
		\item $T_{r}(0)=Q_r(0)\cap\{x:x_1=0\}$, $T_{r}(y)=T_r(0)+y$.
	\end{itemize}
We omit the center if it is clear from the context.

Generic positive constants are denoted as $c$ and their value may differ from line to line. We often write the dependence of $c$ in the parenthesis as $c=c(\data)$. Sometimes to mention additional dependencies of the constant, we may omit the dependency on $\data$.
For $X,Y\in \mathbb{R}$, we say $X\lesssim Y$, if there exists a positive constant $c$ such that $X\leq c Y$. Moreover, by $X\simeq Y$, we mean that $X\lesssim Y$ and $Y\lesssim X$ hold. For a measurable function $v$ and a measurable set $S$ with finite measure, the integral average over $S$ is denoted by
\begin{align*}
	(v)_S=\mean{S}v\,dx=\dfrac{1}{|S|}\int_{S}v\,dx.
\end{align*} 
In particular, if $S=Q_{r}$ for some $Q_r\subset\ern$, then we simply denote $(v)_r=(v)_{Q_r}$.	

For $0<s<1\leq p<\infty$ and an open set $U\subset\ern$, the fractional Sobolev space $W^{s,p}(U)$ is defined as below:
\begin{align*}
	W^{s,p}(U):=\{ v:U\rightarrow\er \ : \ \|v\|_{W^{s,p}(U)}<\infty\}  
\end{align*} 
equipped with the norm $\|\cdot\|_{W^{s,p}(U)}$, where 
\begin{align*}
	\|v\|_{W^{s,p}(U)}=\|v\|_{L^p(U)}+[v]_{W^{s,p}(U)}:=\left(\int_{U}|v|^p\,dx\right)^{\frac{1}{p}}+\left(\int_{U}\int_{U}\dfrac{|v(x)-v(y)|^p}{|x-y|^{n+sp}}\,dxdy\right)^{\frac{1}{p}}.
\end{align*}
We denote the closure of $C^{\infty}_{0}(U)$ in $W^{s,p}(U)$ as $W^{s,p}_0(U)$. Also, $v\in W^{s,p}_{\loc}(U)$ means that $v\in W^{s,p}(\tilde{U})$ for all $\tilde{U}\Subset U$. On the other hand, for $0<\ell,\alpha<\infty$, we define the tail space as 
\begin{align*}
   L^{\ell}_{\alpha}(\mathbb R^n) = \bigg\{ u:\mathbb{R}^{n}\to \mathbb{R}\text{ is a measurable function } : \int_{\mathbb R^n} \frac{|u(y)|^{\ell}}{(1+|y|)^{n+\alpha}}\,dy <\infty \bigg\}.
\end{align*}
\par

For $\ell\in[1,\infty)$, we denote the critical Sobolev exponent $\ell^*$ 
by
\begin{align*}
	\ell^*:=
	\begin{cases}
		\dfrac{\ell n}{n-\ell}&\text{if}\,\,\ell<n\\
		\text{any number in }(\ell,\infty)&\text{if}\,\,\ell\geq n.
	\end{cases}
\end{align*}
We also define the constant $\theta\in (s,1)$ as
\begin{align}\label{eq:theta1}
		\theta:=
	\begin{cases}
		1 &\text{if}\,\, q\leq p,\\
		\frac{p}{q} &\text{if}\,\, p<q<\frac{p}{s}
	\end{cases}
\end{align} 
  and fix $\tau\in(0,1)$ satisfying (thanks to \eqref{eq:sq<p})
	\begin{align}\label{eq:tau}
		\max\left\{s,\dfrac{q-1}{q}\theta\right\}<\tau<\theta.
	\end{align} 
Now we record some useful embedding results as below. 
\begin{lemma}[Sobolev-Poincar\'{e} inequalilty]\label{lem:SP}
 Let $1< \ell<\infty$ and $r>0$. Then, for any $v\in W^{1,\ell}(Q_r)$, we have
    \begin{align*}
      \left(\mean{Q_{r}}\left|\dfrac{v(x)-(v)_{r}}{r}\right|^{\ell^*}\,dx\right)^{\frac{1}{\ell^*}}\leq c\left(\mean{Q_{r}}|Dv|^{\ell}\,dx\right)^{\frac{1}{\ell}},
    \end{align*}
  where $c=c(n,\ell)$ is a constant.	
\end{lemma}
\begin{corollary}\label{cor:emb1}
 Assume that $0<s<1<p,q<\infty$ and $p>sq$ hold. 
 Let $w\in\mathbb{X}(Q_r)$, where the space $\mathbb{X}$ is as defined in \eqref{eq:X.l} and $r>0$. 
 Then for every $\frac{n}{n+p}\leq \sigma\leq 1$, we have 
	\begin{align*}
		\left(\mean{Q_{r}}|w(x)-(w)_{r}|^q\,dx\right)^{\frac{1}{q}}
		\leq c \, r^{\theta} \|w\|^{1-\theta}_{L^{\infty}(Q_{r})}\left(\mean{Q_{r}}|Dw|^{p\sigma}\,dx\right)^{\frac{\theta}{p\sigma}}
	\end{align*}
	for some $c=c(n,p,q)$.
\end{corollary}
\begin{proof}
We first observe by the Sobolev-Poincar\'{e} inequality that for all $\frac{n}{n+p}\leq \sigma\leq 1$, there holds
\begin{align}\label{eq:emb.4}
 \left(\mean{Q_{r}}|w(x)-(w)_{r}|^p\,dx\right)^{\frac{1}{p}} \leq c \, r\left(\mean{Q_{r}}|Dw|^{p\sigma}\,dx\right)^{\frac{1}{p\sigma}}.
\end{align}
For the case, $q\leq p$, the required result of the Corollary follows by \eqref{eq:emb.4} and H\"older's inequality. Moreover, for the case $q>p$, noting that $w\in L^\infty(Q_r)$, we have
\begin{align*}
   \left(\mean{Q_{r}}|w(x)-(w)_{r}|^q\,dx\right)^{\frac{1}{q}} \leq c \, \|w\|^{1-p/q}_{L^{\infty}(Q_{r})} \left(\Big(\mean{Q_{r}}|w(x)-(w)_{r}|^p\,dx\Big)^{\frac{1}{p}}\right)^\frac{p}{q}, 
\end{align*}
which yields the result of the Corollary by using \eqref{eq:emb.4} for $\theta=p/q$.
\end{proof}

\begin{lemma}\label{lem:embed2}
 Assume that the numbers $0<s<1<p,q<\infty$ are such that $p>sq$.
 Let $w\in\mathbb{X}(Q_r)$ for $r>0$, 
 then $w\in W^{s,q}(Q_{r})$ and the following holds
	\begin{align*}
		\left(\mean{Q_{r}}\int_{Q_{r}}\dfrac{|w(x)-w(y)|^q}{|x-y|^{n+sq}}\,dx\,dy\right)^{\frac{1}{q}} 
		\leq c\, r^{\theta-s} \|w\|^{1-\theta}_{L^{\infty}(Q_{r})}\left(\mean{Q_{r}}|Dw|^{p}\,dx\right)^{\frac{\theta}{p}}
	\end{align*}
	for $c=c(n,s,p,q)$.
\end{lemma}
\begin{proof}
 We first consider the case $r=1$. From \cite[Proposition 2.2]{DPV}, for $0<t<1\leq\ell<\infty$ we have
\begin{align*}
\|v\|_{W^{t,\ell}(Q_1)}\leq c \|v\|_{W^{1,\ell}(Q_1)},
\end{align*}
where $c=c(n,\ell,t)$. Set $v(x)=w(x)-(w)_{Q_1}$ to obtain
\begin{align*}
\mean{Q_1}\int_{Q_1}\frac{|w(x)-w(y)|^\ell}{|x-y|^{n+\ell t}}\,dx\,dy \leq c \mean{Q_1}|w(x)-(w)_{Q_1}|^\ell \,dx + c\mean{Q_1}|Dw(x)|^\ell \,dx.
\end{align*}
 After rescaling and using the Poincar\'e inequality, it follows that
\begin{align}\label{eq:emb.1}
\mean{Q_r}\int_{Q_r}\frac{|w(x)-w(y)|^\ell}{|x-y|^{n+\ell t}}\,dx\,dy &\leq c r^{-\ell t}\mean{Q_r}|w(x)-(w)_{Q_r}|^\ell \,dx + c r^{\ell-\ell t} \mean{Q_r}|Dw(x)|^\ell \,dx \nonumber\\
&\leq c r^{\ell-\ell t} \mean{Q_r}|Dw(x)|^\ell \,dx.
\end{align}
In the case of $q\leq p$, taking $(\ell,t)=(q,s)$, using H\"{o}lder's inequality and recalling the definition of $\theta$ yield the result of the lemma. In the case of $q>p$, since $w\in L^\infty(Q_r)$, we deduce that 
	\begin{align}\label{eq:emb.2}
		\mean{Q_r}\int_{Q_r}\frac{|w(x)-w(y)|^q}{|x-y|^{n+q s}}\,dx\,dy &\leq c \|w\|^{q-p}_{L^\infty(Q_r)} \mean{Q_r}\int_{Q_r}\frac{|w(x)-w(y)|^p}{|x-y|^{n+p\frac{sq}{p}}}\,dx\,dy \nonumber\\
		&\leq c r^{p-sq} \|w\|^{q-p}_{L^\infty(Q_r)}\mean{Q_r}|Dw(x)|^p \,dx,
	\end{align}
 where on the last line we used \eqref{eq:emb.1} for the choice $(\ell,t)=(p,sq/p)$. Combining \eqref{eq:emb.1} and \eqref{eq:emb.2}, and again recalling \eqref{eq:theta1}, we conclude the proof.
\end{proof}

Now we prove the following existence result.
\begin{theorem}
 Suppose that \eqref{eq:s1pq}, \eqref{eq:A} and \eqref{eq:K} hold. Let $\Omega\Subset\Omega'\subset\ern$ be two bounded domains and let  $g\in W^{1,p}(\Omega)\cap W^{s,q}(\Omega')\cap L^{q-1}_{sq}(\ern)$. If $|F|\in L^p(\Omega)$, then there exists a unique solution $u\in \mathbb{X}_{g}(\Omega;\Omega')$ of the problem \eqref{eq:bdary}.
\end{theorem}
\begin{proof}
 We define the functional $\mathcal{A}: \mathbb{X}_0(\Omega;\Omega')\to \big(W^{1,p}(\Omega)\cap W^{s,q}(\Omega')\big)^*$ by 
	\begin{align*}
		\langle \mathcal{A}(v), \phi \rangle &= \int_{\Omega}A(x,Dv)\cdot D\phi \,dx\\
  &\quad+ \left[ \int_{\Omega'}\int_{\Omega'}[v(x)+g(x)-v(y)-g(y)]^{q-1}\big(\phi(x)-\phi(y)\big)K(x,y)\,dx\,dy \right.\\
		& \left.\quad\quad + 2\int_{\Omega}\int_{\ern\setminus\Omega'}[v(x)+g(x)-v(y)]^{q-1}\phi(x)K(x,y)\,dx\,dy\right] \\
		&=: \langle \mathcal{A}_p(v),\phi \rangle + \langle \mathcal{A}_q(v),\phi\rangle,
	\end{align*}
for all $\phi\in W^{1,p}(\Omega)\cap W^{s,q}(\Omega')$, where we have set $[\xi]^{\ell-1}=|\xi|^{\ell-2}\xi$, for $\ell>1$ and $\xi\in\ern$. Note that the operator $\mathcal{A}_q$ resembles the  operator $\mathcal{A}_0$ defined in \cite[Proposition 2.12, page 802]{BLS}. Consequently, it can be easily verified that the operator $\mathcal{A}$ is monotone,  coercive and weakly continuous in $\mathbb{X}_0(\Omega;\Omega')$; that is,
\begin{itemize}
    \item[(i)] for all $v,w\in \mathbb{X}_0(\Omega;\Omega')$, there holds
    \begin{align*}
        \langle \mathcal{A}(v)-\mathcal{A}(w), v-w \rangle \geq 0;
    \end{align*}
    \item[(ii)] the following holds 
    \begin{align*}
        \lim_{\|u\|_{W^{1,p}(\Omega)\cap W^{s,q}(\Omega')}\to\infty}\frac{\langle \mathcal{A}(u), u \rangle}{\|u\|_{W^{1,p}(\Omega)\cap W^{s,q}(\Omega')}}=\infty;
    \end{align*}
    \item[(iii)] if $\{u_k\}\subset \mathbb{X}_0(\Omega;\Omega')$ is such that $u_k\to u$ in $W^{1,p}(\Omega)\cap W^{s,q}(\Omega')$, then \begin{align*}
        \lim_{k\to\infty} \langle \mathcal{A}(u_k)-\mathcal{A}(u), \phi \rangle =0 \quad\mbox{for all }\phi\in W^{1,p}(\Omega)\cap W^{s,q}(\Omega').
    \end{align*}
\end{itemize}
Moreover, the map $T_F:W^{1,p}(\Omega)\cap W^{s,q}(\Omega')\to\mathbb{R}$ defined by $T_F(\vp)=\int_{\Omega}|F|^{p-2}F\cdot D\vp$ is in the dual of $W^{1,p}(\Omega)\cap W^{s,q}(\Omega')$.
Hence, the existence and uniqueness of $u$ follow by the standard theory of monotone operators.
\end{proof}

\begin{remark}
From now on, we implicitly use the following measure density property.  
	\begin{align*}
		\text{If } \,\, \Omega\,\,\text{ is }\,(\delta,4R_0)\text{--Reifenberg flat with }\,\delta\in\left(0,\frac{1}{2^{n+1}}\right)\,\text{ and }\,R_0>0,
	\end{align*}
	then there holds:
	\begin{align}\label{eq:m.den1}
		\sup_{0<r\leq 4R_0}\sup_{y\in\Omega}\dfrac{|Q_r(y)|}{|\Omega\cap Q_r(y)|}\leq 4^n \quad\mbox{and}\quad \inf_{0<r\leq 4R_0}\inf_{y\in\partial\Omega}\dfrac{|Q_r(y)\cap\Omega^c|}{|Q_r(y)|}\geq 4^{-n}.
	\end{align}
\end{remark}

Now we can show Sobolev-Poincar\'{e} inequality on the boundary of the domain with the assumption \eqref{eq:m.den1}.

\begin{lemma}[See \cite{BBDL}]\label{lem:embed3}
 Let $\Omega$ be an open set satisfying \eqref{eq:m.den1} with $R_0>0$. Let $Q_r\equiv Q_r(x_0)$ with $x_0\in\Omega$ be such that $Q_{\frac{3}{2}r}(x_0)\not\subset\Omega$, for $r\in (0,R_0)$. Then, for any $v\in W^{1,\ell}(\Omega_{2r})$ with $v=0$ on $\partial_w\Omega_{2r}$, we have
	\begin{align*}
		\left(\mean{\Omega_{2r}}\left|\dfrac{v(x)}{r}\right|^{\ell^*}\,dx\right)^{\frac{1}{\ell^*}}\leq c\left(\mean{\Omega_{2r}}|Dv|^{\ell}\,dx\right)^{\frac{1}{\ell}},
	\end{align*}
	where $c=c(n,\ell)$ is a constant.
\end{lemma}

We end this section with the following lemma.

\begin{lemma}[See \cite{BBL,DFTW}]
 Define the vector field $V(\cdot):\ern\rightarrow\ern$ by
	\begin{align*}
		V(z):=|z|^{\frac{p-2}{2}}z.
	\end{align*} 
 Then, for all $z_1,z_2\in\ern$ and $\vartheta\in(0,1)$, we have
	\begin{align}\label{eq:AV}
		\left(A(x,z_1)-A(x,z_2)\right)\cdot(z_1-z_2)\eqsim|V(z_1)-V(z_2)|^2 \quad\mbox{for a.e. }x\in\Omega,
	\end{align}
	and 
	\begin{align}\label{eq:V,theta}
		|z_1-z_2|^p\lesssim \vartheta|z_1|^p+c(\vartheta)|V(z_1)-V(z_2)|^2.
	\end{align}
 \end{lemma}

\section{Self-improving Property}
In this section, we prove the higher integrability (for small exponents) of the gradient of weak solutions. We present the proof only for the boundary case since the interior case can be obtained similarly. Throughout this section, \eqref{eq:s1pq}, \eqref{eq:A} and \eqref{eq:K} are assumed. We start with the following Caccioppoli type inequality.

\begin{lemma}[Caccioppoli inequality]\label{lem:cacc.b}
Let $u\in\mathbb{X}_g(\Omega;\Omega')$ be the weak solution to problem \eqref{eq:bdary}. Then for any $\rho\in(0,1)$ and $x_0\in\Omega$ such that $Q_\rho\equiv Q_{\rho}(x_0)\Subset\Omega'$, there holds
\begin{align}\label{eq:cacc.b}
&\mean{\Omega_{\frac{\rho}{2}}}|Du|^p\,dx+|\Omega_{\frac{\rho}{2}}|^{-1}[u]^q_{W^{s,q}(Q_{\frac{\rho}{2}})} \nonumber\\
&\quad\lesssim \rho^{-p}\mean{\Omega_{\rho}}|u-g|^p\,dx+\mean{\Omega_{\rho}}|Dg|^p\,dx+\mean{\Omega_{\rho}}|F|^p\,dx+|\Omega_{\rho}|^{-1}[g]^q_{W^{s,q}(Q_{\rho})} \nonumber\\
&\quad\quad+\rho^{-sq}\left(\mean{Q_{\rho}}|g-(g)_{\rho}|^q\,dx+\mean{\Omega_{\rho}}|u-g|^q\,dx\right)\nonumber\\
&\quad\quad+\mean{\Omega_{\rho}}|u-g|\,dx\left(\int_{\ern\setminus Q_{\rho}}\dfrac{|u(y)-(u)_{\rho}|^{q-1}}{|y-x_0|^{n+sq}}\,dy \right)
\end{align}
with the implicit constant $c=c(\data)$.
\end{lemma}
\begin{proof}
We  take $\eta\in C^{1}_0(Q_{\rho})$ such that $\eta\equiv 1$ in $Q_{\rho/2}$, $\supp(\eta)\subset Q_{2\rho/3}$ and $0\leq\eta\leq 1$ with $|D\eta|\leq\frac{c}{\rho}$ in $Q_{2\rho/3}$. 
For $w=u-g$ and $\ell:=\max\{p,q\}$, taking $\vp=\eta^\ell w\in \mathbb{X}_0(\Omega;\Omega')$ as a test function in the weak formulation of \eqref{eq:bdary}, we obtain
	\begin{align}\label{eq:cacc.b.1}
		I_0&:=\int_{\Omega_{\rho}}A(x,Du)\cdot D(u-g)\eta^\ell\,dx\nonumber\\
		&\leq\ell\int_{\Omega_{\rho}}|A(x,Du)|\eta^{\ell-1}|u-g||D\eta|\,dx+\int_{\Omega_{\rho}}|F|^{p-1}|D\vp|\,dx\nonumber\\
		&\quad-\int_{Q_{\rho}}\int_{Q_{\rho}}|u(x)-u(y)|^{q-2}(u(x)-u(y))(\vp(x)-\vp(y))K(x,y)\,dx\,dy\nonumber\\			&\quad+2\Lambda\int_{Q_{\rho}}\int_{Q^c_{\rho}}\dfrac{|u(x)-u(y)|^{q-1}|\vp(x)|}{|x-y|^{n+sq}}\,dx\,dy\nonumber\\
			&=:I_1+I_2-I_3+I_4.
		\end{align}
First, by \eqref{eq:A} and Young's inequality, $I_0$ is estimated as
\begin{align*}
I_0\geq \int_{\Omega_{\rho}}|Du|^p\eta^{\ell}\,dx-\int_{\Omega_{\rho}}|Du|^{p-1}|Dg|\eta^{\ell}\,dx\geq\frac{1}{2}\int_{\Omega_{\rho}}|Du|^p\eta^{\ell}\,dx-c\int_{\Omega_{\rho}}|Dg|^p\eta^{\ell}\,dx.
\end{align*}
For $I_1$ and $I_2$, using the structure condition \eqref{eq:A}, the bound on $|D\eta|$ and Young's inequality, for any $\ep\in(0,1)$ we have
	\begin{align*}
		I_1\leq \ep\int_{\Omega_{\rho}}|Du|^p\eta^\ell\,dx+c(\ep)\int_{\Omega_{\rho}}|D\eta|^p|u-g|^p\,dx\leq\ep\int_{\Omega_{\rho}}|Du|^p\eta^\ell\,dx+\dfrac{c(\ep)}{\rho^p}\int_{\Omega_{\rho}}|u-g|^p\,dx
	\end{align*}
	and
	\begin{align*}
		I_2&\leq\int_{\Omega_{\rho}}|F|^{p-1}|D(u-g)|\eta^{\ell}\,dx+c\int_{\Omega_{\rho}}|F|^{p-1}\eta^{\ell-1}|D\eta||u-g|\,dx\\
  &\leq \ep\int_{\Omega_{\rho}}|D(u-g)|^p\eta^\ell\,dx+c(\ep)\int_{\Omega_{\rho}}|F|^p\eta^\ell\,dx+\dfrac{c}{\rho^p}\int_{\Omega_{\rho}}|u-g|^p\,dx.
	\end{align*}

	We next write
	\begin{align*}
		-I_3&=-\underbrace{\int_{Q_{\rho}}\int_{Q_{\rho}}|u(x)-u(y)|^{q-2}(u(x)-u(y))(\eta^\ell(x)\bar{u}(x)+\eta^\ell(y)\bar{u}(y))K(x,y)\,dx\,dy}_{=:I_{31}} \\
		&\quad+ \underbrace{\int_{Q_{\rho}}\int_{Q_{\rho}}|u(x)-u(y)|^{q-2}(u(x)-u(y))(\eta^\ell(x)\bar{g}(x)-\eta^\ell(y)\bar{g}(y))K(x,y)\,dx\,dy}_{=:I_{32}},
	\end{align*} 
	where we have set $\bar{u}:=u-(u)_\rho$ and $\bar{g}:=g-(u)_\rho$. To estimate the term $I_{31}$, we observe that
	\begin{align*}
		|u(x)-u(y)|^{q-2}&(u(x)-u(y))(\eta^\ell(x)\bar{u}(x)-\eta^\ell(y)\bar{u}(y)) \\
		&\geq |\bar{u}(x)-\bar{u}(y)|^{q}\frac{\eta^\ell(x)+\eta^\ell(y)}{2}- |\bar{u}(x)-\bar{u}(y)|^{q-1}\frac{|\bar{u}(x)+\bar{u}(y)|}{2}|\eta^\ell(x)-\eta^\ell(y)|.
	\end{align*}
	Now, using the estimate (note that $\ell\geq q$)
	\begin{align*}
		|\eta^\ell(x)-\eta^\ell(y)|\leq q \big(\eta^\ell(x)+\eta^\ell(y) \big)^{(q-1)/q}|\eta^\frac{\ell}{q}(x)-\eta^\frac{\ell}{q}(y)|
	\end{align*}
	together with Young's inequality, we conclude that
	\begin{align*}
		|\bar{u}(x)-\bar{u}(y)|^{q-1}\frac{|\bar{u}(x)+\bar{u}(y)|}{2}|\eta^\ell(x)-\eta^\ell(y)| &\leq \frac{1}{8} |\bar{u}(x)-\bar{u}(y)|^{q}\frac{\eta^\ell(x)+\eta^\ell(y)}{2} \nonumber\\
		&\quad+c (|\bar{u}(x)|+|\bar{u}(y)|)^q |\eta^\frac{\ell}{q}(x)-\eta^\frac{\ell}{q}(y)|^q.
	\end{align*}
	Thus,
	\begin{align}\label{eq:cacc.b.7}
		I_{31}&\geq \frac{7}{8}\int_{Q_{\rho}}\int_{Q_{\rho}} |\bar{u}(x)-\bar{u}(y)|^{q}\frac{\eta^\ell(x)+\eta^\ell(y)}{2}K(x,y)\,dx\,dy \nonumber\\
		&\quad-c \int_{Q_{\rho}}\int_{Q_{\rho}} (|\bar{u}(x)|+|\bar{u}(y)|)^q |\eta^\frac{\ell}{q}(x)-\eta^\frac{\ell}{q}(y)|^q K(x,y)\,dx\,dy.
	\end{align}
	On a similar account, we have
	\begin{align}\label{eq:cacc.b.9}
		|I_{32}|&\leq \frac{1}{8}\int_{Q_{\rho}}\int_{Q_{\rho}} |\bar{u}(x)-\bar{u}(y)|^{q}\frac{\eta^\ell(x)+\eta^\ell(y)}{2}K(x,y)\,dxdy +c \int_{Q_{\rho}}\int_{Q_{\rho}} |\bar{g}(x)-\bar{g}(y)|^{q}K(x,y)\,dxdy \nonumber\\
		&\quad+c\int_{Q_{\rho}}\int_{Q_{\rho}} (|\bar{g}(x)|+|\bar{g}(y)|)^q |\eta^\frac{\ell}{q}(x)-\eta^\frac{\ell}{q}(y)|^q K(x,y)\,dx\,dy.
	\end{align}
 Coupling \eqref{eq:cacc.b.7} and \eqref{eq:cacc.b.9}, and noting $\eta\equiv 1$ on $Q_{\rho/2}$, with $c_{*}=c_{*}(\data)\geq 1$ and $c=c(\data)\geq 1$ we obtain
	\begin{align*}
		\begin{split}
			-I_3\leq-\dfrac{1}{c_*}[u]^q_{W^{s,q}(Q_{\rho/2})}  +c[g]^q_{W^{s,q}(Q_{\rho})}
			+c\rho^{-sq}\left[\int_{\Omega_{\rho}}|u-g|^q\,dx+\int_{Q_{\rho}}|u-(u)_{\rho}|^q\,dx\right],
		\end{split}
	\end{align*}
 where we have additionally used the facts $|\eta^\frac{m}{q}(x)-\eta^\frac{m}{q}(y)|^q\leq c\rho^{-q}|x-y|^q$ in $Q_\rho\times Q_\rho$ and $u-g=0$ in $Q_\rho\cap\Omega^c$.

 Next, for $x\in Q_{2\rho/3}$ and $y\in Q^{c}_{\rho}$, we have
\begin{align*}
\begin{split}
|x-x_0|\leq 2\sqrt{2}\rho/3\quad\mbox{ and }\quad|y-x_0|\geq\rho,
\end{split}
\end{align*}
  which implies that
	\begin{align}\label{eq:1st.l.6}
		|x-y|\geq 
		|y-x_0|\left(1-\frac{|x-x_0|}{|y-x_0|}\right)
		\geq |y-x_0|\left(1-\dfrac{2\sqrt{2}\rho/3}{\rho}\right)=\dfrac{3-2\sqrt{2}}{3}|y-x_0|.
	\end{align}
 On account of the relation $B_{2\rho/3}\subset Q_{2\rho/3}$, we have 
	\begin{align*}
		\int_{Q^c_{2\rho/3}}\dfrac{dy}{|y-x_0|^{n+sq}}\leq	\int_{B^c_{2\rho/3}}\dfrac{dy}{|y-x_0|^{n+sq}}
		\lesssim \rho^{-sq}.
	\end{align*}
	Therefore, using the above two displays implies that
	\begin{align*}
		I_4&\lesssim \rho^{-sq}\int_{Q_{\rho}}|u(x)-(u)_{\rho}|^{q-1}|u(x)-g(x)|\,dx +\int_{\ern\setminus Q_{\rho}}\dfrac{|u(y)-(u)_{\rho}|^{q-1}}{|y-x_0|^{n+sq}}\,dy\int_{Q_{\rho}}|u-g|\,dx \nonumber\\
			&\lesssim \rho^{-sq} \left[\int_{Q_{\rho}}|u-g|^q\,dx+\int_{Q_{\rho}}|u-(u)_{\rho}|^q\,dx\right]+ \left(\int_{Q_{\rho}}|u-g|\,dx\right) \int_{\ern\setminus Q_{\rho}}\dfrac{|u(y)-(u)_{\rho}|^{q-1}}{|y-x_0|^{n+sq}}\,dy.
		\end{align*}

Combining the resulting estimates for $I_0$ through $I_4$ with \eqref{eq:cacc.b.1} and choosing $\ep\in(0,1)$ sufficiently small, we obtain
	\begin{align}\label{eq:cacc.b.8}
		\int_{\Omega_{\rho/2}}|Du|^p\,dx&+[u]^q_{W^{s,q}(Q_{\rho/2})} \nonumber\\
		&\leq c\rho^{-p} \int_{\Omega_{\rho}}|u-g|^p\,dx +c\int_{\Omega_{\rho}}|Dg|^p\,dx+c\int_{\Omega_{\rho}}|F|^p\,dx+c[g]^q_{W^{s,q}(Q_{\rho})} \nonumber\\
		&\quad +c\rho^{-sq}\left(\int_{Q_{\rho}}|u-(u)_{\rho}|^q\,dx+\int_{\Omega_{\rho}}|u-g|^q\,dx\right)\nonumber\\
		 &\quad+c\left(\int_{\Omega_{\rho}}|u-g|\,dx\right)\int_{\ern\setminus Q_{\rho}}\dfrac{|u(y)-(u)_{\rho}|^{q-1}}{|y-x_0|^{n+sq}}\,dy
		\end{align}
with $c=c(\data)\geq 1$. Finally, applying the following estimate
	\begin{align}\label{eq:b.10}
		\begin{split}
		\mean{Q_{\rho}}|u-(u)_{\rho}|^l\,dx&\leq c(n,l) \mean{Q_{\rho}}|g-(g)_{\rho}|^l\,dx+ c(n,l)\mean{\Omega_{\rho}}|u-g|^l\,dx\quad(\text{for any }l,\rho>0)
		\end{split}
	\end{align}
 in \eqref{eq:cacc.b.8}, we complete the proof of the lemma.
\end{proof}

Similar to the above lemma, we can obtain the following estimate which holds inside of the domain $\Omega$. The estimate is proved by testing $\varphi=\eta^{\ell}(u-(u)_{\frac{29\rho}{40}})\in\mathbb{X}_0(\Omega;\Omega')$ to \eqref{eq:bdary} and applying appropriately modified argument of the proof of Lemma \ref{lem:cacc.b}.

\begin{corollary}\label{cor:cacc.b.int}
Let $u\in\mathbb{X}_g(\Omega;\Omega')$ be the weak solution to problem \eqref{eq:bdary}. Then for any $\rho\in(0,1)$ and $x_0\in\Omega$ such that $Q_{3\rho/4}(x_0)\subset\Omega$, there holds	\begin{align}\label{eq:cacc.b.int}
\mean{Q_{\frac{\rho}{2}}}|Du|^p\,dx
 &\lesssim \rho^{-p}\mean{Q_{\frac{29\rho}{40}}}|u-(u)_\frac{29\rho}{40}|^p\,dx+\rho^{-sq}\mean{Q_{\frac{29\rho}{40}}}|u-(u)_\frac{29\rho}{40}|^q\,dx+\mean{Q_{\frac{29\rho}{40}}}|F|^p\,dx\nonumber\\
&\quad+\mean{Q_{\frac{29\rho}{40}}}|u-(u)_\frac{29\rho}{40}|\,dx\Bigg(\int_{\ern\setminus Q_{\frac{29\rho}{40}}}\dfrac{|u(y)-(u)_{\frac{29\rho}{40}}|^{q-1}}{|y-x_0|^{n+sq}}\,dy \Bigg)
\end{align}
with the implicit constant $c=c(\data)$.
\end{corollary}

Now we can obtain the following reverse H\"{o}lder inequality. 
\begin{lemma}\label{lem:rev.hold}
 Let $\Omega$ be an open set satisfying \eqref{eq:m.den1} with $R_0>0$ and let $u\in \mathbb{X}_{g}(\Omega;\Omega')\cap \mathbb{X}(\Omega)$ be a weak solution to problem \eqref{eq:bdary} with $g\in \mathbb{X}_{\loc}(\Omega')$. Assume that $\rho\in(0,\min\{1/8,R_0\})$, $x_0\in\Omega$ and $L\in\mathbb{N}$ satisfies $Q_{2^{L+1}\rho}(x_0)\Subset\Omega'$ with $2^L\rho\leq 1$. Then for each $\nu\in (0,1)$, there exists a constant $c=c(\data,\mathcal{M}(u;\Omega),\mathcal{M}(g;Q_{2^L\rho}(x_0)),\nu)$ such that
	\begin{align}\label{eq:rev.hol.b}
		&\mean{\Omega_{\frac{\rho}{2}}}|Du|^p\,dx+|\Omega_{\frac{\rho}{2}}|^{-1}[u]^q_{W^{s,q}(Q_{\frac{\rho}{2}})}\nonumber\\
			&\quad\leq c\left(\mean{\Omega_{\rho}}|Du|^{p\sigma}\,dx\right)^{\frac{1}{\sigma}}+c\mean{\Omega_{\rho}}|Dg|^p\,dx+c\mean{\Omega_{\rho}}(|F|^p+1)\,dx \nonumber\\
			&\quad\quad+\nu\left[ \sum^{L+1}_{k=1}(2^k)^{\theta-\frac{q\tau}{q-1}} \left(\left(\mean{\Omega_{2^{k}\rho}}|Du|^{p\sigma}\,dx\right)^{\frac{\theta}{p\sigma}}+\left(\mean{Q_{2^{k}\rho}}|Dg|^{p}\,dx\right)^{\frac{\theta}{p}}\right) \right]^\frac{p}{\theta}\nonumber\\
			&\quad\quad+\nu\left(\rho^{1-\varrho}\int_{\ern\setminus Q_{2^L\rho}}\dfrac{|u(y)-(u)_{2^L\rho}|^{q-1}}{|y-x_0|^{n+sq}}\,dy\right)^\frac{p}{m-1}
	\end{align}
	holds for all $\sigma\in[\frac{n}{n+p}, 1]$ and $\varrho\in(0,m-qs)$ with $m:=\min\{p,q\}$, where $\tau$ is as in \eqref{eq:tau}.
 \end{lemma}
\begin{proof}
To prove the lemma, we will employ the Caccioppoli type inequality and the Sobolev-Poincar\'{e} inequality for two different cases: $Q_{\frac{3}{4}\rho}(x_0)\not\subset\Omega$ or $Q_{\frac{3}{4}\rho}(x_0)\subset\Omega$.\\
\textbf{Case 1}: If $Q_{\frac{3}{4}\rho}(x_0)\not\subset\Omega$ holds, then we estimate the right-hand side quantity of \eqref{eq:cacc.b} as below, in order to use Lemma \ref{lem:cacc.b}.  
 Using Corollary \ref{cor:emb1} and Lemma \ref{lem:embed2}, we have
\begin{align}\label{eq:rev.hol.b.1}
\rho^{-sq}\mean{Q_{\rho}}|g-(g)_{\rho}|^q\,dx+
\mean{Q_\rho}\int_{Q_\rho}\frac{|g(x)-g(y)|^q}{|x-y|^{n+sq}}\,dx\,dy\lesssim \rho^{\theta q-sq} \|g\|_{L^{\infty}(Q_{\rho})}^{q(1-\theta)}\left(\mean{Q_{\rho}}|Dg|^p\,dx\right)^{\frac{\theta q}{p}}.
\end{align}
On a similar account, using Lemma \ref{lem:embed3}, for all $\frac{n}{n+p}\leq\sigma\leq 1$ we get 
\begin{align}\label{eq:rev.hol.b.2}
\rho^{-p}\mean{\Omega_{\rho}}|u-g|^p\,dx\lesssim\left(\mean{\Omega_{\rho}}|Du-Dg|^{p\sigma}\,dx\right)^{\frac{1}{\sigma}}\lesssim\left(\mean{\Omega_{\rho}}|Du|^{p\sigma}\,dx\right)^{\frac{1}{\sigma}}+\mean{\Omega_{\rho}}|Dg|^p\,dx
\end{align}
and
\begin{align}\label{eq:rev.hol.b.3}
\rho^{-sq}\mean{\Omega_{\rho}}|u-g|^q\,dx&\lesssim \rho^{q\theta-sq} \|u-g\|^{q(1-\theta)}_{L^{\infty}(\Omega_{\rho})}\left(\mean{\Omega_{\rho}}|Du-Dg|^{p\sigma}\,dx\right)^{\frac{\theta q}{p\sigma}}\nonumber\\
&\lesssim \rho^{q\theta-sq} \left[\left(\mean{\Omega_{\rho}}|Du|^{p\sigma}\,dx\right)^{\frac{\theta q}{p\sigma}}+\left(\mean{\Omega_{\rho}}|Dg|^{p}\,dx\right)^{\frac{\theta q}{p}}\right].
\end{align}

	Next, set
	\begin{align}\label{eq:rev.hol.b.10}
		I:=\left(\mean{\Omega_{\rho}}|u-g|\,dx\right)\cdot \tilde{I}^{q-1}\quad\mbox{with}\quad
		\tilde{I}:=\left(\int_{\ern\setminus Q_{\rho}} \dfrac{|u(y)-(u)_{\rho}|^{q-1}}{|y-x_0|^{n+sq}}\,dy\right)^\frac{1}{q-1}.
	\end{align}
	Using Minkowski type inequality, we deduce that 
	\begin{align}\label{eq:rev.hol.b.11}
		\tilde{I}&\lesssim \sum^{L-1}_{k=0} \left(\int_{Q_{2^{k+1}\rho}\setminus Q_{2^{k}\rho}}\dfrac{|u(y)-(u)_{\rho}|^{q-1}}{|y-x_0|^{n+sq}}\,dy\right)^\frac{1}{q-1}+\left(\int_{\ern\setminus Q_{2^L\rho}}\dfrac{|u(y)-(u)_{\rho}|^{q-1}}{|y-x_0|^{n+sq}}\,dy\right)^\frac{1}{q-1} \nonumber\\
			&
			\lesssim\sum^{L-1}_{k=0}(2^k\rho)^\frac{-sq}{q-1}\left(\mean{Q_{2^{k+1}\rho}}|u(y)-(u)_{\rho}|^{q-1}\,dy\right)^\frac{1}{q-1}+\left(\int_{\ern\setminus Q_{2^L\rho}}\dfrac{|u(y)-(u)_{2^L\rho}|^{q-1}}{|y-x_0|^{n+sq}}\,dy\right)^\frac{1}{q-1}\nonumber\\
			&\quad+(2^L\rho)^\frac{-sq}{q-1}\sum^{L-1}_{k=0}|(u)_{2^{k+1}\rho}-(u)_{2^k\rho}|\nonumber\\
			&\lesssim\sum^{L-1}_{k=0}(2^k\rho)^\frac{-sq}{q-1}\left[\left(\mean{Q_{2^{k+1}\rho}}|u(y)-(u)_{\rho}|^{q-1}\,dy\right)^\frac{1}{q-1}+\left(\mean{Q_{2^{k+1}\rho}}|u(y)-(u)_{2^{k+1}\rho}|^{q-1}\,dy\right)^\frac{1}{q-1}\right]\nonumber\\
			&\quad+\left(\int_{\ern\setminus Q_{2^L\rho}}\dfrac{|u(y)-(u)_{2^L\rho}|^{q-1}}{|y-x_0|^{n+sq}}\,dy\right)^\frac{1}{q-1}.
		\end{align}
 Furthermore, we observe that
	\begin{align}\label{eq:rev.hol.b.12}
		&\sum^{L-1}_{k=0}(2^k\rho)^\frac{-sq}{q-1}\left(\mean{Q_{2^{k+1}\rho}}|u(y)-(u)_{\rho}|^{q-1}\,dy\right)^\frac{1}{q-1} \nonumber\\
		&\quad\lesssim\sum^{L-1}_{k=0}(2^k\rho)^\frac{-sq}{q-1} \left[\sum^k_{j=1}|(u)_{2^{j+1}\rho}-(u)_{2^j\rho}|+\left(\mean{Q_{2^{k+1}\rho}}|u(y)-(u)_{2^{k+1}\rho}|^{q-1}\,dy\right)^\frac{1}{q-1}\right] \nonumber\\
		 &\quad\lesssim \sum^{L-1}_{k=0}(2^k\rho)^\frac{-sq}{q-1} \sum^k_{j=1}\left(\mean{Q_{2^{j+1}\rho}}|u(y)-(u)_{2^{j+1}\rho}|^{q-1}\,dy\right)^\frac{1}{q-1} \nonumber\\
			&\quad\lesssim \sum^{L-1}_{j=0}(2^j\rho)^\frac{-sq}{q-1} \left(\mean{Q_{2^{j+1}\rho}}|u(y)-(u)_{2^{j+1}\rho}|^{q-1}\,dy\right)^\frac{1}{q-1}.
	\end{align}
Thus, combining \eqref{eq:rev.hol.b.11} and \eqref{eq:rev.hol.b.12} together with the observations $\tau>s$ (thanks to \eqref{eq:tau}) and $2^L\rho\leq 1$, we obtain
	\begin{align*}
		\tilde{I}&\lesssim\sum^{L-1}_{k=0}(2^k\rho)^\frac{-q\tau}{q-1} \left(\mean{Q_{2^{k+1}\rho}}|u(y)-(u)_{2^{k+1}\rho}|^{q-1}\,dy\right)^\frac{1}{q-1}+ \left(\int_{\ern\setminus Q_{2^L\rho}}\dfrac{|u(y)-(u)_{2^L\rho}|^{q-1}}{|y-x_0|^{n+sq}}\,dy\right)^\frac{1}{q-1} \nonumber\\
		&\lesssim \sum^{L-1}_{k=0}(2^k\rho)^\frac{-q\tau}{q-1} \left[\left(\mean{\Omega_{2^{k+1}\rho}}|u(y)-g(y)|^q\,dy\right)^{\frac{1}{q}}+ \left(\mean{Q_{2^{k+1}\rho}}|g(y)-(g)_{2^{k+1}\rho}|^q\,dy\right)^{\frac{1}{q}}\right] \nonumber\\
		&\quad\quad+\left(\int_{\ern\setminus Q_{2^L\rho}}\dfrac{|u(y)-(u)_{2^L\rho}|^{q-1}}{|y-x_0|^{n+sq}}\,dy\right)^\frac{1}{q-1},
	\end{align*}
	where on the last line we have used the relation \eqref{eq:b.10}. Proceeding similar to \eqref{eq:rev.hol.b.1} and \eqref{eq:rev.hol.b.3}, we conclude that 
	\begin{align}\label{eq:rev.hol.b.13}
		\tilde{I}&\lesssim \sum^{L-1}_{k=0}(2^k\rho)^{\theta-\frac{q\tau}{q-1}} \left[\left(\mean{\Omega_{2^{k+1}\rho}}|Du|^{p\sigma}\,dx\right)^{\frac{\theta}{p\sigma}}+\left(\mean{Q_{2^{k+1}\rho}}|Dg|^{p}\,dx\right)^{\frac{\theta}{p}}\right] \nonumber\\
		&\quad\quad+\left(\int_{\ern\setminus Q_{2^L\rho}}\dfrac{|u(y)-(u)_{2^L\rho}|^{q-1}}{|y-x_0|^{n+sq}}\,dy\right)^\frac{1}{q-1}.
	\end{align} 

Recalling the definition of $I$ from \eqref{eq:rev.hol.b.10} and using  \eqref{eq:rev.hol.b.13} with \eqref{eq:rev.hol.b.2} and \eqref{eq:rev.hol.b.3}, we have
	\begin{align}\label{eq:rev.hol.b.19}
		I&\lesssim \rho^{-q\tau+q\theta}\left[\left(\mean{\Omega_{\rho}}|Du|^{p\sigma}\,dx\right)^{\frac{\theta}{p\sigma}}+\left(\mean{\Omega_{\rho}}|Dg|^{p}\,dx\right)^{\frac{\theta}{p}}\right]\nonumber\\
		&\qquad\times
		\left( \sum^{L-1}_{k=0}
         (2^k)^{\theta-\frac{q\tau}{q-1}} 
\left[\left(\mean{\Omega_{2^{k+1}\rho}}|Du|^{p\sigma}\,dx\right)^{\frac{\theta}{p\sigma}}+\left(\mean{Q_{2^{k+1}\rho}}|Dg|^{p}\,dx\right)^{\frac{\theta}{p}}\right]
		\right)^{q-1} \nonumber\\ &\quad+\rho\left[\left(\mean{\Omega_{\rho}}|Du|^{p\sigma}\,dx\right)^{\frac{1}{p\sigma}}+\left(\mean{\Omega_{\rho}}|Dg|^p\,dx\right)^{\frac{1}{p}}\right]\left(\int_{\ern\setminus Q_{2^L\rho}}\dfrac{|u(y)-(u)_{2^L\rho}|^{q-1}}{|y-x_0|^{n+sq}}\,dy\right).
	\end{align}
 Here, by virtue of Young's inequality, for $\nu\in(0,1)$ arbitrary  and $\varrho<m-qs$ with $m=\min\{p,q\}$, we deduce that
 \begin{align}\label{eq:rev.hol.b.18}
     &\rho\left[\left(\mean{\Omega_{\rho}}|Du|^{p\sigma}\,dx\right)^{\frac{1}{p\sigma}}+\left(\mean{\Omega_{\rho}}|Dg|^p\,dx\right)^{\frac{1}{p}}\right]\left(\int_{\ern\setminus Q_{2^L\rho}}\dfrac{|u(y)-(u)_{2^L\rho}|^{q-1}}{|y-x_0|^{n+sq}}\,dy\right) \nonumber\\
     &\leq c(\nu) \rho^{\varrho m}\left[ \Big(\mean{\Omega_{\rho}}|Du|^{p\sigma}\,dx\Big)^\frac{1}{\sigma}+\mean{\Omega_{\rho}}|Dg|^{p}\,dx\right]^\frac{m}{p} + \nu \left(\rho^{1-\varrho}\int_{\ern\setminus Q_{2^L\rho}}\dfrac{|u(y)-(u)_{2^L\rho}|^{q-1}}{|y-x_0|^{n+sq}}\,dy\right)^{m'} \nonumber\\
     &\leq c(\nu)\left[1+\Big(\mean{\Omega_{\rho}}|Du|^{p\sigma}\,dx\Big)^\frac{1}{\sigma}+\mean{\Omega_{\rho}}|Dg|^{p}\,dx\right]+\nu \left(\rho^{1-\varrho}\int_{\ern\setminus Q_{2^L\rho}}\dfrac{|u(y)-(u)_{2^L\rho}|^{q-1}}{|y-x_0|^{n+sq}}\,dy\right)^\frac{p}{m-1}.
 \end{align}
 Now combining \eqref{eq:rev.hol.b.19} and \eqref{eq:rev.hol.b.18}, and using Young's inequality once again (for the first term on the right side of \eqref{eq:rev.hol.b.19}), we obtain
 \begin{align}\label{eq:rev.hol.b.17}
     I&\leq \nu \left[ \sum^{L}_{k=1}(2^k)^{\theta-\frac{q\tau}{q-1}} \left(\left(\mean{\Omega_{2^{k}\rho}}|Du|^{p\sigma}\,dx\right)^{\frac{\theta}{p\sigma}}+\left(\mean{Q_{2^{k}\rho}}|Dg|^{p}\,dx\right)^{\frac{\theta}{p}}\right) \right]^\frac{p}{\theta} \nonumber\\
    &\quad+\nu \left(\rho^{1-\varrho}\int_{\ern\setminus Q_{2^L\rho}}\dfrac{|u(y)-(u)_{2^L\rho}|^{q-1}}{|y-x_0|^{n+sq}}\,dy\right)^\frac{p}{m-1}\nonumber\\
    &\quad+c(\nu)\left[1+\left(\mean{\Omega_{\rho}}|Du|^{p\sigma}\,dx\right)^\frac{1}{\sigma}+\mean{\Omega_{\rho}}|Dg|^{p}\,dx\right].
 \end{align}
 Collecting \eqref{eq:rev.hol.b.1}, \eqref{eq:rev.hol.b.2}, \eqref{eq:rev.hol.b.3} and \eqref{eq:rev.hol.b.17} in \eqref{eq:cacc.b} and using Young's inequality,  we get the required result of \eqref{eq:rev.hol.b} in this case.

 \noindent\textbf{Case 2}: If $Q_{\frac{3}{4}\rho}(x_0)\subset\Omega$ holds, then we estimate the right-hand side display of \eqref{eq:cacc.b.int} as below, in order to use Corollary \ref{cor:cacc.b.int}. Since $Q_\frac{29\rho}{40}\Subset\Omega_\rho\subset\Omega$, the first two terms on the right-hand side of \eqref{eq:cacc.b.int} can be estimated directly by applying the Sobolev-Poincar\'{e} type inequality of Lemma \ref{lem:SP} and Corollary \ref{cor:emb1}. 
For the nonlocal tail term, noting $Q_\frac{29\rho}{40}\subset Q_\rho$, an easy computation yields
\begin{align*}
\tilde{I}_{in}:=\Bigg(\int_{\ern\setminus Q_{\frac{29\rho}{40}}}\dfrac{|u(y)-(u)_{\frac{29\rho}{40}}|^{q-1}}{|y-x_0|^{n+sq}}\,dy \Bigg)^\frac{1}{q-1}\lesssim \rho^\frac{-qs}{q-1}\left(\mean{Q_{\rho}}|u-(u)_\rho|^{q-1}\,dx\right)^\frac{1}{q-1}+\tilde{I},
\end{align*}
where $\tilde{I}$ is as in \eqref{eq:rev.hol.b.10}. Then, following the argument of Case 1, we obtain
\begin{align}\label{eq:rev.hol.int.9}
\tilde{I}_{in}\leq \sum^{L}_{k=0}(2^k\rho)^\frac{-q\tau}{q-1} \left(\mean{Q_{2^{k}\rho}}|u(x)-(u)_{2^{k}\rho}|^{q-1}\,dx\right)^\frac{1}{q-1}+ \left(\int_{\ern\setminus Q_{2^L\rho}}\dfrac{|u(y)-(u)_{2^L\rho}|^{q-1}}{|y-x_0|^{n+sq}}\,dy\right)^\frac{1}{q-1}.
\end{align}
If $Q_{2^k\rho}\subset\Omega$ for some $k\leq L$, then by Corollary \ref{cor:emb1} we have
\begin{align}\label{eq:rev.hold.int.8}
 \Bigg(\mean{Q_{2^k\rho}}|u-(u)_{2^k\rho}|^{q-1}\,dx\Bigg)^\frac{1}{q-1}\lesssim \rho^{\theta} \Big(\mean{Q_{2^k\rho}}|Du|^{p\sigma}\,dx\Big)^\frac{\theta}{p\sigma}.
\end{align}
If $Q_{2^k\rho}\not\subset\Omega$, then $|Q_{2^{k+1}\rho}\cap\Omega^c|>0$ and $Q_{2^{k+1}\rho}\Subset\Omega'$ (thanks to the assumption $Q_{2^{L+1}}\Subset\Omega'$). Thus, in this case, by the use of \eqref{eq:b.10} together with \eqref{eq:rev.hol.b.3} and \eqref{eq:rev.hol.b.1}, we obtain 
\begin{align}\label{eq:rev.hol.int.7}
\Bigg(\mean{Q_{2^k\rho}}|u-(u)_{2^k\rho}|^{q-1}\,dx\Bigg)^\frac{1}{q-1}&\lesssim \Bigg(\mean{Q_{2^{k+1}\rho}}|u-(u)_{2^{k+1}\rho}|^{q}\,dx\Bigg)^\frac{1}{q} \nonumber\\
&\lesssim (2^{k+1}\rho)^{\theta} \Bigg[\left(\mean{\Omega_{2^{k+1}\rho}}|Du|^{p\sigma}\,dx\right)^{\frac{\theta}{p\sigma}}+\left(\mean{Q_{2^{k+1}\rho}}|Dg|^{p}\,dx\right)^{\frac{\theta}{p}}\Bigg].
\end{align}
Therefore, on account of \eqref{eq:rev.hold.int.8} and \eqref{eq:rev.hol.int.7}, we obtain from \eqref{eq:rev.hol.int.9} that  
\begin{align*}
\tilde{I}_{in}&\leq \sum^{L+1}_{k=1}(2^k\rho)^{\theta-\frac{q\tau}{q-1}}\Bigg[\left(\mean{\Omega_{2^{k}\rho}}|Du|^{p\sigma}\,dx\right)^{\frac{\theta}{p\sigma}}+\left(\mean{Q_{2^{k}\rho}}|Dg|^{p}\,dx\right)^{\frac{\theta}{p}}\Bigg]\nonumber\\
&\quad+\left(\int_{\ern\setminus Q_{2^L\rho}}\dfrac{|u(y)-(u)_{2^L\rho}|^{q-1}}{|y-x_0|^{n+sq}}\,dy\right)^\frac{1}{q-1},
\end{align*}
which is a similar estimate to \eqref{eq:rev.hol.b.13}. Now by following the remaining proof of Case 1 with appropriate modification and recalling \eqref{eq:cacc.b.int}, we again have \eqref{eq:rev.hol.b}.
\end{proof}

Now we prove the main result of this section in the boundary case.

\begin{proposition}[Higher integrability]\label{prop:h.i.b}
Let  $\Omega$ be an open set satisfying \eqref{eq:m.den1} with $R_0>0$ and let $u\in\mathbb{X}_g(\Omega;\Omega')\cap \mathbb{X}(\Omega)$ be a weak solution to problem \eqref{eq:bdary} under the structure assumptions \eqref{eq:s1pq}, \eqref{eq:sq<p}, \eqref{eq:A} and \eqref{eq:K}. Assume that  $g\in \mathbb{X}_{\loc}(\Omega')$ and for some small $\varsigma_0\in (0,1)$, there holds
	\begin{align*}
		|F|\in L^{p(1+\varsigma_0)}(\Omega_{3\rho_0}(x_0)) \quad\mbox{and}\quad |Dg|\in L^{p(1+\varsigma_0)}(Q_{3\rho_0}(x_0)),
	\end{align*}
where $x_0\in\overline\Omega$ and $\rho_0\in (0,1)$ satisfying $Q_{3\rho_0}(x_0)\Subset\Omega'$. Then, there exists $\varsigma_1\in (0,1)$ depending only on $\data_1$ and $\varsigma_0$ such that for all $\varsigma\in (0,\varsigma_1)$, $|Du|\in L^{p+\varsigma}(\Omega_{\rho_0}(x_0))$. Moreover, there exists a constant $c=c(\data_1,\varsigma_0)$ such that 
	\begin{align*}
		&\left(\mean{\Omega_{\rho_0}(x_0)}|Du|^{p+\varsigma}\,dx\right)^\frac{1}{p+\varsigma} \nonumber\\
		&\leq c \left(\mean{\Omega_{3\rho_0}(x_0)}|Du|^{p}\,dx\right)^{\frac{1}{p}}+c\left(\mean{Q_{3\rho_0}(x_0)}|Dg|^{p(1+\varsigma)}\,dx\right)^\frac{1}{p(1+\varsigma)} \nonumber\\
     &\quad+c\left(\mean{\Omega_{3\rho_0}(x_0)}(|F|^{p}+1)^{1+\varsigma}\,dx\right)^\frac{1}{p(1+\varsigma)}
		+ c \left(\rho_0^{1-\varrho}\int_{\ern\setminus Q_{3\rho_0}(x_0)}\dfrac{|u(y)-(u)_{3\rho_0}|^{q-1}}{|y-x_0|^{n+sq}}\,dy\right)^\frac{1}{m-1}
	\end{align*}
 for all $\varrho\in(0,m-qs)$ with $m=\min\{p,q\}$, where we have set 
 \begin{align*}
 \data_1:=\data_1\big(\data,\mathcal{M}(u;\Omega),\mathcal{M}(g;Q_{3\rho_0}(x_0))\big).
 \end{align*}
\end{proposition}
\begin{proof}
Proceeding with the notations of the proposition, for $Q_R(x)\subset Q_{3\rho_0}(x_0)$, we define the following functionals:
\begin{align*}
 &\Upsilon(x,R):=\left(\mean{\Omega_R(x)}(|F|^{p}+1)^{1+\varsigma}\,dy\right)^\frac{1}{p(1+\varsigma)}+\left(\mean{Q_R(x)}|Dg|^{p(1+\varsigma)}\,dy\right)^\frac{1}{p(1+\varsigma)},\nonumber\\
  &\Psi_M(x,R):= \left(\mean{\Omega_R(x)}|Du|^{p}\,dy\right)^\frac{1}{p} +M\left(\mean{\Omega_R(x)}(|F|^{p}+1)\,dy\right)^\frac{1}{p}+M\left(\mean{Q_R(x)}|Dg|^{p}\,dy\right)^\frac{1}{p},
\end{align*}
where $\varsigma\in (0,1)$ and $M\geq 1$ are constants which will be determined later, and 
\begin{align}\label{eq:sec3.tail}
	{\rm Tail}(x,R):&=\left[\sum_{k=1}^{l+1}\beta_k \Bigg( \left(\mean{\Omega_{2^kR}(x)}|Du|^{p\sigma}\,dy\right)^\frac{\theta}{p\sigma}+\left(\mean{Q_{2^kR}(x)}|Dg|^{p}\,dy\right)^\frac{\theta}{p}\Bigg)\right]^\frac{1}{\theta} \nonumber\\ 
	&\quad + \left(R^{1-\varrho} \int_{\ern\setminus Q_{2^lR}(x)}\dfrac{|u(y)-(u)_{2^lR}|^{q-1}}{|y-x|^{n+sq}}\,dy\right)^\frac{1}{m-1},
\end{align}
where $\beta_k:=(2^k)^{\theta-\frac{q\tau}{q-1}}$, $\varrho<q(1-s)$ and $l=l(R)\in\mathbb{N}$ is such that $\frac{\rho_0}{2}\leq 2^lR<\rho_0$. Moreover, we set
\begin{align*}
	&\Xi(x,R):=\Upsilon(x,R)+\Psi_M(x,R)+{\rm Tail}(x,R) \quad\mbox{and} \nonumber\\
	&\Xi_0(x,R):=\Upsilon(x,R)+\Psi_1(x,R)+ \left(R^{1-\varrho} \int_{\ern\setminus Q_{R}(x)}\dfrac{|u(y)-(u)_{R}|^{q-1}}{|y-x|^{n+sq}}\,dy\right)^\frac{1}{m-1}.
\end{align*}
Fix $r_1$ and $r_2$ such that $\frac{\rho_0}{2}\leq r_1<r_2\leq\rho_0$, and define 
\begin{align}\label{eq:lamb.1}
	\lambda_0:= \sup \left\{\Xi(x,R) \ : \ x\in \Omega_{r_1}(x_0), \ \frac{r_2-r_1}{16}<R\leq \frac{\rho_0}{2} \right\}.
\end{align}
Then the proof is divided into five steps as below.\\
\textbf{Step 1}: \textit{An upper bound on $\lambda_0$}. For any $x\in \Omega_{r_1}(x_0)$ and $\frac{r_2-r_1}{16}<R\leq \frac{\rho_0}{2}$, we observe that 
\begin{align*}
	\Upsilon(x,R)\leq \left(\frac{3\rho_0}{R}\right)^\frac{n}{p(1+\varsigma)}\Upsilon(x_0,3\rho_0)\leq \left(16\frac{3\rho_0}{r_2-r_1}\right)^\frac{n}{p}\Upsilon(x_0,3\rho_0).
\end{align*}
 Additionally, for $k=0,\dots,l+1$, noting $Q_{2^{l+1}R}(x)\subset Q_{3\rho_0}(x_0)$, we have
\begin{align*}
	\left(\mean{\Omega_{2^kR}(x)}|Du|^{p\sigma}\,dy\right)^\frac{1}{p\sigma} \leq \left(16\frac{3\rho_0}{r_2-r_1}\right)^\frac{n+p}{p}\left(\mean{\Omega_{3\rho_0}(x_0)}|Du|^{p}\,dy\right)^\frac{1}{p}
\end{align*}
for all $\frac{n+p}{n}\leq \sigma\leq 1$ and 
\begin{align*}
	\left(\mean{Q_{2^kR}(x)}|Dg|^{p}\,dy\right)^\frac{1}{p} \leq \left(16\frac{3\rho_0}{r_2-r_1}\right)^\frac{n}{p}\left(\mean{Q_{3\rho_0}(x_0)}|Dg|^{p}\,dy\right)^\frac{1}{p}.
\end{align*}
Furthermore, for the nonlocal tail term, similar to the proof of Lemma \ref{lem:rev.hold}, we see that
\begin{align*}
	J_0:&=\left( \int_{\ern\setminus Q_{2^lR}(x)}\dfrac{|u(y)-(u)_{2^lR}|^{q-1}}{|y-x|^{n+sq}}\,dy\right)^\frac{1}{m-1} \nonumber\\
	&\lesssim \left(\frac{3\rho_0}{r_2-r_1}\right)^\frac{n+sq}{m-1} \left( \int_{\ern\setminus Q_{3\rho_0}(x_0)} \dfrac{|u(y)-(u)_{3\rho_0}|^{q-1}}{|y-x_0|^{n+sq}}\,dy\right)^\frac{1}{m-1} \\
	&\quad+ \left(\frac{3\rho_0}{r_2-r_1}\right)^\frac{2n+sq}{m-1} (3\rho_0)^{-\frac{sq}{m-1}} 
	\left(\mean{Q_{3\rho_0}(x_0)}|u(y)-(u)_{3\rho_0}|^{q-1}\,dy\right)^\frac{1}{m-1}.
\end{align*}
Noting that $m\leq p$ and $u-g\in L^\infty(Q_{3\rho_0}(x_0))$ if $q>m$, by H\"older's inequality and Poincar\'e's inequality, we have 
\begin{align*}
	\left(\mean{Q_{3\rho_0}(x_0)}|u(y)-(u)_{3\rho_0}|^{q-1}\,dy\right)^\frac{1}{m-1} &\lesssim \|u-g\|_{L^\infty(Q_{3\rho_0})}^\frac{q-m}{m-1} \left(\mean{Q_{3\rho_0}(x_0)}|u(y)-(u)_{3\rho_0}|^{m-1}\,dy\right)^\frac{1}{m-1} \nonumber\\
	&\lesssim \rho_0 \left( \Big(\mean{\Omega_{3\rho_0}(x_0)}|Du|^{p}\,dy\Big)^\frac{1}{p}+\Big(\mean{Q_{3\rho_0}(x_0)}|Dg|^{p}\,dy\Big)^\frac{1}{p}\right),
\end{align*}
where we have also used \eqref{eq:b.10}.
Combining the above estimates and noting that $\varrho<m-qs$, we obtain
\begin{align}\label{eq:uppr.bnd.lmda-1}
	\lambda_0\leq cM \left(\frac{3\rho_0}{r_2-r_1}\right)^{\frac{2n+sq}{m-1}+\frac{n+p}{p}} \, \Xi_0(x_0,3\rho_0).
\end{align}
\textbf{Step 2}:\textit{Vitali covering}. For $\lambda\geq \lambda_0$ and $t>0$, we define
\begin{align}\label{eq:e-lamb}
	E_{\lambda}^t:=\Big\{ x\in \Omega_t(x_0) \, : \, \sup_{0\leq R\leq\frac{r_2-r_1}{16}}\Psi_M(x,R)>\lambda \Big\}.
\end{align}
From \eqref{eq:lamb.1}, it is clear that for each $x\in \Omega_{r_1}(x_0)$ and $R\in \big[\frac{r_2-r_1}{16},\frac{\rho_0}{2}\big]$, there holds
\begin{align*}
	\Psi_M(x,R)\leq \lambda_0\leq\lambda.
\end{align*}
Moreover, for each $x\in E_{\lambda}^{r_1}$ with $\lambda\geq\lambda_0$, there exists $R_x\in (0,\frac{r_2-r_1}{16}\big]$ such that
\begin{align}\label{eq:vit.cov.1}
	\Psi_M(x,R_x)\geq \lambda \quad\mbox{and}\quad \Psi_M(x,R)\leq\lambda\quad\mbox{for all }R\in (R_x,\tfrac{r_2-r_1}{16}\big].
\end{align}
By Vitali's covering lemma, there exists a sequence of pairwise disjoint cylinders $\{Q_{2R_{x_j}}(x_j)\}$ such that 
\begin{align}\label{eq:vit.cov.4}
	E_{\lambda}^{r_1}\subset\bigcup_{j=1}^{\infty}\Omega_{10R_{x_j}}(x_j).
\end{align} 
For each $j\in\mathbb{N}$, setting $R_j:=R_{x_j}$ and $Q_j:= Q_{R_{x_j}}(x_j)$,  we choose $l_{j}\in\mathbb{N}$ satisfying 
\begin{align}\label{eq:vit.cov.2}
	\frac{\rho_0}{2}\leq 2^{l_j}R_j<\rho_0.
\end{align}
Since $R_j\leq \frac{1}{16}\frac{\rho_0}{2}$, it is evident that $l_j\geq 3$ for all $j\in\mathbb{N}$. Moreover, $Q_{2^{l_j+1}R_j}(x_j)\subset Q_{3\rho_0}(x_0)$.\\
\textbf{Step 3}: \textit{Estimate on $|\Omega_j|$}. On account of \eqref{eq:e-lamb}, either of the following holds:
\begin{align}
	&\left(\mean{\Omega_j}|Du|^p\,dy\right)^\frac{1}{p}\geq \frac{\lambda}{2} \quad\mbox{or} \label{eq:lamb-case1}\\
	&M \left(\mean{\Omega_{j}}(|F|^{p}+1)\,dy\right)^\frac{1}{p}+M \left(\mean{Q_{j}}|Dg|^{p}\,dy\right)^\frac{1}{p}\geq \frac{\lambda}{2} \label{eq:lamb-case2}.
\end{align}
\texttt{Case 1}: Suppose that \eqref{eq:lamb-case1} holds. Then, 
by the reverse H\"older inequality \eqref{eq:rev.hol.b}, we have
\begin{align*}
    \lambda^p &\leq  c\left(\mean{2\Omega_{j}}|Du|^{p\sigma}\,dy\right)^{\frac{1}{\sigma}}+c\mean{2\Omega_{j}}(|F|^p+1)\,dy+c\mean{2Q_{j}}|Dg|^p\,dy \nonumber\\
	&\quad+\nu \left[\sum^{l_j+1}_{k=1}\beta_k \left(\Big(\mean{2^{k}\Omega_j}|Du|^{p\sigma}\,dy\Big)^{\frac{\theta}{p\sigma}}+\Big(\mean{2^{k}Q_j}|Dg|^{p}\,dy\Big)^\frac{\theta}{p}\right) \right]^\frac{p}{\theta}\\
	&\quad+\nu  \left(R_{j}^{1-\varrho}\int_{\ern\setminus 2^{l_j}Q_j}\dfrac{|u(y)-(u)_{2^{l_j}R_j}|^{q-1}}{|y-x_0|^{n+sq}}\,dy\right)^\frac{p}{m-1},
\end{align*}
where $aQ_j=Q_{aR_{x_j}}(x_j)$ for $a>0$. As a consequence of \eqref{eq:vit.cov.1}, for all $k=1,\dots,l_j-1$, we obtain
\begin{align*}
	\left(\mean{2^{k}\Omega_j}|Du|^{p\sigma}\,dy\right)^{\frac{1}{p\sigma}} \leq \lambda,\quad
	\left(\mean{2^{k}Q_j}|Dg|^{p}\,dy\right)^{\frac{1}{p}} \leq \frac{\lambda}{M} 
\end{align*}
and 
\begin{align*}
    \left(\mean{2\Omega_{j}}(|F|^p+1)\,dy\right)^\frac{1}{p}\leq \frac{\lambda}{M}.
\end{align*}
Next, recalling the definition of $\mathrm{Tail}(x_j,R)$ from \eqref{eq:sec3.tail} for $R=2^{l_j-1}R_j\in \big[\frac{r_2-r_1}{16},\frac{\rho_0}{2}\big]$ (with $l_R=1$), on account of \eqref{eq:lamb.1}, we have 
\begin{align*}
 \beta_{l_j}^\frac{1}{\theta}&\left[  \left(\mean{2^{l_j}\Omega_{j}}|Du|^{p\sigma}\,dy\right)^\frac{\theta}{p\sigma}+\left(\mean{2^{l_j}Q_{R}}|Dg|^{p}\,dy\right)^\frac{\theta}{p}\right]\\
&\quad\quad\quad+\beta_{l_j}^\frac{1}{\theta}\left[\left(\mean{2^{l_j+1}\Omega_{j}}|Du|^{p\sigma}\,dy\right)^\frac{\theta}{p\sigma}+\left(\mean{2^{l_j+1}Q_{R}}|Dg|^{p}\,dy\right)^\frac{\theta}{p}\right]^\frac{1}{\theta} \nonumber\\
 &\quad\quad\quad+\left[(2^{l_j}R_j)^{1-\varrho}\int_{\ern\setminus 2^{l_j}Q_j}\dfrac{|u(y)-(u)_{2^{l_j}R_j}|^{q-1}}{|y-x_0|^{n+sq}}\,dy\right]^\frac{1}{m-1} \nonumber\\
 &\leq \sup_{x\in \Omega_{r_1}(x_0)}{\rm Tail}(x,2^{l_j-1}R_j) \lesssim \lambda.
\end{align*}
Consequently, for $c=c(\data_1)$, since $\sum^{\infty}_{k=0}\beta_k=c<\infty$ (thanks to \eqref{eq:tau}), we obtain
\begin{align*}
	\lambda^p \leq c(\nu) \left(\mean{2\Omega_j}|Du|^{p\sigma}\,dy\right)^{\frac{1}{\sigma}}+ c\frac{\lambda^p}{M^p}+c\nu\lambda^p.
\end{align*}
Now, for suitable choice of $\nu$ (small enough) and $M$ (large enough), both depending only on $\data_1$, we have 
\begin{align*}
	|\Omega_j|\leq \frac{c}{\lambda^{p\sigma}}\int_{2\Omega_j}|Du|^{p\sigma}\,dy,
\end{align*}
which after simple manipulation implies that
\begin{align}\label{eq:vit.cov.8}
	|\Omega_j|\leq \frac{c}{\lambda^{p\sigma}}\int_{2\Omega_j\cap\{|Du|>\kappa\lambda\}}|Du|^{p\sigma}\,dy,
\end{align}
for some $\kappa\in (0,1)$ small enough which is independent of $\lambda$ and $j$.\\
\texttt{Case 2}: Now suppose that \eqref{eq:lamb-case2} holds, then either
\begin{align}\label{eq:lamb-case3}
	M \left(\mean{Q_{j}}|Dg|^{p}\,dy\right)^\frac{1}{p}\geq \frac{\lambda}{4} \quad\mbox{or}\quad
	M \left(\mean{\Omega_{j}}(|F|^{p}+1)\,dy\right)^\frac{1}{p}\geq \frac{\lambda}{4}.
\end{align} 
We now assume that the first inequality of \eqref{eq:lamb-case3} holds true. As a consequence, with some elementary manipulations, we have
\begin{align}\label{eq:vit.cov.9}
	|Q_j|\leq c\left(\frac{M}{\lambda}\right)^p \int_{2Q_j\cap\{|Dg|>\hat\kappa\lambda\}}|Dg|^p\,dy,
\end{align}
for some $\hat\kappa>0$ small enough that depends only on $\data_1$. If the second inequality of \eqref{eq:lamb-case3} holds, then we can proceed similarly to obtain
\begin{align}\label{eq:vit.cov.10}
	|\Omega_j|\leq c\left(\frac{M}{\lambda}\right)^p \int_{2\Omega_j\cap\{(|F|^p+1)^{1/p}>\tilde\kappa\lambda\}}(|F|^p+1)\,dy
\end{align}
for some small $\tilde\kappa>0$ that depends only on $\data_1$. Combining \eqref{eq:vit.cov.8}, \eqref{eq:vit.cov.9} and \eqref{eq:vit.cov.10}, we get 
\begin{align}\label{eq:vit.cov.om-j}
	|\Omega_j|&\leq \frac{c}{\lambda^{p\sigma}}\int_{2\Omega_j\cap\{|Du|>\kappa\lambda\}}|Du|^{p\sigma}\,dy+\frac{c}{\lambda^p} \int_{2Q_j\cap\{|Dg|>\hat\kappa\lambda\}}|Dg|^p\,dy\nonumber\\
	&\quad+\frac{c}{\lambda^p} \int_{2\Omega_j\cap\{(|F|^p+1)^{1/p}>\tilde\kappa\lambda\}}(|F|^p+1)\,dy.
\end{align}
\textbf{Step 4}: \textit{The good $\lambda$ inequality}.
Since $\{2Q_j\}$ is a collection of pairwise disjoint sets and $\cup_{j=1}^{\infty}2Q_j\subset Q_{r_2}(x_0)$, \eqref{eq:vit.cov.om-j} yields
\begin{align}\label{eq:vit.cov.3}
	\sum_{j=1}^{\infty}|\Omega_j|&\leq \frac{c}{\lambda^{p\sigma}}\int_{\Omega_{r_2}(x_0)\cap\{|Du|>\kappa\lambda\}}|Du|^{p\sigma}\,dy+\frac{c}{\lambda^p} \int_{Q_{r_2}(x_0)\cap\{|Dg|>\hat\kappa\lambda\}}|Dg|^p\,dy \nonumber\\
	& \quad+\frac{c}{\lambda^p} \int_{\Omega_{r_2}(x_0)\cap\{(|F|^p+1)^{1/p}>\tilde\kappa\lambda\}}(|F|^p+1)\,dy.
\end{align}
After some elementary manipulations and using \eqref{eq:vit.cov.2} along with  \eqref{eq:vit.cov.3} and \eqref{eq:vit.cov.4} (up to relabeling the constants $\hat\kappa$ and $\tilde\kappa$), we deduce that
\begin{align}\label{eq:slf-imp.gdlmbda}
	\int_{\Omega_{r_1}(x_0)\cap\{|Du|>\lambda\}}|Du|^p \,dy &\leq \frac{c}{\lambda^{p\sigma-p}}\int_{\Omega_{r_2}(x_0)\cap\{|Du|>\lambda\}}|Du|^{p\sigma}\,dy+c \int_{Q_{r_2}(x_0)\cap\{|Dg|>\hat\kappa\lambda\}}|Dg|^p\,dy \nonumber\\
	& \quad+c \int_{\Omega_{r_2}(x_0)\cap\{(|F|^p+1)^{1/p}>\tilde\kappa\lambda\}}(|F|^p+1)\,dy.
\end{align} 
\textbf{Step 5}: \textit{Conclusion}.
Now the rest of the proof follows by standard arguments using \eqref{eq:slf-imp.gdlmbda}. Indeed, for $t\geq 0$, we define the truncated function
	\begin{align*}
		[|Du|]_t:=\min\{|Du|,t\}.
	\end{align*}
Then,   for $\varsigma\in (0,1)$, to be determined later, we  have (e.g., see \cite{KMS})
\begin{align*}
\int_{\Omega_{r_1}(x_0)}  [|Du|]_t^{\varsigma} \,|Du|^p \,dy\leq \lambda_0^\varsigma \int_{\Omega_{r_1}(x_0)} |Du|^p \,dy+ \varsigma \int_{\lambda_0}^{t}\lambda^{\varsigma-1}\int_{\Omega_{r_1}(x_0)\cap \{|Du|\geq \lambda\}}  |Du|^p \,dy \,d\lambda. 
\end{align*}
The above expression with the help of \eqref{eq:slf-imp.gdlmbda} and the definition of $\lambda_0$ imply that
	\begin{align*}
	\int_{\Omega_{r_1}(x_0)}& [|Du|]_t^{\varsigma} \,|Du|^p \,dy \nonumber\\
		&\leq  c\lambda_0^{\varsigma+p} |\Omega_{r_1}(x_0)|+ c\, \varsigma\int_{\lambda_0}^{t}\lambda^{\varsigma-1+p-p\sigma}\int_{\Omega_{r_2}(x_0)\cap\{|Du|>\lambda\}}|Du|^{p\sigma}\,dy\, d\lambda \nonumber\\
		& \quad+c \, \varsigma\int_{\lambda_0}^{t}\lambda^{\varsigma-1} \int_{Q_{r_2}(x_0)\cap\{|Dg|>\hat\kappa\lambda\}}|Dg|^{p}\,dy\, d\lambda \nonumber\\
		& \quad+ c \, \varsigma \int_{\lambda_0}^{t}\lambda^{\varsigma-1}\int_{\Omega_{r_2}(x_0)\cap\{(|F|^p+1)^{1/p}>\tilde\kappa\lambda\}}(|F|^p+1)\,dy \,d\lambda.
	\end{align*}
 Now, by using the upper bound on $\lambda_0$ given by \eqref{eq:uppr.bnd.lmda-1} and Fubini's theorem (see \cite{KMS} for details), we can deduce that
	\begin{align*}
		\int_{\Omega_{r_1}(x_0)}&  [|Du|]_t^{\varsigma}\,|Du|^p \,dy \nonumber\\
		&\leq  \frac{c \, \varsigma}{p-p\sigma+\varsigma} \int_{\Omega_{r_2}(x_0)}[|Du|]_t^{\varsigma}\,|Du|^{p}\,dy +c |Q_{r_2}(x_0)| \mean{Q_{r_2}(x_0)}|Dg|^{p+\varsigma}\,dy \nonumber\\
		& \quad+c|\Omega_{r_2}(x_0)| \mean{\Omega_{r_2}(x_0)}(|F|^p+1)^{1+\varsigma}\,dy+c \ \Xi_0(x_0,3\rho_0)^{p+\varsigma} |\Omega_{r_1}(x_0)|.
	\end{align*}
Choosing $\varsigma_1\in (0,\varsigma_0)$ small enough depending only on $\data_1$, $\varsigma_0$ and $\sigma$ (as the constant $c$ depends only on $\data_1$) such that  $\frac{c\,\varsigma_1}{p-p\sigma}\leq \frac{1}{16}$, from a standard iteration argument, for all $\varsigma<\varsigma_1$, we obtain 
\begin{align*}
 \left(	\int_{\Omega_{\rho_0/2}(x_0)}  [|Du|]_t^{\varsigma} \, |Du|^p \,dy\right)^\frac{1}{p+\varsigma} \leq c \ \Xi_0(x_0,3\rho_0),
\end{align*}
which upon taking the limit as $t\to\infty$ yields
\begin{align}\label{eq:slf-impv.9}
   \left(	\int_{\Omega_{\rho_0/2}(x_0)} |Du|^{p+\varsigma} \,dy\right)^\frac{1}{p+\varsigma} \leq c \ \Xi_0(x_0,3\rho_0).
\end{align}

To complete the proof of the proposition, we use a covering argument as follows. For $r=\frac{4\rho_0}{5}$, there exist finitely many points $x_1,\dots,x_k\in Q_{\frac{\rho_0}{2}}(x_0)$ such that
\begin{align}\label{eq:slf-impv.8}
    \overline{Q_{\frac{9}{10}\rho_0}(x_0)}\subset \bigcup_{i=1}^{k} Q_{\frac{r}{2}}(x_i) \quad\mbox{and}\quad Q_{3r}(x_i)\Subset Q_{3\rho_0}(x_0).
\end{align}
Therefore, using \eqref{eq:slf-impv.9} for each $(x_i,r)$, we get
\begin{align*}
    \left(	\int_{\Omega_{r/2}(x_i)} |Du|^{p+\varsigma} \,dy\right)^\frac{1}{p+\varsigma} \leq c \ \Xi_0(x_i,3r)\leq c(k) \ \Xi(x_0,3\rho_0), 
\end{align*}
where the last inequality follows from the assertion of Step 1. Noting \eqref{eq:slf-impv.8} and summing over $i=1$ to $k$, we deduce that
\begin{align*}
    \left(	\int_{\Omega_{9\rho_0/10}(x_0)} |Du|^{p+\varsigma} \,dy\right)^\frac{1}{p+\varsigma} \leq c(k) \ \Xi(x_0,3\rho_0).
\end{align*}
Now, the conclusion of the proposition follows by repeating the above procedure for a different choice of $r$, for instance,  $r=\rho_0/10$.
\end{proof}



\section{Interior Gradient Estimate}
Throughout the section, we assume that $u$ is a local weak solution to \eqref{eq:inter} with $|F|\in L^p(\Omega)$. For
fixed $\tilde{\Omega}\Subset\Omega$, we select $R\in(0,1/9)$ and choose $L=L(R)\in\mathbb{N}$ such that  $3^LR\leq 1$ and $Q_{3^LR}\Subset\tilde{\Omega}$.\par 
Consider the problem:
\begin{equation}\label{eq:1st.eq.l}
	\left\{ \begin{array}{rlll}
		-\divo(A(x,Dw))&=0\quad\text{in}\,\, Q_{2R},\\
		w&=u\quad\text{on}\,\,\partial Q_{2R}.
	\end{array}
	\right.
\end{equation}
Before giving the first comparison estimate, we mention the following energy estimate and maximum principle which are local analogues of \eqref{eq:1st.eq.b} and \eqref{eq:max.pr.b}, respectively. 
\begin{lemma}\label{lem:energy-int.wu}
 Let $w$ be a weak solution to problem \eqref{eq:1st.eq.l}, then there exists $c=c(n,p,\Lambda)$ such that 
	\begin{align*}
		\mean{Q_{2R}}|Dw|^p\,dx\leq c\mean{Q_{2R}}|Du|^p\,dx,
	\end{align*}	
	and for the case $p<q<\frac{p}{s}$, there holds
	\begin{align*}
		\|w\|_{L^{\infty}(Q_{2R})}\leq c \|u\|_{L^{\infty}(Q_{2R})}.
	\end{align*}		
\end{lemma}
Now we have the following comparison lemma.
\begin{lemma}\label{lem:1st.l}
	For every $\ep$, $\nu\in (0,1)$ and $\varrho<m-qs$, there holds
	\begin{align*}
		&\mean{Q_{2R}}|V(Du)-V(Dw)|^2\,dx \nonumber\\
		 &\leq(\ep+c(\nu)R^{\Bbbk})\mean{Q_{3R}}|Du|^p\,dx +c(\ep,\nu)\mean{Q_{3R}}(|F|^p+1)\,dx+\nu \left[\sum^{L}_{k=2}\beta_k \left(\mean{Q_{3^{k}R}}|Du|^p\,dx\right)^\frac{\theta}{p}\right]^\frac{p}{\theta} \nonumber\\
			&\quad+\nu \left(R^{1-\varrho}\int_{\ern\setminus Q_{3^LR}}\dfrac{|u(y)-(u)_{3^LR}|^{q-1}}{|y-x_0|^{n+sq}}\,dy\right)^\frac{p}{m-1},
		\end{align*}
	where $m:=\min\{p,q\}$, $\theta$ is as in \eqref{eq:theta1}, $\tau$ is given by  \eqref{eq:tau}, $\Bbbk:=\min\{q(\theta-\tau),\varrho p\}>0$, $\beta_k:= 3^{k({\theta-\frac{q\tau}{q-1}})}$ with $\sum_{k=1}^{\infty}\beta_k<\infty$, and $c(\ep,\nu)=c(\data,\mathcal{M}(u;\tilde{\Omega}),\ep,\nu)$ (analogously $c(\nu)$ is defined).
\end{lemma}

\begin{remark}
	Note that the term
	\begin{align*}
		\sum^L_{k=2}\beta_k\left(\mean{Q_{3^kR}}|Du|^p\,dx\right)^\frac{\theta}{p}
	\end{align*}
	is taken to be zero if $L=1$.	
\end{remark}

\begin{proof}
	We first observe that $u-w\in W^{1,p}_0(Q_{2R})$. Thus, it can be extended to $\tilde{\Omega}$ by setting $u-w= 0$ outside $Q_{2R}$. Moreover, for the case $q>p$, $u\in L^{\infty}(\tilde{\Omega})$ and hence $u-w\in L^{\infty}(\tilde{\Omega})$ by Lemma \ref{lem:energy-int.wu}. Consequently, by Lemma \ref{lem:embed2}, $u-w\in W^{1,p}(Q_{3R})\cap W^{s,q}(Q_{3R})$ is an admissible test function for \eqref{eq:inter} and \eqref{eq:1st.eq.l}. Therefore, with the help of \eqref{eq:AV}, we obtain
	\begin{align*}
		&\Lambda^{-1}\mean{Q_{2R}}|V(Du)-V(Dw)|^2\,dx \lesssim \mean{Q_{2R}} (A(x,Du)-A(x,Dw))\cdot(Du-Dw)\,dx\nonumber\\
		&\quad\lesssim |Q_{2R}|^{-1}\int_{\ern}\int_{\ern}|u(x)-u(y)|^{q-1}|u(x)-w(x)-u(y)+w(y)|K(x,y)\,dx\,dy \nonumber\\
		&\quad\quad+\mean{Q_{2R}}|F|^{p-1}|D(u-w)|\,dx=:I+J.
	\end{align*}

	By Young's inequality, for $\ep>0$ we get
	\begin{align}\label{eq:1st.l.2}
		J\leq\ep\mean{Q_{2R}}|Du-Dw|^p\,dx+c(\ep)\int_{Q_{2R}}|F|^p\,dx.
	\end{align}
	To estimate $I$, we split the integral into two parts as below:
	\begin{align*}
		I_1:=\left(\dfrac{3}{2}\right)^n\Lambda\mean{Q_{3R}}\int_{Q_{3R}}\dfrac{|u(x)-u(y)|^{q-1}}{|x-y|^{n+sq}}|u(x)-w(x)-u(y)+w(y)|\,dx\,dy
	\end{align*}
	and
	\begin{align*}
		I_2:=2\Lambda\mean{Q_{2R}}\int_{\ern\setminus Q_{3R}}\dfrac{|u(x)-u(y)|^{q-1}}{|x-y|^{n+sq}}|u(x)-w(x)|\,dx\,dy.
	\end{align*}
	For $I_1$, applying H\"{o}lder's inequality, Lemma \ref{lem:embed2}, Lemma \ref{lem:energy-int.wu} and Young's inequality, we obtain
	\begin{align}\label{eq:1st.l.5}
		I_1&\lesssim \left(\mean{Q_{3R}}\int_{Q_{3R}}\dfrac{|u(x)-u(y)|^q}{|x-y|^{n+sq}}\,dx\,dy\right)^{\frac{q-1}{q}}\left(\mean{Q_{3R}}\int_{Q_{3R}}\dfrac{|u(x)-w(x)-u(y)+w(y)|^q}{|x-y|^{n+sq}}\,dx\,dy\right)^{\frac{1}{q}} \nonumber\\
			&\lesssim R^{(q-1)(\theta-s)}\left(\mean{Q_{3R}}|Du|^p\,dx\right)^{\frac{\theta(q-1)}{p}}\left(\mean{Q_{3R}}|Du-Dw|^p\,dx\right)^{\frac{\theta}{p}}R^{(\theta-s)}\|u\|^{(1-\theta)q}_{L^{\infty}(Q_{3R})} \nonumber\\
			&\lesssim R^{q(\theta-s)} \left(\mean{Q_{3R}}|Du|^p\,dx\right)^{\frac{\theta q}{p}}
			\lesssim R^{q(\theta-s)}\mean{Q_{3R}}(|Du|^p+1)\,dx.
		\end{align}
	Next, similar to \eqref{eq:1st.l.6}, for $x\in Q_{2R}\equiv Q_{2R}(x_0)$ and $y\in Q^{c}_{3R}\equiv Q^c_{3R}(x_0)$, we have
	\begin{align*}
		|x-y|\geq\dfrac{3-2\sqrt{2}}{3}|y-x_0|.
	\end{align*}
	Consequently,
	\begin{align}\label{eq:1st.l.8}
	I_2&\lesssim\mean{Q_{3R}}\int_{\ern\setminus Q_{3R}}\dfrac{|u(x)-(u)_{3R}|^{q-1}+|u(y)-(u)_{3R}|^{q-1}}{|x_0-y|^{n+sq}}|u(x)-w(x)|\,dx\,dy\nonumber\\
			&\lesssim \mean{Q_{3R}}|u(x)-(u)_{3R}|^{q-1}|u(x)-w(x)|\,dx\int_{\ern\setminus Q_{3R}}\dfrac{1}{|y-x_0|^{n+sq}}\,dy \nonumber\\
			&\quad+\mean{Q_{3R}}|u(x)-w(x)|\,dx\int_{\ern\setminus Q_{3R}}\dfrac{|u(y)-(u)_{3R}|^{q-1}}{|y-x_0|^{n+sq}}\,dy\nonumber\\
			&=:I_3+I_4.
		\end{align}
	Using H\"{o}lder's inequality, Corollary \ref{cor:emb1} and Lemma \ref{lem:energy-int.wu}, we get
	\begin{align}\label{eq:1st.l.9}
		  I_3&\lesssim R^{-sq}\left(\mean{Q_{3R}}|u(x)-(u)_{3R}|^q\,dx\right)^{\frac{q-1}{q}}\left(\mean{Q_{2R}}|u(x)-w(x)|^q\,dx\right)^{\frac{1}{q}} \nonumber\\
			&\lesssim  R^{q(\theta-s)}  \left[\|u\|^{1-\theta}_{L^{\infty}(Q_{3R})}\left(\mean{Q_{3R}}|Du|^p\,dx\right)^{\frac{\theta}{p}}\right]^{q-1}  \|u-w\|^{1-\theta}_{L^{\infty}(Q_{3R})} \left(\mean{Q_{3R}}|Du-Dw|^p\,dx\right)^{\frac{\theta}{p}}\nonumber\\
			&\lesssim R^{q(\theta-s)} \left(\mean{Q_{3R}}|Du|^p\,dx\right)^{\frac{\theta q}{p}}
			\lesssim R^{q(\theta-s)}\mean{Q_{3R}}(|Du|^p+1)\,dx,
		\end{align}
	where for the last relation we have used Young's inequality.

	Next, set
	\begin{align*}
	\tilde{I}_4:=\left(\int_{\ern\setminus Q_{3R}}\dfrac{|u(y)-(u)_{3R}|^{q-1}}{|y-x_0|^{n+sq}}\,dy\right)^\frac{1}{q-1}.
	\end{align*}
	Then, proceeding as in \eqref{eq:rev.hol.b.11} and \eqref{eq:rev.hol.b.12}, and using Corollary \ref{cor:emb1} we obtain  
	\begin{align}\label{eq:1st.l.13}
		\tilde{I}_4
		&\lesssim \sum^{L-1}_{k=1}(3^kR)^\frac{-q\tau}{q-1}\left(\mean{Q_{3^{k+1}R}}|u(y)-(u)_{3^{k+1}R}|^q\,dy\right)^{\frac{1}{q}}+\left(\int_{\ern\setminus Q_{3^LR}}\dfrac{|u(y)-(u)_{3^LR}|^{q-1}}{|y-x_0|^{n+sq}}\,dy\right)^\frac{1}{q-1}\nonumber\\
		&\lesssim \sum^{L-1}_{k=1}(3^kR)^{\theta-\frac{q\tau}{q-1}}\|u\|_{L^{\infty}(\tilde{\Omega})}^{(1-\theta)}\left(\mean{Q_{3^{k+1}R}}|Du|^p\,dx\right)^{\frac{\theta}{p}}+\left(\int_{\ern\setminus Q_{3^LR}}\dfrac{|u(y)-(u)_{3^LR}|^{q-1}}{|y-x_0|^{n+sq}}\,dy\right)^\frac{1}{q-1},
	\end{align}
	where we have additionally used the facts that $Q_{3^LR}\Subset\tilde{\Omega}$, $3^LR\leq 1$ and $\tau>s$. 
Moreover, as in $I_3$, we have
	\begin{align}\label{eq:1st.l.16}
		\mean{Q_{3R}}|u(x)-w(x)|\,dx\leq c R^{\theta}\|u\|^{1-\theta}_{L^{\infty}(Q_{3R})}\left(
		\mean{Q_{3R}}|Du|^p\,dx\right)^{\frac{\theta}{p}}.
	\end{align}
	Coupling \eqref{eq:1st.l.13} and \eqref{eq:1st.l.16}, and applying Lemma \ref{lem:embed3} with Lemma \ref{lem:energy-int.wu}, we have
	\begin{align*}
		 I_4&\lesssim R^{-q\tau+q\theta}\left(\mean{Q_{3R}}|Du|^p\,dx\right)^{\frac{\theta}{p}}
		 \left(\sum^{L-1}_{k=1}(3^k)^{\theta-\frac{q\tau}{q-1}}\left(\mean{Q_{3^{k+1}R}}|Du|^p\,dx\right)^{\frac{\theta}{p}}\right)^{q-1}\nonumber\\ &\quad+R\left(\mean{Q_{3R}}|Du|^p\,dx\right)^{\frac{1}{p}}\left(\int_{\ern\setminus Q_{3^LR}}\dfrac{|u(y)-(u)_{3^LR}|^{q-1}}{|y-x_0|^{n+sq}}\,dy\right).
	\end{align*}
Then, similar to \eqref{eq:rev.hol.b.18}, using Young's inequality and recalling the definition of various constants defined in the Lemma, we obtain 
\begin{align}\label{eq:1st.l.17}
   	I_4 &\leq \nu \left[\sum^{L-1}_{k=1}(3^k)^{\theta-\frac{q\tau}{q-1}} \left(\mean{Q_{3^{k+1}R}}|Du|^p\,dx\right)^\frac{\theta}{p}\right]^\frac{p}{\theta} \nonumber\\
			&\quad+c(\nu) +c(\nu) R^{\Bbbk}\mean{Q_{3R}}|Du|^p\,dx+\nu \left(R^{1-\varrho}\int_{\ern\setminus Q_{3^LR}}\dfrac{|u(y)-(u)_{3^LR}|^{q-1}}{|y-x_0|^{n+sq}}\,dy\right)^\frac{p}{m-1}.
\end{align}
 Finally, by using \eqref{eq:1st.l.9} and \eqref{eq:1st.l.17} in \eqref{eq:1st.l.8}, and combining it with \eqref{eq:1st.l.5} and \eqref{eq:1st.l.2}, we conclude the proof of the lemma.
\end{proof}

\textbf{Proof of Theorem \ref{thm:l}}: The proof is similar to the boundary case Theorem \ref{thm:b}, which will be presented in detail in the next section. Indeed, we use the comparison estimate of Lemma \ref{lem:1st.l} instead of Lemma \ref{lem:1st.bdry}, Lemma \ref{lem:2nd} for the case $Q_R\Subset\Omega$ and the similar covering argument as in Subsection \ref{sub.cov.CZ} without the $g$-term.\qed


\section{Global Gradient Estimate}
Throughout this section we assume that $u\in \mathbb{X}_{g}(\Omega;\Omega')\cap\mathbb{X}(\Omega)$ is the weak solution of \eqref{eq:bdary} with $|F|\in L^{\gamma}(\Omega)$, $g\in \mathbb{X}(\Omega')$ and $|Dg|\in L^{\gamma}(\Omega')$. Also, $\Omega$ is $(\delta,r_*)$-Reifenberg flat where $\delta\in(0,\frac{1}{2^{n+1}})$ and $r_*>0$ are determined later. With the notation of \eqref{eq:data}, we set
 \begin{align*}
 \data_2:=\data_2\big(\data,\mathcal{M}(u;\Omega),\mathcal{M}(g;\Omega')\big).
 \end{align*}

\subsection{The Comparison Estimates}
Let $x_0\in\overline\Omega$ and $R\in (0,1/9)$ such that $Q_{2R}(x_0)\not\subset\Omega$ and $Q_{3R}(x_0)\Subset\Omega'$. For the convenience, we abbreviate $$\Omega_{R}=\Omega_{R}(x_0).$$  Let $h_b\in W^{1,p}(\Omega_{2R})$ be the weak solution of
\begin{equation}\label{eq:1st.eq.b}
\left\{\begin{array}{rlll}
-\divo (A(x,Dh_b))&=0\,\,&&\text{in}\quad\Omega_{2R},\\
h_b&=u-g\,\,&&\text{on}\quad\partial(\Omega_{2R}).
\end{array}
\right.
\end{equation}
Then, we have the following standard energy estimate
\begin{align}\label{eq:energy.est.h.b}
\mean{\Omega_{2R}}|Dh_b|^p\,dx\leq c\mean{\Omega_{2R}}|Du|^p\,dx+c\mean{\Omega_{2R}}|Dg|^p\,dx
\end{align}
with $c=c(n,p,\Lambda)$. Indeed, by testing $h_b-u+g\in W^{1,p}_0(\Omega_{2R})$ to \eqref{eq:1st.eq.b},
\begin{align*}
\mean{\Omega_{2R}}A(x,Dh_b)\cdot D(h_b-u+g)\,dx=0
\end{align*}
holds and so by \eqref{eq:A} and Young's inequality, we find
\begin{align*}
\mean{\Omega_{2R}}|Dh_b|^p\,dx&\leq c\mean{\Omega_{2R}}|Dh_b|^{p-1}|Du|\,dx+c\mean{\Omega_{2R}}|Dh_b|^{p-1}|Dg|\,dx\\
&\leq c\,\ep \mean{\Omega_{2R}}|Dh_b|^{p}\,dx+c(\ep)\mean{\Omega_{2R}}|Du|^p\,dx+c(\ep)\mean{\Omega_{2R}}|Dg|^p\,dx
\end{align*}
for any $\ep\in(0,1)$ with $c=c(n,p,\Lambda)$ and $c(\ep)=c(n,p,\Lambda,\ep)>0$. In the above display, by choosing $\ep=\ep(n,p,\Lambda)$ sufficiently small and absorbing the first term of the right-hand side to the left-hand side, we obtain \eqref{eq:energy.est.h.b}. \par 
We also have the following self-improving property for $Dh_b$ (see \cite{KK}):
\begin{itemize}
 \item Let $\varsigma_1\in (0,1)$ be given by Proposition \ref{prop:h.i.b}. Then, there exists  $\sigma_0=\sigma_0(n,p,\Lambda,\varsigma_1)>0$, small enough such that for all $\sigma<\sigma_0$, there holds
\begin{align}\label{eq:hi.w.b}
\begin{split}
\left(\mean{\Omega_{R}}|Dh_b|^{p(1+\sigma)}\,dx\right)^\frac{1}{p(1+\sigma)}&\leq c\left(\mean{\Omega_{2R}}|Dh_b|^{p}\,dx\right)^\frac{1}{p}+c\left(\mean{\Omega_{2R}}|Du|^{p(1+\sigma)}\,dx\right)^\frac{1}{p(1+\sigma)}\\
&\quad+c\left(\mean{\Omega_{2R}}|Dg|^{p(1+\sigma)}\,dx\right)^\frac{1}{p(1+\sigma)}
\end{split}
\end{align}
with $c=c(n,p,\Lambda)>0$.
\end{itemize}

We now prove the first comparison lemma.
\begin{lemma}\label{lem:1st.bdry}
Let $L\in\mathbb{N}$ be such that $3^LR\leq 1$ and $Q_{3^LR}\Subset\Omega'$.
Then for every $\ep\in(0,1)$ and $\nu\in (0,1)$, there holds
\begin{align}\label{eq:1st.b.cmp}
&\mean{\Omega_{2R}}|V(Du)-V(Dh_b)|^2\,dx \nonumber\\
&\quad\leq(\ep+c(\nu) R^{\Bbbk})\mean{\Omega_{3R}}|Du|^p\,dx+ c(\ep,\nu)\mean{\Omega_{3R}}(|F|^p+1)\,dx  \nonumber\\
&\quad\quad+c(\ep,\nu)\mean{Q_{3R}}|Dg|^p\,dx+\nu \left(R^{1-\varrho}\int_{\ern\setminus Q_{3^LR}}\dfrac{|u(y)-(u)_{3^LR}|^{q-1}}{|y-x_0|^{n+sq}}\,dy\right)^\frac{p}{m-1}  \nonumber\\
&\quad\quad+\nu \left[\sum^{L}_{k=2}\beta_k\left( \left(\mean{Q_{3^{k}R}}|Dg|^p\,dx\right)^\frac{\theta}{p}+\left(\mean{\Omega_{3^{k}R}}|Du|^p\,dx\right)^\frac{\theta}{p}\right)\right]^\frac{p}{\theta},
\end{align}
where $\Bbbk,\varrho,m,\beta_k$ and $\theta$ are as in Lemma \ref{lem:1st.l}, and $c(\ep,\nu)=c(\data_2,\ep,\nu)$ (analogously $c(\nu)$ is defined).
\end{lemma}
\begin{proof}
Set $w:=u-g-h_b\in W^{1,p}_0(\Omega_{2R})$. We extend $w$ by $0$ outside $\Omega_{2R}$.  Moreover, for the case $q>p$, since we have assumed that $u\in L^{\infty}(\Omega)$ and $g\in L^{\infty}(\Omega_{2R})$, by the standard maximum principle, 
\begin{align}\label{eq:max.pr.b}
\|h_b\|_{L^{\infty}(\Omega_{2R})}\lesssim \|u\|_{L^{\infty}(\Omega_{2R})}+\|g\|_{L^{\infty}(\Omega_{2R})}
\end{align}
holds with $c=c(n,p,\Lambda)$ and so $w\in L^{\infty}(\Omega_{2R})$. Thus, by Lemma \ref{lem:embed2}, it can easily be verified that $w\in\mathbb{X}_0(\Omega;\Omega')$.
Taking $w$ as a test function in the weak formulation of \eqref{eq:bdary} and \eqref{eq:1st.eq.b}, we obtain
\begin{align*}
&\int_{\Omega_{2R}}\left(A(x,Du)-A(x,Dh_b)\right)\cdot Dw\,dx\\
&\quad+\int_{\ern}\int_{\ern}|u(x)-u(y)|^{q-2}(u(x)-u(y))(w(x)-w(y))K(x,y)\,dx\,dy\\
&\quad\quad=\int_{\Omega_{2R}}|F|^{p-2}F\cdot Dw\,dx.
\end{align*}

Then \eqref{eq:A} and \eqref{eq:K} yield
\begin{align*}
\int_{\Omega_{2R}}|V(Du)-V(Dh_b)|^2\,dx
&\eqsim\int_{\Omega_{2R}}\left(A(x,Du)-A(x,Dh_b)\right)\cdot D(u-h_b)\,dx \nonumber\\
&=\int_{\Omega_{2R}}|F|^{p-2}F\cdot Dw\,dx+\int_{\Omega_{2R}}\big(A(x,Du)-A(x,Dh_b)\big)\cdot Dg\,dx \nonumber\\
&\ +\int_{\ern}\int_{\ern}|u(x)-u(y)|^{q-2}(u(x)-u(y))(w(x)-w(y))K(x,y)\,dx\,dy \nonumber\\
&\lesssim\underbrace{\int_{\Omega_{2R}}|F|^{p-1}|Dw|\,dx+\int_{\Omega_{2R}}(|A(x,Du)+|A(x,Dh_b)|)|Dg|\,dx}_{=:I} \nonumber\\
&\quad+\underbrace{\int_{\ern}\int_{\ern}\dfrac{|u(x)-u(y)|^{q-1}|w(x)-w(y)|}{|x-y|^{n+sq}}\,dx\,dy}_{=:J}.
\end{align*}
Using  Young's inequality and \eqref{eq:energy.est.h.b}, we have
\begin{align*}
I&\leq c(\ep)\int_{\Omega_{2R}}|F|^p\,dx+\ep\int_{\Omega_{2R}}|Dw|^p\,dx+c(\ep)\int_{\Omega_{2R}}|Dg|^p\,dx+\ep\int_{\Omega_{2R}}(|Du|^p+|Dh_b|^p)\,dx \nonumber\\
&\lesssim c(\ep)\int_{\Omega_{2R}}|F|^p\,dx+c(\ep)\int_{\Omega_{2R}}|Dg|^p\,dx+\ep\int_{\Omega_{2R}}|Du|^p\,dx
\end{align*}
for any $\ep\in(0,1)$ with $c(\ep)=c(n,p,\Lambda,\ep)>0$. For the nonlocal term, we write
\begin{align*}
J&=\underbrace{\int_{Q_{3R}}\int_{Q_{3R}}\dfrac{|u(x)-u(y)|^{q-1}|w(x)-w(y)|}{|x-y|^{n+sq}}\,dxdy}_{=:J_1}\\
&\quad+\underbrace{2\int_{Q_{2R}}\int_{\ern\setminus Q_{3R}}\dfrac{|u(x)-u(y)|^{q-1}}{|x-y|^{n+sq}}|w(x)|\,dydx}_{=:J_2}.
\end{align*}
Then, by H\"{o}lder's inequality, Lemma \ref{lem:embed2}, \eqref{eq:energy.est.h.b} and \eqref{eq:max.pr.b}, we estimate
\begin{align*}
  J_1&\lesssim\int_{Q_{3R}}\int_{Q_{3R}}\dfrac{|u(x)-g(x)-u(y)+g(y)|^{q-1}}{|x-y|^{n+sq}}|w(x)-w(y)|\,dx\,dy \nonumber\\
&\quad+\int_{Q_{3R}}\int_{Q_{3R}}\dfrac{|g(x)-g(y)|^{q-1}}{|x-y|^{n+sq}}|w(x)-w(y)|\,dx\,dy \nonumber\\
&\lesssim|Q_{3R}|\left(\mean{Q_{3R}}\int_{Q_{3R}}\dfrac{|u(x)-g(x)-u(y)+g(y)|^q}{|x-y|^{n+sq}}\,dxdy\right)^{\frac{1}{q'}}\left(\mean{Q_{3R}}\int_{Q_{3R}}\dfrac{|w(x)-w(y)|^q}{|x-y|^{n+sq}}\,dxdy\right)^{\frac{1}{q}} \nonumber\\
&\quad+|Q_{3R}|\left(\mean{Q_{3R}}\int_{Q_{3R}}\dfrac{|g(x)-g(y)|^q}{|x-y|^{n+sq}}\,dx\,dy\right)^{\frac{1}{q'}}\left(\mean{Q_{3R}}\int_{Q_{3R}}\dfrac{|w(x)-w(y)|^q}{|x-y|^{n+sq}}\,dx\,dy\right)^{\frac{1}{q}} \nonumber\\
&\lesssim|Q_{3R}|R^{(\theta-s)q}\|u-g\|^{(1-\theta)(q-1)}_{L^{\infty}(Q_{3R})}\left(\mean{Q_{3R}}|Du-Dg|^p\,dx\right)^{\frac{\theta(q-1)}{p}}\|w\|_{L^{\infty}(Q_{3R})}^{1-\theta}\left(\mean{Q_{3R}}|Dw|^p\,dx\right)^{\frac{\theta}{p}} \nonumber\\
&\quad+|Q_{3R}|R^{(\theta-s)q}\|g\|^{(1-\theta)(q-1)}_{L^{\infty}(Q_{3R})}\left(\mean{Q_{3R}}|Dg|^p\,dx\right)^{\frac{\theta(q-1)}{p}}\|w\|_{L^{\infty}(Q_{3R})}^{1-\theta}\left(\mean{Q_{3R}}|Dw|^p\,dx\right)^{\frac{\theta}{p}} \nonumber\\
&\lesssim |Q_{3R}|R^{(\theta-s)q}\left[\mean{\Omega_{3R}}|Du|^p\,dx+\mean{Q_{3R}}(|Dg|^p+1)\,dx\right]
\end{align*}
with the implicit constant $c=c(\data_2)$. 
As for $J_2$, noting 
	\begin{align*}
		 J_2&\lesssim \int_{Q_{2R}}|w(x)||u(x)-(u)_{3R}|^{q-1}\int_{\ern\setminus Q_{3R}}\dfrac{1}{|x-y|^{n+sq}}\,dy\,dx \nonumber\\
			&\quad+\int_{Q_{2R}}|w(x)|\int_{\ern\setminus Q_{3R}}\dfrac{|u(y)-(u)_{3R}|^{q-1}}{|x_0-y|^{n+sq}}\,dy\,dx,
	\end{align*}
we see that it resembles with \eqref{eq:1st.l.8}. Therefore, proceeding similarly to the proof of Lemma \ref{lem:1st.l} and additionally using \eqref{eq:b.10},  we conclude the proof. 
\end{proof}

Let $w_b\in W^{1,p}_0(\Omega_{R})$ be the weak solution of
\begin{equation}\label{eq:2nd.eq.b}
	\left\{\begin{array}{rlll}
		-\divo(\bar{A}_{B'_{R}}(x_1,Dw_b))&=0&&\text{in }\Omega_{R},\\
		w_b&=h_b&&\text{on }\partial(\Omega_{R}),
	\end{array}
	\right.
\end{equation}
where 
\begin{align*}
	\bar{A}_{B'_{R}}(x_1,z)=\mean{B'_{R}(x_0)}A(x_1,x',z)\,dx'.
\end{align*}
We then have the following energy estimate:
 there exists a constant $c=c(n,p,\Lambda)\geq 1$ such that
\begin{align}\label{eq:energy.est.w.b}
\mean{\Omega_{R}}|Dw_b|^p\,dx\leq c\mean{\Omega_{R}}|Dh_b|^p\,dx+c\mean{\Omega_{R}}|Dg|^p\,dx.
\end{align}

Now we proceed to the second comparison.
\begin{lemma}\label{lem:2nd}
Let $A$ be $(\delta,R_0)$-vanishing of codimension 1, for some $\delta$ and $R_0>0$. Then for all $R<R_0$, we have
\begin{align*}
\begin{split}
\mean{\Omega_{R}}&|V(Dw_b)-V(Dh_b)|^2\,dx\leq c \; \delta^{\frac{\sigma}{p'(1+\sigma)}} \left[\left(\mean{\Omega_{2R}}|Du|^{p(1+\sigma)}\,dx\right)^{\frac{1}{1+\sigma}}+\left(\mean{\Omega_{R}}|Dg|^{p(1+\sigma)}\,dx\right)^{\frac{1}{1+\sigma}}\right],
\end{split}
\end{align*}
where $\sigma\in (0,1)$ is as in \eqref{eq:hi.w.b} and $c=c(\data_2)$.
\end{lemma}
\begin{proof}
Test $h_b-w_b\in W^{1,p}_0(\Omega_{R})$ to both \eqref{eq:1st.eq.b} and \eqref{eq:2nd.eq.b} to have
\begin{align*}
\int_{\Omega_{R}}\left(A(x,Dh_b)-\bar{A}_{B'_R}(x_1,Dw_b)\right)\cdot(Dh_b-Dw_{b})\,dx=0.
\end{align*}
Then we obtain
\begin{align}\label{eq:2nd.b.6}
\mean{\Omega_{R}}&|V(Dw_b)-V(Dh_b)|^2\,dx \nonumber\\ &\lesssim\mean{\Omega_{R}}\left(\bar{A}_{B'_R}(x_1,Dh_b)-\bar{A}_{B'_R}(x_1,Dw_{b})\right)\cdot(Dh_b-Dw_{b})\,dx \nonumber\\
&=\mean{\Omega_{R}}\left(\bar{A}_{B'_R}(x_1,Dh_b)-A(x,Dh_{b})\right)\cdot(Dh_b-Dw_{b})\,dx \nonumber\\
&\lesssim \left(\mean{\Omega_{R}}\left|\bar{A}_{B'_R}(x_1,Dh_b)-A(x,Dh_{b})\right|^{p'}\,dx\right)^\frac{1}{p'}\left(\mean{\Omega_{R}}\left|Dh_b-Dw_{b}\right|^p\,dx\right)^\frac{1}{p},
\end{align}
where in the last line we have used H\"older's inequality.
For the first term of the right-hand side of \eqref{eq:2nd.b.6}, we compute
\begin{align}\label{eq:2nd.b.7}
\begin{split}
&\mean{\Omega_{R}}\left|\bar{A}_{B'_R}(x_1,Dh_b)-A(x,Dh_{b})\right|^{p'}\,dx\\
&\lesssim \mean{\Omega_{R}}\kappa(A,Q_{R})^{p'}\left|Dh_b\right|^{p}\,dx
\lesssim\underbrace{\left(\mean{\Omega_{R}}\kappa(A,Q_{R})^{\frac{p'(1+\sigma)}{\sigma}}\,dx\right)^{\frac{\sigma}{1+\sigma}}}_{=:I_{1}}\underbrace{\left(\mean{\Omega_{R}}\left|Dh_b\right|^{p(1+\sigma)}\,dx\right)^{\frac{1}{1+\sigma}}}_{=:I_{2}}.
\end{split}
\end{align}

For $I_{1}$, we use \eqref{eq:kappa.2L} and Definition \ref{def:BMO.b} to obtain
\begin{align}\label{eq:2nd.b.8}
I_{1}\lesssim \left(\mean{\Omega_{R}}\kappa(A,Q_{R})^{\frac{p'(1+\sigma)}{\sigma}}\,dx\right)^{\frac{\sigma}{1+\sigma}}\lesssim \Lambda^{\left(\frac{p'(1+\sigma)}{\sigma}-1\right)\frac{\sigma}{1+\sigma}}\delta^{\frac{\sigma}{1+\sigma}}\lesssim\delta^{\frac{\sigma}{1+\sigma}}.
\end{align}
Also, for $I_{2}$, \eqref{eq:energy.est.h.b}, \eqref{eq:energy.est.w.b} and \eqref{eq:hi.w.b} yield that
\begin{align}\label{eq:2nd.b.9}
I_{2}&\lesssim\mean{\Omega_{R}}\left|Dh_b\right|^{p}\,dx+\left(\mean{\Omega_{2R}}|Du|^{p(1+\sigma)}\,dx\right)^{\frac{1}{1+\sigma}}+\left(\mean{\Omega_{2R}}|Dg|^{p(1+\sigma)}\,dx\right)^{\frac{1}{1+\sigma}} \nonumber\\
&\lesssim\left(\mean{\Omega_{2R}}|Du|^{p(1+\sigma)}\,dx\right)^{\frac{1}{1+\sigma}}+\left(\mean{\Omega_{2R}}|Dg|^{p(1+\sigma)}\,dx\right)^{\frac{1}{1+\sigma}}.
\end{align}
For the second term of \eqref{eq:2nd.b.6}, using \eqref{eq:energy.est.w.b} and \eqref{eq:energy.est.h.b}, we have
\begin{align}\label{eq:2nd.b.10}
\mean{\Omega_{R}}\left|Dh_b-Dw_{b}\right|^p\,dx
&\lesssim\mean{\Omega_{R}}\left|Du\right|^p\,dx+\mean{\Omega_{R}}\left|Dg\right|^p\,dx \nonumber\\
&\lesssim\left(\mean{\Omega_{2R}}|Du|^{p(1+\sigma)}\,dx\right)^{\frac{1}{1+\sigma}}+\left(\mean{\Omega_{2R}}|Dg|^{p(1+\sigma)}\,dx\right)^{\frac{1}{1+\sigma}}.
\end{align}
Finally using \eqref{eq:2nd.b.10} and \eqref{eq:2nd.b.7} together with \eqref{eq:2nd.b.8} and \eqref{eq:2nd.b.9} in \eqref{eq:2nd.b.6}, we complete the proof of the lemma.
\end{proof}
Let $v_b\in W^{1,p}(Q^+_{\frac{R}{2}}(0))$ be a weak solution of
\begin{equation}\label{eq:3rd.eq.b}
	\left\{\begin{array}{rlll}
		-\divo(\bar{A}_{B'_{R}}(x_1,Dv_b))&=0\quad\text{in }Q^+_{\frac{R}{2}}(0)\\
		v_b&=0\quad\text{on }T_{\frac{R}{2}}(0).
	\end{array}
	\right.
\end{equation}

Then we have the following higher integrability estimate of $Dv_b$.
\begin{lemma}\label{lem:ext.5}
For any $\tilde{\gamma}_b\in(p,\infty)$, there exists a constant $c=c(n,p,\Lambda,\tilde{\gamma}_b)\geq 1$ such that
\begin{align}\label{eq:tilde v}
\mean{\Omega_{\frac{3R}{8}}(x_0)}|D\bar{v}|^{\tilde{\gamma}_b}\,dx\leq c\left(\mean{\Omega_{\frac{R}{2}}(x_0)}|D\bar{v}|^p\,dx\right)^{\frac{\tilde{\gamma}_b}{p}},
\end{align}
where $\bar{v}$ is the zero extension of $v_b$ from $Q^+_{\frac{R}{2}}(0)$ to $Q_{\frac{R}{2}}(0)$.
\end{lemma}
\begin{proof}
We abbreviate $\bar{A}_{B'_{R}}(x_1,\xi)=:\bar{A}(x_1,\xi)$. Define
\begin{align}\label{eq:ext.w}
	\hat{v}(x_1,x')=
	\begin{cases}
		v(x_1,x')&\text{for }x\in Q^+_{\frac{R}{2}}(0),\\
		-v(-x_1,x')&\text{for }x\in Q^-_{\frac{R}{2}}(0),
	\end{cases}
\end{align}
and
\begin{align*}
	\hat{A}_i(x_1,\xi_1,\xi')=
	\begin{cases}
		\bar{A}_i(x_1,\xi_1,\xi')&\text{for}\quad x_1\in(0,r)\quad\text{and}\quad 1\leq i\leq n,\\
		\bar{A}_1(-x_1,\xi_1,-\xi')&\text{for}\quad x_1\in(-r,0)\quad\text{and}\quad i=1,\\
		-\bar{A}_i(-x_1,\xi_1,-\xi')&\text{for}\quad x_1\in(-r,0)\quad\text{and}\quad i\geq 2.
	\end{cases}
\end{align*}
By the argument of the proofs of \cite[Lemma 2.2]{KR} and \cite[Theorem 3.4]{M1}, one can see that the vector field $\hat{A}(x_1,\xi)$ satisfies growth condition \eqref{eq:A} and $\hat{v}$ is a weak solution of 
\begin{align*}
-\divo(\hat{A}(x_1,D\hat{v}))=0\quad\text{in}\quad Q_{\frac{R}{2}}(0).
\end{align*}	
Now employing \cite[Theorem 2.1]{KimY1} (or \cite[Theorem 3.1]{BL1}) and a covering argument, recalling \eqref{eq:ext.w} and then considering that $\bar{v}$ is the zero extension of $v_b$ from $Q^+_{\frac{R}{2}}(0)$ to $Q_{\frac{R}{2}}(0)$, we obtain \eqref{eq:tilde v}.
\end{proof}

Now we prove the following comparison estimate.
\begin{lemma}\label{lem:3rd}
Let $\lambda\geq 1$ and $w_b\in W^{1,p}(\Omega_{R})$ be the weak solution of \eqref{eq:2nd.eq.b} such that
  \begin{align}\label{eq:3rd.1}
    \mean{\Omega_{R}}|Dw_b|^p\,dx\leq c_1 \lambda^p
\end{align}
for some  constant $c_1=c_1(n,p,\Lambda)$. 
Then for any $\ep\in(0,1)$, there exists a small number $\delta=\delta(n,p,\Lambda,\ep)$ such that if
	\begin{align}\label{eq:3rd.2}
		Q^+_{\frac{R}{2}}(0)\subset\Omega_{\frac{R}{2}}\subset Q_{\frac{R}{2}}(0)\cap\{x_1>-\delta R\}
	\end{align}
holds for such $\delta\in(0,1)$, then there exists a weak solution $v_b$ of 
\eqref{eq:3rd.eq.b} satisfying
  \begin{align}\label{eq:3rd.3}
			\mean{Q^+_{\frac{R}{2}}(0)}|Dv_b|^p\,dx\leq c_1 \lambda^p
		\end{align}
  and
\begin{align}\label{eq:3rd.4}
 \mean{Q^+_{\frac{R}{2}}(0)}|w_b-v_b|^p\,dx\leq\ep \lambda^p R^p.
	\end{align}
\end{lemma}
\begin{proof}
Before proving the lemma, we first apply scaling and normalization for problems \eqref{eq:2nd.eq.b} and \eqref{eq:3rd.eq.b} to use the compactness argument.

\textbf{Step 1}: \textit{Scaling and normalization}.
 For $r>0$, we define the scaled domain as
\begin{align*}
\tilde{\Omega}_r=\left\{\frac{1}{r}x:x\in\Omega\right\},
\end{align*}
a Carath\'{e}odory map $\tilde{A}:\ern\times\er\times\ern\rightarrow\ern$ and the normalized function $\tilde{h}_b\in W^{1,p}(\Omega_{2R/r})$ as
\begin{align*}
\tilde{A}(x,\xi)=\dfrac{A(rx,\lambda\xi)}{\lambda^{p-1}} \quad\mbox{and}\quad \tilde{h}_b(x)=\frac{h_b(rx)}{\lambda\;r}.
\end{align*}
Analogous to $\tilde{h}_b$, we also define  $\tilde{w}_b\in \tilde{h}_b+W^{1,p}_0(\Omega_{R/r})$ and $\tilde{v}_b\in W^{1,p}(Q^+_{R/(2r)}(0))$.
Then one can easily obtain the following:
\begin{itemize}
    \item $\tilde{A}$ satisfies the uniform ellipticity and growth conditions as in \eqref{eq:A}.
    \item  $\widetilde{w}_b$ is a solution to the following problem:
\begin{equation*}
	\left\{ \begin{array}{rlll}
		-\divo(\tilde{A}_{R/r}(x_1,D\widetilde{w}_b))&= 0\quad\text{in}\quad\tilde{\Omega}_{R/r},\\
		 \ \ \tilde{w}_b&=0\quad\text{on}\quad \partial_w{\tilde\Omega}_{R/r},
	\end{array}
	\right. 
\end{equation*}
 where $$\tilde{A}_{R/r}(x_1,\xi):=\mean{B'_{R/r}}\tilde{A}(x_1,x',z)\,dx'=\frac{\bar{A}_{B'_{R}}(y_1,\xi/\lambda)}{\lambda^{p-1}}.$$
 \item We have
 \begin{align*}
 Q^+_{\frac{R}{2r}}(0)\subset\tilde{\Omega}_{\frac{R}{2r}}\subset Q_{\frac{R}{2r}}(0)\cap\left\{x_1>-\delta \frac{R}{r}\right\}.   
 \end{align*}
 \item There holds
  \begin{align*}
  \mean{\Omega_{R/r}}|D\tilde{w}_b|^p \,dx\leq c_1.
  \end{align*}
\end{itemize}

\textbf{Step 2}: In this step we prove the lemma for $R=1$ and $\lambda=1$. 
On the contrary suppose that there exist $\ep_0>0$ and sequences $\{w_k\}^{\infty}_{k=1}$ and $\{\tilde{\Omega}_{\frac{1}{2},k}\}^{\infty}_{k=1}$ with $\tilde{\Omega}_{\frac{1}{2},k}:=\tilde{\Omega}_k\cap Q_{\frac{1}{2}}(0)$ such that $w_k\in W^{1,p}(\tilde{\Omega}_{\frac{1}{2},k})$ is a weak solution of the following
\begin{equation}\label{eq:3rd.5}
\left\{\begin{array}{rlll}
-\divo\bar{A}_{B'_{1}}(x_1,Dw_k)&=0\quad\text{in }\tilde{\Omega}_{\frac{1}{2},k},\\
w_k&=0\quad\text{on }\partial_w\tilde{\Omega}_{\frac{1}{2},k}
\end{array}
\right.
\end{equation}
with
\begin{align}\label{eq:3rd.6}
\mean{\tilde{\Omega}_{\frac{1}{2},k}}|Dw_k|^p\,dx\leq c_1 
\end{align}
and
\begin{align}\label{eq:3rd.7}
Q^+_{\frac{1}{2}}(0)\subset\tilde{\Omega}_{\frac{1}{2},k}\subset Q_{\frac{1}{2}}(0)\cap\left\{x_1>-1/k\right\},
\end{align}
but
\begin{align}\label{eq:3rd.8}
\mean{Q^+_{\frac{1}{2}}(0)}|w_k-\tilde{v}|^p\,dx>\ep_0
\end{align}
holds for any weak solution $\tilde{v}\in W^{1,p}(Q^+_{\frac{1}{2}}(0))$ of
\begin{equation*}
\left\{\begin{array}{rlll}
-\divo \bar{A}_{B'_1}(x_1,D\tilde{v})&=0\quad\text{in } Q^+_{\frac{1}{2}}(0),\\
\tilde{v}&=0\quad\text{on } T_{\frac{1}{2}}(0)
\end{array}
\right.
\end{equation*}
with
\begin{align}\label{eq:3rd.10}
\mean{Q^+_{\frac{1}{2}}(0)}|D\tilde{v}|^p\,dx\leq c_1.
\end{align}
	With the aid of \eqref{eq:3rd.6} and \eqref{eq:3rd.7}, the sequence $\{w_k\}^{\infty}_{k=1}$ is uniformly bounded in $W^{1,p}(Q^+_{\frac{1}{2}}(0))$. Hence, up to a subsequence, there is $w_0\in W^{1,p}(Q^+_{\frac{1}{2}}(0))$ such that
	\begin{align}\label{eq:3rd.11}
		\begin{cases}
			w_{k}\rightharpoonup w_0\quad\text{weakly in }W^{1,p}(Q^+_{\frac{1}{2}}(0)),\\
			w_k\rightarrow w_0\quad\text{in }L^p(Q^+_{\frac{1}{2}}(0)).
		\end{cases}
	\end{align}
	Then together with \eqref{eq:3rd.5} and \eqref{eq:3rd.7}, we notice that $w_0=0$ on $T_{\frac{1}{2}}(0)$ in the trace sense. Moreover, using the method of \textit{Browder and Minty} as in \cite[Lemma 4.5]{BW0}, we see that $w_0$ is a weak solution of
	\begin{equation*}
		\left\{\begin{array}{rlll}
			-\divo\bar{A}_{B'_1}(x_1,Dw_0)&=0\quad\text{in }Q^+_{\frac{1}{2}}(0),\\
			w_0&=0\quad\text{on }T_{\frac{1}{2}}(0).
		\end{array}
		\right.
	\end{equation*}
	On the other hand, from \eqref{eq:3rd.6} and \eqref{eq:3rd.11} we obtain that
	\begin{align}\label{eq:3rd.13}
		\mean{Q^+_{\frac{1}{2}}(0)}|Dw_0|^p\,dx\leq\liminf_{k\rightarrow\infty}\mean{\tilde{\Omega}_{\frac{1}{2},k}}|Dw_k|^p\,dx\leq c_1.
	\end{align}
	Then one can get a contradiction by taking $k$ sufficiently large and comparing \eqref{eq:3rd.8}--\eqref{eq:3rd.10} to \eqref{eq:3rd.11}--\eqref{eq:3rd.13}. Thus we have shown Lemma \ref{lem:3rd} for $R=1$ and $\lambda=1$.\\
 \textbf{Step 3}. We scale back the result of Step 2 to obtain the required result of the lemma for general $R>0$ and $\lambda\geq 1$. Indeed, setting $w_b(x)=\lambda R\tilde{w}(x/R)$, $v_b(x)=\lambda R\tilde{v}(x/R)$ and $\Omega_{R}=R\tilde{\Omega}_R$, we get \eqref{eq:3rd.3} and \eqref{eq:3rd.4}.  This completes the proof of the lemma.
\end{proof}

Finally, we show third comparison estimate as follows.

\begin{lemma}\label{lem:4th}
Under the same assumptions of Lemma \ref{lem:3rd}, we have
		\begin{align*}
			\mean{\Omega_{\frac{R}{4}}}|Dw_b-D\bar{v}|^p\,dx\leq\ep\lambda^p,
		\end{align*}
		where $\bar{v}$ is the zero extension of $v_b$ from $Q^+_{\frac{R}{2}}(0)$ to $Q_{\frac{R}{2}}(0)$.
\end{lemma}
\begin{proof}
By Lemma \ref{lem:3rd}, we see that for any $\ep_1>0$, there exist a small $\delta_1=\delta_1(n,\Lambda,\ep_1)$ and a weak solution $v_b\in W^{1,p}(Q^+_{\frac{R}{2}}(0))$ of
    \begin{equation*}
     \left\{\begin{array}{rlll}
		-\divo\bar{A}_{B'_{\frac{R}{2}}}(x_1,Dv_b)&=0\quad\text{in }Q^+_{\frac{R}{2}}(0),\\
		v_b&=0\quad\text{on }T_{\frac{R}{2}}(0)
		\end{array}
		\right.
     \end{equation*}
   with
     \begin{align*}
		\mean{Q^+_{\frac{R}{2}}(0)}|Dv_b|^p\,dx\leq c \lambda^p
		\end{align*}
 such that if \eqref{eq:3rd.1} and \eqref{eq:3rd.2} are true for $\delta_1$, then
		\begin{align}\label{eq:4th.4}
			\mean{Q^+_{\frac{R}{2}}(0)}|w_b-v_b|^p\,dx\leq\ep_1 \lambda^p R^p.
		\end{align}
 One can see that $\bar{v}$ is a weak solution of
    \begin{equation}\label{eq:4th.8}
	\left\{\begin{array}{rlll}
		 -\divo\bar{A}_{B'_{R}}(x_1,D\bar{v})&=-D_1(\bar{A}_{B'_{R},n}(0,Dv(0,x'))\chi_{\{x_1<0\}})\quad\text{in }\Omega_{\frac{3R}{8}},\\
		 \bar{v}&=0\quad\text{on }\partial\Omega_{\frac{3R}{8}}.
		\end{array}
		\right.
		\end{equation}
 Choose a cut-off function $\vp\in C^{\infty}_c(Q_{\frac{3R}{8}})$ such that
    \begin{align}\label{eq:4th.9}
		\vp=1\,\,\text{on}\,\, Q_{\frac{R}{4}},\,\,0\leq\vp\leq 1\,\,\text{and}\,\,|D\vp|\leq \frac{c}{R}\,\,\text{in}\,\,Q_{\frac{3R}{8}}.
      \end{align}
Testing $(\tilde{w}-\bar{v})\vp^p\in W^{1,p}_0(\Omega_{\frac{3R}{8}})$ for \eqref{eq:2nd.eq.b} and \eqref{eq:4th.8} implies that
	\begin{align}\label{eq:4th.10}
		\begin{split}
			I_1&:=\int_{\Omega_{\frac{3R}{8}}}\left(\bar{A}_{B'_{R}}(x_1,Dw_b)-\bar{A}_{B'_{R}}(x_1,D\bar{v})\right)\cdot(D[(w_b-\bar{v})\vp^p])\,dx\\
			&=\int_{\Omega_{\frac{3R}{8}}\setminus Q^+_{\frac{3R}{8}}(0)}\bar{A}_{B'_{R},n}(0,D\bar{v}(0,x'))D_1[(w_b-\bar{v})\vp^p]\,dx=:I_2.
		\end{split}
	\end{align}
	By Young's inequality and \eqref{eq:A}, we get
	\begin{align}\label{eq:4th.11}
		I_1\geq c^*\int_{\Omega_{\frac{3R}{8}}}|Dw_b-D\bar{v}|^p\vp^p\,dx-c\int_{\Omega_{\frac{3R}{8}}}|w_b-\bar{v}|^p|D\vp|^p\,dx
	\end{align}
with $c^*=c^*(n,p,\Lambda)\in(0,1)$ and $c=c(n,p,\Lambda)\geq 1$. Furthermore, using H\"{o}lder's inequality and Sobolev embedding, we discover that
	\begin{align}\label{eq:4th.12}
		\begin{split}
			\int_{\Omega_{\frac{3R}{8}}\setminus Q^+_{\frac{3R}{8}}(0)}|w_b-\bar{v}|^p|D\vp|^p\,dx&=\int_{\Omega_{\frac{3R}{8}}\setminus Q^+_{\frac{3R}{8}}(0)}|w_b|^p|D\vp|^p\,dx\\
			&\leq\frac{c}{R^p}\left(\int_{\Omega_{\frac{3R}{8}}\setminus Q^+_{\frac{3R}{8}}(0)}|w_b|^{p^*}\,dx\right)^{\frac{p}{p^*}} |\Omega_{\frac{3R}{8}}\setminus Q^+_{\frac{3R}{8}}(0)|^{\frac{p}{n}} \\
			&\leq \left(\int_{\Omega_{\frac{3R}{8}}}|Dw_b|^p\,dx\right)|\Omega_{\frac{3R}{8}}\setminus Q^+_{\frac{3R}{8}}(0)|^{\frac{p}{n}}.
		\end{split}
	\end{align}
	Merging  \eqref{eq:4th.4}, \eqref{eq:4th.11} and \eqref{eq:4th.12}, it follows that
	\begin{align}\label{eq:4th.14}
		I_1\geq c^*\int_{\Omega_{\frac{3R}{8}}}|Dw_b-D\bar{v}|^p\vp^p\,dx-\big(c\ep_1\lambda^p+c\delta^{\frac{p}{n}}\lambda^p\big) |\Omega_{\frac{3R}{8}}|,
	\end{align}
	where we have also used \eqref{eq:3rd.1}, \eqref{eq:3rd.2} and \eqref{eq:4th.9}.
	For $I_2$, we use \eqref{eq:A}, \eqref{eq:energy.est.w.b} and \eqref{eq:4th.9} together with H\"{o}lder's inequality and Poincar\'{e}'s inequality to find that
	\begin{align}\label{eq:4th.15}
		 I_2&\leq c\int_{\Omega_{\frac{3R}{8}}\setminus Q^+_{\frac{3R}{8}}(0)}|\bar{A}_{B'_{R}}(0,Dv_b(0,x'))|\;|D_1[\vp(w_b-\bar{v})]|\,dx \nonumber\\
			&\leq c\int_{\Omega_{\frac{3R}{8}}\setminus Q^+_{\frac{3R}{8}}(0)}|Dv_b(0,x')|^{p-1} (R^{-1}|w_b-\tilde{v}|+|D(w_b-\tilde{v})|)\,dx \nonumber\\
			&\leq c\left(\int_{\Omega_{\frac{3R}{8}}\setminus Q^+_{\frac{3R}{8}}(0)}|Dv_b(0,x')|^{2p}\,dx\right)^{\frac{p-1}{2p}}\nonumber\\
   &\quad\quad\quad\times\left(\int_{\Omega_{\frac{3R}{8}}\setminus Q^+_{\frac{3R}{8}}(0)}(R^{-1}|w_b|+|Dw_b|)^p\,dx\right)^{\frac{1}{p}}|\Omega_{\frac{3R}{8}}\setminus Q^+_{\frac{3R}{8}}(0)|^{\frac{1}{2p'}} \nonumber\\
			&\leq c\left(\mean{\Omega_{\frac{3R}{8}}}|Dw_b|^p\,dx\right)|\Omega_{\frac{3R}{8}}\setminus Q^+_{\frac{3R}{8}}(0)|\leq c\delta \lambda^p |\Omega_{\frac{3R}{8}}|,
	\end{align}
 where on the second last inequality we have also used Lemma \ref{lem:ext.5}.
 Combining \eqref{eq:3rd.2}, \eqref{eq:4th.9}, \eqref{eq:4th.10}, \eqref{eq:4th.14} and \eqref{eq:4th.15}, we discover
	\begin{align*}
		\mean{\Omega_{\frac{R}{4}}}|Dw_b-D\bar{v}|^p\,dx\leq\mean{\Omega_{\frac{3R}{8}}}|Dw_b-D\bar{v}|^p\vp^p\,dx\leq c\left(\ep_1+\delta^{\frac{p}{n}}+\delta\right)\lambda^p\leq\ep \lambda^p
	\end{align*}
	by choosing $\ep_1$ and then $\delta\in(0,\delta_1]$ small enough. This completes the proof.
\end{proof}

\subsection{Proof of Theorem \ref{thm:b}}\label{sub.cov.CZ}
Fix $x_0\in\Omega$ and $\rho_0\in (0,1)$ such that $Q_{2\rho_0}(x_0)\subset{\Omega'}$. With $\gamma>p$ and $Q_R(x)\subset Q_{2\rho_0}(x_0)$, we define the following functional:
\begin{align*}
	&\Psi_\delta(x,R)\nonumber\\
	&:= \left(\mean{\Omega_R(x)}|Du|^{p}\,dy\right)^\frac{1}{p} +\frac{1}{\delta}\left(\mean{\Omega_R(x)}(|F|^{p}+1)^{1+\sigma}\,dy\right)^\frac{1}{p(1+\sigma)}+\frac{1}{\delta}\left(\mean{Q_R(x)}|Dg|^{p(1+\sigma)}\,dy\right)^\frac{1}{p(1+\sigma)}
\end{align*}
for $0<\sigma<\min\{1,\frac{\gamma-p}{p}\}$ small enough given by Lemma \ref{lem:2nd}, and for some $\delta\in (0,1)$ which will be determined later. We also define
\begin{align*}
 {\rm Tail}(x,R):= &\left[\sum_{k=1}^{l}\beta_k \Bigg( \left(\mean{\Omega_{3^kR}(x)}|Du|^{p}\,dy\right)^\frac{\theta}{p}+\left(\mean{Q_{3^kR}(x)}|Dg|^{p}\,dy\right)^\frac{\theta}{p}\Bigg)\right]^\frac{1}{\theta} \nonumber\\ 
	& +\left( R^{1-\varrho}\int_{\ern\setminus Q_{3^lR}(x)}\dfrac{|u(y)-(u)_{3^lR}|^{q-1}}{|y-x|^{n+sq}}\,dy\right)^\frac{1}{m-1},
\end{align*}
where $l=l(R)\in\mathbb{N}$ is such that $\frac{\rho_0}{3}\leq 3^lR<\rho_0$, and $\varrho<m-qs$ is any number with $m:=\min\{p,q\}$. Moreover, we set
\begin{align*}
	&\Xi(x,R):=\Psi_\delta(x,R)+{\rm Tail}(x,R) \quad\mbox{and}\nonumber\\
	&\Xi_0(x,R):=
	\Psi_1(x,R)+\left(R^{1-\varrho} \int_{\ern\setminus Q_{R}(x)}\dfrac{|u(y)-(u)_{R}|^{q-1}}{|y-x|^{n+sq}}\,dy\right)^\frac{1}{m-1}.
\end{align*}
Fix $r_1$ and $r_2$ such that $\frac{\rho_0}{3}\leq r_1<r_2\leq\rho_0$. Then, we define 
\begin{align}\label{eq:CZ-B.lamb.1}
	\lambda_1:= \sup \bigg\{\Xi(x,R) \ : \ x\in \Omega_{r_1}(x_0), \ \frac{r_2-r_1}{729}<R\leq \frac{\rho_0}{3}\bigg\}.
\end{align}

Now the proof is divided into five steps as below.\\
\textbf{Step 1}: \textit{Upper bound on $\lambda_1$}. 
Similar to Step 1 of Proposition \ref{prop:h.i.b}, it is easy to deduce that
\begin{align}\label{eq:CZ-B.upper.lamb}
	\lambda_1\leq \frac{c}{\delta} \left(\frac{2\rho_0}{r_2-r_1}\right)^{\frac{2n+sq}{m-1}+\frac{n+p}{p}} \, \Xi_0(x_0,2\rho_0),
\end{align}
where $c=c(\data_2)$ is a constant.

\noindent\textbf{Step 2}:\textit{Vitali covering}. For $\lambda\geq \lambda_1$ and $T>0$, we set
\begin{align*}
	E_{\lambda}^T:=\Big\{ x\in \Omega_T(x_0) \, : \, \sup_{0\leq R\leq\frac{r_2-r_1}{729}}\Psi_\delta(x,R)>\lambda \Big\}.
\end{align*}
From \eqref{eq:CZ-B.lamb.1}, for each $x\in \Omega_{r_1}(x_0)$ and $R\in \big[\frac{r_2-r_1}{729},\frac{\rho_0}{3}\big]$, there holds
\begin{align*}
	\Psi_\delta(x,R)\leq \lambda_1\leq\lambda.
\end{align*}
Moreover, for each $x\in E_{\lambda}^{r_1}$ with $\lambda\geq\lambda_1$, there exists a radius $R_x\in (0,\frac{r_2-r_1}{729}\big]$ such that
\begin{align}\label{eq:CZ-B.vit.cov.1}
	\Psi_\delta(x,R_x)\geq \lambda \quad\mbox{and}\quad \Psi_\delta(x,R)\leq\lambda\quad\mbox{for all }R\in (R_x,\tfrac{\rho_0}{3}\big].
\end{align}
By Vitali's covering lemma, there exists a sequence of disjoint cylinders $\{Q_{R_{x_j}}(x_j)\}$ such that 
\begin{align}\label{eq:CZ-B.vit.cov.0}
	E_{\lambda}^{r_1}\subset\bigcup_{j=1}^{\infty}\Omega_{5R_{x_j}}(x_j).
\end{align} 
Setting $R_j:=R_{x_j}$ and $Q_j:= Q_{R_{x_j}}(x_j)$, for each $j\in\mathbb{N}$, we choose $l_{j}\in\mathbb{N}$ satisfying 
\begin{align}\label{eq:CZ-B.vit.cov.2}
	\frac{\rho_0}{3}\leq 3^{l_j}R_j<\rho_0.
\end{align}
Since $R_j\leq \frac{2}{729}\frac{\rho_0}{3}$, it is clear that $l_j\geq 5$ for all $j\in\mathbb{N}$.\\
\textbf{Step 3}: \textit{The comparison estimates}. We set $\varsigma\Omega_j:=\Omega_{\varsigma R_j}$ and $\varsigma Q_j:=Q_{\varsigma R_j}$, for $\varsigma>0$.
Let $h_j$ be the weak solution to problem \eqref{eq:1st.eq.b} with $R=27R_j$ and $x_0=x_j$, there. Then, from \eqref{eq:1st.b.cmp} we have
\begin{align*}
	\begin{split}
		&\mean{54\Omega_{j}}|V(Du)-V(Dh_j)|^2\,dx\\
		&\leq(\ep+c(\nu) R_j^{\Bbbk})\mean{81\Omega_{j}}|Du|^p\,dx+ c(\ep,\nu)\mean{81\Omega_{j}}(|F|^p+1)\,dx \\
		&\quad+c(\ep,\nu)\mean{81Q_{j}}|Dg|^p\,dx
		+\nu \left(R_j^{1-\varrho}\int_{\ern\setminus 3^{l_j+3} Q_{j}}\dfrac{|u(y)-(u)_{3^{l_j+3}R_j}|^{q-1}}{|y-x_0|^{n+sq}}\,dy\right)^\frac{p}{m-1} \\
		& \quad+\nu \left[\sum^{l_j}_{k=2}\beta_k \left(\mean{3^{k+3}Q_j}|Dg|^p\,dx\right)^\frac{\theta}{p}+\beta_k\left(\mean{3^{k+3}\Omega_j}|Du|^p\,dx\right)^\frac{\theta}{p}\right]^\frac{p}{\theta},
	\end{split}
\end{align*}
where the constants $\varrho,m,\theta,\tau,\Bbbk$  and $\beta_k$ are as in Lemma \ref{lem:1st.l} and $l_j\geq 5$ is given by \eqref{eq:CZ-B.vit.cov.2}. Proceeding as in Step 3 of the proof of Proposition \ref{prop:h.i.b} (thanks to \eqref{eq:CZ-B.vit.cov.1}), we deduce that
\begin{align}\label{eq:CZ-B.vit.cov.7}
	\mean{54\Omega_{j}}|V(Du)-V(Dh_j)|^2\,dx &\leq (\ep+c(\nu) R_j^{\Bbbk})\lambda^p +c(\ep,\nu) \lambda^p\delta^p +\nu c_{\beta} \lambda^p +\nu \lambda^p \nonumber\\
 &\leq (\ep+c(\nu) R_j^{\Bbbk})\lambda^p +c(\ep,\nu) \lambda^p\delta^p +\nu c \lambda^p \nonumber\\
	&=: S_1(\ep,\nu,\delta,R_j) \lambda^p,
\end{align}
where $c=c(\data_2,\gamma)$ and $c_{\beta}\simeq (\sum_{k=0}^{\infty}\beta_k)^\frac{p}{\theta}\leq c<\infty$.
Next, for the above $h_j$, let $w_j$ be the weak solution to problem \eqref{eq:2nd.eq.b} with $R=27R_j$ and $x_0=x_j$. Then, by the energy estimates \eqref{eq:energy.est.h.b} and \eqref{eq:energy.est.w.b}, we have
 \begin{align}\label{eq:CZ-B.vit.cov.8}
    \mean{27\Omega_j} |Dw_j|^p\,dx\leq c_1 \lambda^p. 
 \end{align}
 Moreover, from Lemma \ref{lem:2nd} and Proposition \ref{prop:h.i.b} (for $\rho_0=54R_j$ and $x_0=x_j$, there), we have 
 \begin{align}\label{eq:CZ-B.vit.cov.9}
  & \mean{27\Omega_j} |V(Dw_j)-V(Dh_j)|^2\,dx \nonumber\\
  &\leq c \; \delta^{\frac{\sigma}{p'(1+\sigma)}} \left[\left(\mean{54\Omega_{j}}|Du|^{p(1+\sigma)}\,dx\right)^{\frac{1}{1+\sigma}}+c\left(\mean{54\Omega_{j}}|Dg|^{p(1+\sigma)}\,dx\right)^{\frac{1}{1+\sigma}} \right]\nonumber\\
   &\leq c \; \delta^{\frac{\sigma}{p'(1+\sigma)}}  \left[\mean{162\Omega_{j}}|Du|^{p}\,dx +\left(\mean{162\Omega_{j}}(|F|^{p}+1)^{1+\sigma}\,dx\right)^\frac{1}{(1+\sigma)} +\left(\mean{162\Omega_{j}}|Dg|^{p(1+\sigma)}\,dx\right)^{\frac{1}{1+\sigma}} \right. \nonumber\\
		& \left. \qquad+ R_{j}^{(1-\varrho)p'} \left(\int_{\ern\setminus 162 Q_{j}}\dfrac{|u(y)-(u)_{162R_j}|^{q-1}}{|y-x_0|^{n+sq}}\,dy\right)^\frac{p}{m-1}  \right]
    \nonumber\\
   &\leq c \; \delta^{\frac{\sigma}{p'(1+\sigma)}} \lambda^p, 
 \end{align}
 where on the last line we used the exit-time condition \eqref{eq:CZ-B.vit.cov.1}.
 Now we consider two cases: 
\begin{align*}
27Q_{R_j}\not\subset\Omega\quad\text{or}\quad 27Q_{R_j}\subset\Omega.
\end{align*}
If $27Q_{R_j}\not\subset\Omega$ holds, then there exists a point $y_0\in\partial\Omega\cap 27Q_{R_j}$. By the assumption on $\Omega$ as in Definition \ref{def:BMO.b}, there exists a new coordinate system $\{x_1,\dots,x_n\}$ such that in this coordinate system $y_0$ is the origin and there holds
\begin{align*}
Q_{\frac{27}{4}R_j}(0)\cap\{x_1> \tfrac{27}{4}\delta R_j\}\subset\Omega\cap Q_{\frac{27}{4}R_j}(0)\subset Q_{\frac{27}{4}R_j}(0)\cap\{x_1>-\tfrac{27}{4}\delta R_j\}.
\end{align*}
Let the corresponding $v_j$ be as in Lemma \ref{lem:3rd}; i.e., solution to problem \eqref{eq:3rd.eq.b} defined in $Q_{\frac{27}{4}R_j}(0)\cap\{x_1> \tfrac{27}{4}\delta R_j\}$ and extended by $0$ in $Q_{\frac{27}{4}R_j}(0)\cap\{x_1\leq \tfrac{27}{4}\delta R_j\}$. 
Eventually, by \eqref{eq:CZ-B.vit.cov.7}, \eqref{eq:CZ-B.vit.cov.8}, \eqref{eq:CZ-B.vit.cov.9} and Lemma \ref{lem:4th}, we obtain
	\begin{align}\label{eq:CZ-B.vit.cov.5}
		\mean{6\Omega_{j}}|V(Du)-V(Dv_j)|^2\,dx 
		&\leq S(\ep,\nu,\delta,R_j) \lambda^p,
\end{align}
where
\begin{align*}
	S(\ep,\nu,\delta,R_j):=\ep+c(\nu) R_j^{\Bbbk} +c(\ep,\nu) \delta^p+c_* \; \delta^{\frac{\sigma}{p'(1+\sigma)}}  +c_*\nu
\end{align*}
with $c(\nu)=c(\nu,\data_2,\gamma)$, $c(\ep,\nu)=c(\ep,\nu,\data_2,\gamma)$, and $c_*=c_*(\data_2,\gamma)$ being universal constants. If $27Q_{R_j}\subset\Omega$ holds, then we proceed with $w_j$ instead of $v_j$ and the estimate \eqref{eq:CZ-B.vit.cov.5} holds with $w_j$ replaced by $v_j$ from the results of \cite{KimY1,BL1}.\\
\noindent\textbf{Step 4}: \textit{Estimate involving level sets}. We first obtain some uniform estimates for $|\Omega_j|$. On account of \eqref{eq:CZ-B.vit.cov.1}, similar to Section 3, we have the occurrence of at least one of \eqref{eq:lamb-case1} and \eqref{eq:lamb-case3}. Therefore, after some elementary simplifications, we get 
\begin{align}\label{eq:CZ-B.om-j.1}
	|\Omega_j|&\leq \frac{c(p)}{\lambda^p} \int_{\Omega_j\cap \{|Du|\geq \lambda/4\}}  |Du|^p \,dx \nonumber\\ &+\frac{c(p)}{(\delta\lambda)^{p(1+\sigma)}} \left[\int_{\Omega_j\cap \{(|F|^p+1)^{1/p}\geq \lambda/8\}} \big(|F|^p+1\big)^{1+\sigma}\,dx +\int_{Q_j\cap\{|Dg|\geq\lambda/8\}} |Dg|^{p(1+\sigma)}\,dx \right].
\end{align}
Next, we recall the following inequality (see \cite[Lemma 6.1]{KimY1}):
\begin{align*}
	&\text{Let }\xi,\zeta\in\ern \text{ be such that }|\xi|\geq \kappa \text{ for some }\kappa>0. \text{ Then for all }\alpha\geq p,\\
	&\qquad|\xi|^p\leq c|\xi-\zeta|^p +c(\alpha)\kappa^{p-\alpha}|\zeta|^\alpha.
\end{align*}
Consequently, if $27Q_{R_j}\not\subset\Omega$ holds, we obtain
\begin{align}\label{eq:CZ-B.vit.cov.4}
	\int_{5\Omega_j\cap \{|Du|\geq \kappa \lambda\}}  |Du|^p \,dx \leq c\int_{5\Omega_j\cap \{|Du|\geq \kappa \lambda\}}  \big(|Du-Dv_j|^p + (\kappa\lambda)^{-\gamma}|Dv_j|^{p+\gamma}\big)\,dx.
\end{align}
Moreover, from the $L^{\gamma}$ estimate of Lemma \ref{lem:ext.5} and the exit-time condition \eqref{eq:CZ-B.vit.cov.1}, we have 
		\begin{align}\label{eq:CZ-B.vit.cov.4.1}
			\left(\mean{5\Omega_j}|Dv_j|^{p+\gamma} \right)^\frac{1}{p+\gamma} \lesssim \left(\mean{6\Omega_j}|Dv_j|^{p} \right)^\frac{1}{p} 
			\lesssim \left(\mean{6\Omega_j}|Du|^{p} \right)^\frac{1}{p} + \left(\mean{6\Omega_j}|Du-Dv_j|^{p} \right)^\frac{1}{p}\lesssim \lambda,
	\end{align}
where we have also used the estimate of \eqref{eq:CZ-B.vit.cov.5} and \eqref{eq:V,theta}.
 Therefore, \eqref{eq:CZ-B.om-j.1} and \eqref{eq:CZ-B.vit.cov.4} imply that
	\begin{align*}
		&\int_{5\Omega_j\cap \{|Du|\geq \kappa \lambda\}}  |Du|^p \,dx \nonumber\\
		&\leq \big(S(\ep,\nu,\delta,R_j) + \kappa^{-\gamma}\big)\lambda^p|5\Omega_j| \nonumber\\
		&\lesssim \big(S(\ep,\nu,\delta,R_j) + \kappa^{-\gamma}\big) \left[ \int_{\Omega_j\cap \{|Du|\geq \lambda/4\}}  |Du|^p \,dx +\frac{|\Omega_j|}{|Q_j|}\frac{\lambda^{-p\sigma}}{\delta^{p(1+\sigma)}} \int_{Q_j\cap\{|Dg|\geq\lambda/8\}} |Dg|^{p(1+\sigma)}\,dx \nonumber \right.\\
		&\left.\quad+ 	\frac{\lambda^{-p\sigma}}{\delta^{p(1+\sigma)}}\int_{\Omega_j\cap \{(|F|^p+1)^{1/p}\geq \lambda/8\}} \big(|F|^p+1\big)^{1+\sigma}\,dx\right].
	\end{align*}
 On the other hand, if $27Q_{R_j}\subset\Omega$ holds, \eqref{eq:CZ-B.vit.cov.4} and \eqref{eq:CZ-B.vit.cov.4.1} are obtained with $Dw_j$ instead of $Dv_j$ and so we have the above estimate again. Since $\{Q_j\}$ is a disjoint collection such that each $Q_j$ is contained in  $Q_{r_2}(x_0)$, the above expression and \eqref{eq:CZ-B.vit.cov.0} yield
	\begin{align}\label{eq:CZ-B.vit.cov.3}
		&\int_{\Omega_{r_1}(x_0)\cap \{|Du|\geq \kappa \lambda\}}  |Du|^p \,dx \nonumber\\
		&\leq	\sum_{j=1}^{\infty}\int_{5\Omega_j\cap \{|Du|\geq \kappa \lambda\}}  |Du|^p \,dx \nonumber\\
		&\lesssim \big(S(\ep,\nu,\delta,\rho_0) + \kappa^{-\gamma}\big) \left[\int_{\Omega_{r_2}(x_0)\cap\{|Du|>\kappa\lambda\}}|Du|^{p}\,dx+\frac{\lambda^{-p\sigma}}{\delta^{p(1+\sigma)}} \int_{Q_{r_2}(x_0)\cap\{|Dg|>\hat\kappa\lambda\}}|Dg|^{p(1+\sigma)}\,dx \nonumber\right.\\
		& \left.\quad+\frac{\lambda^{-p\sigma}}{\delta^{p(1+\sigma)}} \int_{\Omega_{r_2}(x_0)\cap\{(|F|^p+1)^{1/p}>\tilde\kappa\lambda\}}\big(|F|^p+1\big)^{1+\sigma}\,dx\right].
	\end{align}
	\textbf{Step 5}: \textit{Conclusion}. For $t\geq 0$, we define the truncated function
	\begin{align*}
		[|Du|]_t:=\min\{|Du|,t\}.
	\end{align*}
	Then, for $t\geq  \lambda_1$, from \eqref{eq:CZ-B.vit.cov.3}, we easily obtain
	\begin{align*}
		\int_{\lambda_1}^{t}\lambda^{\gamma-1-p}&\int_{\Omega_{r_1}(x_0)\cap \{|Du|\geq \kappa \lambda\}}  |Du|^p \,dx d\lambda \nonumber\\
		&\lesssim \big(S(\ep,\nu,\delta,\rho_0) + \kappa^{-\gamma}\big) \left[\int_{\lambda_1}^{t}\lambda^{\gamma-1-p}\int_{\Omega_{r_2}(x_0)\cap\{|Du|>\kappa\lambda\}}|Du|^{p}\,dx d\lambda \nonumber\right.\\
		& \left.\quad+ \int_{\lambda_1}^{t}\lambda^{\gamma-1-p(1+\sigma)} \int_{Q_{r_2}(x_0)\cap\{|Dg|>\hat\kappa\lambda\}}\frac{|Dg|^{p(1+\sigma)}}{\delta^{p(1+\sigma)}}\,dx d\lambda \nonumber\right.\\
		& \left.\quad+ \int_{\lambda_1}^{t}\lambda^{\gamma-1-p(1+\sigma)}\int_{\Omega_{r_2}(x_0)\cap\{(|F|^p+1)^{1/p}>\tilde\kappa\lambda\}}\frac{(|F|^p+1)^{1+\sigma}}{\delta^{p(1+\sigma)}}\,dx d\lambda\right].
	\end{align*}
	Now, by Fubini's theorem and a simple change of variable (see, e.g., \cite{col-ming-JFA}), we can deduce that
	\begin{align}\label{eq:CZ-B.final.2}
		\mean{\Omega_{r_1}(x_0)}& [|Du|]_t^{\gamma-p} |Du|^p \,dx \nonumber\\
		&\leq c\big(S(\ep,\nu,\delta,\rho_0) + \kappa^{-\gamma}\big) \left[\mean{\Omega_{r_2}(x_0)}[|Du|]_t^{\gamma-p}|Du|^{p}\,dx+\ \mean{Q_{r_2}(x_0)}\frac{|Dg|^{\gamma}}{\delta^{\gamma}}\,dx \nonumber\right.\\
		& \left.\quad+ \mean{\Omega_{r_2}(x_0)}\frac{(|F|^p+1)^{\gamma/p}}{\delta^{\gamma}}\,dx\right]+c\kappa^{\gamma}\lambda_1^\gamma,
	\end{align}
	where $c$ depends only on $\data_2$ and $\gamma$. 
 We choose $\kappa=\kappa(\data_2,\gamma)>1$ large enough, $\ep,\nu\in (0,1)$, $\delta\in (0,\delta_0)$ small enough depending only on $\data_2$, and then $\gamma$, and $r_*\in (0,1)$ small enough depending on $\data_2,\gamma$ and $\nu$ (in the same order)  such that  
	\begin{align*}
		c\kappa^{-\gamma} \leq\frac{1}{16}, \;\; c\ep\leq \frac{1}{16}, \;\; c\nu \leq \frac{1}{16}, \;\; c_* c(\ep,\nu)\delta^{\frac{\sigma}{p'(1+\sigma)}}\leq \frac{1}{16}\;\;\mbox{and}\;\; c_*c(\nu)r_*^{\Bbbk}\leq\frac{1}{16}.
	\end{align*}
Consequently, for all $\rho_0\leq r_*$, we have
\begin{align*}
    S(\ep,\nu,\delta,\rho_0) + \kappa^{-\gamma} \leq \frac{1}{2}
\end{align*}
and noting \eqref{eq:CZ-B.upper.lamb}, by a standard iteration argument, we deduce from \eqref{eq:CZ-B.final.2} that
\begin{align*}
    \mean{\Omega_{\frac{\rho_0}{3}}(x_0)} &[|Du|]_t^{\gamma-p} |Du|^p \,dx \nonumber\\
		&\leq  c \mean{Q_{\rho_0}(x_0)}|Dg|^{\gamma}\,dx + \mean{\Omega_{\rho_0}(x_0)}(|F|^p+1)^{\gamma/p}\,dx+c\, \Xi_0(x_0,2\rho_0)^\gamma.
\end{align*}
Taking the limit as $t\to\infty$ in the above expression, by a standard covering argument as in the proof of Proposition \ref{prop:h.i.b}, we obtain the required result of Theorem \ref{thm:b} (where we also use $p(1+\sigma)<\gamma$ and H\"older's inequality in $\Xi_0(x_0,2\rho_0)$). \qed

\medskip	
\textbf{Data availability:} Data sharing not applicable to this article as no datasets were generated or analysed during
the current study.


\begin{thebibliography}{10}

\bibitem{AM1}
E.~Acerbi and G.~Mingione, \emph{Gradient estimates for a class of parabolic
  systems}, Duke Math. J. \textbf{136} (2007), no.~2, 285--320.

\bibitem{BBL}
S.~Baasandorj, S.-S. Byun, and H.-S. Lee, \emph{Global gradient estimates for a
  general class of quasilinear elliptic equations with {O}rlicz growth}, Proc.
  Amer. Math. Soc. \textbf{149} (2021), no.~10, 4189--4206.

\bibitem{BBDL}
A.~Kh. Balci, S.-S. Byun, L.~Diening, and H.-S. Lee, \emph{Global maximal
  regularity for equations with degenerate weights}, arXiv preprint
  arXiv:2201.03524 (2022).

\bibitem{BX}
J.~Bao and J.~Xiong, \emph{Sharp regularity for elliptic systems associated
  with transmission problems}, Potential Anal. \textbf{39} (2013), no.~2,
  169--194.

\bibitem{Biag-prse}
S.~Biagi, S.~Dipierro, E.~Valdinoci, and E.~Vecchi, \emph{Semilinear elliptic
  equations involving mixed local and nonlocal operators}, Proc. Roy. Soc.
  Edinburgh Sect. A \textbf{151} (2021), no.~5, 1611--1641.

\bibitem{Biag-cpde}
\bysame, \emph{Mixed local and nonlocal elliptic operators: regularity and
  maximum principles}, Comm. Partial Differential Equations \textbf{47} (2022),
  no.~3, 585--629.

\bibitem{Biag-math-eng}
\bysame, \emph{A {H}ong-{K}rahn-{S}zeg\"{o} inequality for mixed local and
  nonlocal operators}, Math. Eng. \textbf{5} (2023), no.~1, Paper No. 014, 25.

\bibitem{BLS}
L.~Brasco, E.~Lindgren, and A.~Schikorra, \emph{Higher {H}\"{o}lder regularity
  for the fractional {$p$}-{L}aplacian in the superquadratic case}, Adv. Math.
  \textbf{338} (2018), 782--846.

\bibitem{BD}
T.~A. Bui and X.~T. Duong, \emph{Global {L}orentz estimates for nonlinear
  parabolic equations on nonsmooth domains}, Calc. Var. Partial Differential
  Equations \textbf{56} (2017), no.~2, Paper No. 47, 24.

\bibitem{BKO}
S.-S. Byun, H.~Kim, and J.~Ok, \emph{Local h{\"o}lder continuity for fractional
  nonlocal equations with general growth}, Math. Ann.,
  https://doi.org/10.1007/s00208-022-02472-y, to appear.

\bibitem{BK-CZ-pfrac}
S.-S. Byun and K.~Kim, \emph{${L}^{q}$ estimates for nonlocal p-{L}aplacian
  type equations with {BMO} kernel coefficients in divergence form},
  arXiv:2303.08517 (2023).

\bibitem{BKK-PubM}
S.-S. Byun, K.~Kim, and D.~Kumar, \emph{Regularity results for a class of
  nonlocal double phase equations with {VMO} coefficients}, arXiv:2303.07749
  (2023), to appear in Pub. Math.

\bibitem{BK}
S.-S. Byun and Y.~Kim, \emph{Elliptic equations with measurable nonlinearities
  in nonsmooth domains}, Adv. Math. \textbf{288} (2016), 152--200.

\bibitem{BL1}
S.-S. Byun and H.-S. Lee, \emph{Optimal regularity for elliptic equations with
  measurable nonlinearities under nonstandard growth}, Int. Math. Res. Not.
  IMRN, https://doi.org/10.1093/imrn/rnad040, to appear.

\bibitem{BLS2}
S.-S. Byun, H.-S. Lee, and K.~Song, \emph{Regularity results for mixed local
  and nonlocal double phase functionals}, arXiv preprint arXiv:2301.06234
  (2023).

\bibitem{BOS}
S.-S. Byun, J.~Ok, and K.~Song, \emph{H\"{o}lder regularity for weak solutions
  to nonlocal double phase problems}, J. Math. Pures Appl. (9) \textbf{168}
  (2022), 110--142.

\bibitem{byun-song}
S.-S. Byun and K.~Song, \emph{Mixed local and nonlocal equations with measure
  data}, Calc. Var. Partial Differential Equations \textbf{62} (2023), no.~1,
  Paper No. 14, 35.

\bibitem{BW0}
S.-S. Byun and L.~Wang, \emph{Elliptic equations with {BMO} coefficients in
  {R}eifenberg domains}, Comm. Pure Appl. Math. \textbf{57} (2004), no.~10,
  1283--1310.

\bibitem{BW1}
\bysame, \emph{Elliptic equations with measurable coefficients in {R}eifenberg
  domains}, Adv. Math. \textbf{225} (2010), no.~5, 2648--2673.

\bibitem{CP}
L.~Caffarelli and I.~Peral, \emph{On {$W^{1,p}$} estimates for elliptic
  equations in divergence form}, Comm. Pure Appl. Math. \textbf{51} (1998),
  no.~1, 1--21.

\bibitem{caff-silv}
L.~Caffarelli and L.~Silvestre, \emph{Regularity results for nonlocal equations
  by approximation}, Arch. Ration. Mech. Anal. \textbf{200} (2011), no.~1,
  59--88.

\bibitem{chaker}
J.~Chaker, M.~Kim, and M.~Weidner, \emph{Regularity for nonlocal problems with
  non-standard growth}, Calc. Var. Partial Differential Equations \textbf{61}
  (2022), no.~6, Paper No. 227, 31.

\bibitem{C2}
Y.~Cho, \emph{Global gradient estimates for divergence-type elliptic problems
  involving general nonlinear operators}, J. Differential Equations
  \textbf{264} (2018), no.~10, 6152--6190.

\bibitem{col-ming-JFA}
M.~Colombo and G.~Mingione, \emph{Calder\'{o}n-{Z}ygmund estimates and
  non-uniformly elliptic operators}, J. Funct. Anal. \textbf{270} (2016),
  no.~4, 1416--1478.

\bibitem{defilip-ming-mixed}
C.~De~Filippis and G.~Mingione, \emph{{G}radient regularity in mixed local and
  nonlocal problems}, Math. Ann., https://doi.org/10.1007/s00208-022-02512-7,
  to appear.

\bibitem{defil-pala}
C.~De~Filippis and G.~Palatucci, \emph{H\"{o}lder regularity for nonlocal
  double phase equations}, J. Differential Equations \textbf{267} (2019),
  no.~1, 547--586.

\bibitem{DKP}
A.~Di~Castro, T.~Kuusi, and G.~Palatucci, \emph{Local behavior of fractional
  {$p$}-minimizers}, Ann. Inst. H. Poincar\'{e} C Anal. Non Lin\'{e}aire
  \textbf{33} (2016), no.~5, 1279--1299.

\bibitem{D2}
G.~Di~Fazio, \emph{{$L^p$} estimates for divergence form elliptic equations
  with discontinuous coefficients}, Boll. Un. Mat. Ital. A (7) \textbf{10}
  (1996), no.~2, 409--420.

\bibitem{DPV}
E.~Di~Nezza, G.~Palatucci, and E.~Valdinoci, \emph{Hitchhiker's guide to the
  fractional {S}obolev spaces}, Bull. Sci. Math. \textbf{136} (2012), no.~5,
  521--573.

\bibitem{DFTW}
L.~Diening, M.~Fornasier, R.~Tomasi, and M.~Wank, \emph{A relaxed {K}a\v{c}anov
  iteration for the {$p$}-{P}oisson problem}, Numer. Math. \textbf{145} (2020),
  no.~1, 1--34.

\bibitem{DiNo-CZ}
L.~Diening and S.~Nowak, \emph{Calderón-{Z}ygmund estimates for the fractional
  $p$-{L}aplacian}, arXiv:2303.02116 (2023).

\bibitem{Dip-Nem-logst}
S.~Dipierro, E.~Proietti~Lippi, and E.~Valdinoci, \emph{({N}on) local logistic
  equations with {N}eumann conditions}, Ann. Inst. H. Poincar\'{e} C Anal. Non
  Lin\'{e}aire, to appear.

\bibitem{Dip-Val-phy}
S.~Dipierro and E.~Valdinoci, \emph{Description of an ecological niche for a
  mixed local/nonlocal dispersal: an evolution equation and a new {N}eumann
  condition arising from the superposition of {B}rownian and {L}\'{e}vy
  processes}, Phys. A \textbf{575} (2021), Paper No. 126052, 20.

\bibitem{DK}
H.~Dong and D.~Kim, \emph{Elliptic equations in divergence form with partially
  {BMO} coefficients}, Arch. Ration. Mech. Anal. \textbf{196} (2010), no.~1,
  25--70.

\bibitem{ERS}
J.~Elschner, J.~Rehberg, and G.~Schmidt, \emph{Optimal regularity for elliptic
  transmission problems including {$C^1$} interfaces}, Interfaces Free Bound.
  \textbf{9} (2007), no.~2, 233--252.

\bibitem{FMSY}
M.~M. Fall, T.~Mengesha, A.~Schikorra, and S.~Yeepo,
  \emph{Calder\'{o}n-{Z}ygmund theory for non-convolution type nonlocal
  equations with continuous coefficient}, Partial Differ. Equ. Appl. \textbf{3}
  (2022), no.~2, Paper No. 24, 27.

\bibitem{GK1}
P.~Garain and J.~Kinnunen, \emph{On the regularity theory for mixed local and
  nonlocal quasilinear elliptic equations}, Trans. Amer. Math. Soc.
  \textbf{375} (2022), no.~8, 5393--5423.

\bibitem{Garain-cvpde}
P.~Garain and E.~Lindgren, \emph{Higher {H}\"{o}lder regularity for mixed local
  and nonlocal degenerate elliptic equations}, Calc. Var. Partial Differential
  Equations \textbf{62} (2023), no.~2, Paper No. 67, 36.

\bibitem{GKS}
J.~Giacomoni, D.~Kumar, and K.~Sreenadh, \emph{Interior and boundary regularity
  results for strongly nonhomogeneous $p,q$-fractional problems}, Adv. Calc.
  Var. (2021), pp. 35.

\bibitem{GKS2}
\bysame, \emph{Global regularity results for non-homogeneous growth fractional
  problems}, J. Geom. Anal. \textbf{32} (2022), no.~1, Paper No. 36, pp. 41.

\bibitem{IMS}
A.~Iannizzotto, S.~Mosconi, and M.~Squassina, \emph{Global {H}\"{o}lder
  regularity for the fractional {$p$}-{L}aplacian}, Rev. Mat. Iberoam.
  \textbf{32} (2016), no.~4, 1353--1392.

\bibitem{KT}
C.~E. Kenig and T.~Toro, \emph{Free boundary regularity for harmonic measures
  and {P}oisson kernels}, Ann. of Math. (2) \textbf{150} (1999), no.~2,
  369--454.

\bibitem{KK}
T.~Kilpel{\"a}inen and P.~Koskela, \emph{Global integrability of the gradients
  of solutions to partial differential equations}, Nonlinear Anal. \textbf{23}
  (1994), no.~7, 899--909.

\bibitem{KK2}
D.~Kim and N.~V. Krylov, \emph{Elliptic differential equations with
  coefficients measurable with respect to one variable and {VMO} with respect
  to the others}, SIAM J. Math. Anal. \textbf{39} (2007), no.~2, 489--506.

\bibitem{KimY1}
Y.~Kim, \emph{Gradient estimates for elliptic equations with measurable
  nonlinearities}, J. Math. Pures Appl. (9) \textbf{114} (2018), 118--145.

\bibitem{KR}
Y.~Kim and S.~Ryu, \emph{Global gradient estimates for parabolic equations with
  measurable nonlinearities}, Nonlinear Anal. \textbf{164} (2017), 77--99.

\bibitem{KZ}
J.~Kinnunen and S.~Zhou, \emph{A local estimate for nonlinear equations with
  discontinuous coefficients}, Comm. Partial Differential Equations \textbf{24}
  (1999), no.~11-12, 2043--2068.

\bibitem{KMS}
T.~Kuusi, G.~Mingione, and Y.~Sire, \emph{Nonlocal self-improving properties},
  Anal. PDE \textbf{8} (2015), no.~1, 57--114.

\bibitem{LMS}
A.~Lemenant, E.~Milakis, and L.~V. Spinolo, \emph{On the extension property of
  {R}eifenberg-flat domains}, Ann. Acad. Sci. Fenn. Math. \textbf{39} (2014),
  no.~1, 51--71.

\bibitem{M1}
O.~Martio, \emph{Reflection principle for solutions of elliptic partial
  differential equations and quasiregular mappings}, Ann. Acad. Sci. Fenn. Ser.
  A I Math. \textbf{6} (1981), no.~1, 179--187.

\bibitem{MP}
T.~Mengesha and N.~C. Phuc, \emph{Global estimates for quasilinear elliptic
  equations on {R}eifenberg flat domains}, Arch. Ration. Mech. Anal.
  \textbf{203} (2012), no.~1, 189--216.

\bibitem{MSY}
T.~Mengesha, A.~Schikorra, and S.~Yeepo, \emph{Calderon-{Z}ygmund type
  estimates for nonlocal {PDE} with {H}\"{o}lder continuous kernel}, Adv. Math.
  \textbf{383} (2021), Paper No. 107692, 64.

\bibitem{M0}
N.~G. Meyers, \emph{An {$L^{p}$}-estimate for the gradient of solutions of
  second order elliptic divergence equations}, Ann. Scuola Norm. Sup. Pisa Cl.
  Sci. (3) \textbf{17} (1963), 189--206.

\bibitem{nakamura}
K.~Nakamura, \emph{Harnack's estimate for a mixed local-nonlocal doubly
  nonlinear parabolic equation}, Calc. Var. Partial Differential Equations
  \textbf{62} (2023), no.~2, Paper No. 40, 45.

\bibitem{Now-AAG}
S.~Nowak, \emph{Higher integrability for nonlinear nonlocal equations with
  irregular kernel}, Analysis and partial differential equations on manifolds,
  fractals and graphs, Adv. Anal. Geom., vol.~3, De Gruyter, Berlin, 2021,
  pp.~459--492.

\bibitem{P1}
N.~C. Phuc, \emph{Nonlinear {M}uckenhoupt--{W}heeden type bounds on
  {R}eifenberg flat domains, with applications to quasilinear {R}iccati type
  equations}, Adv. Math. \textbf{250} (2014), 387--419.

\bibitem{shang-zhang-par}
B.~Shang and C.~Zhang, \emph{{H}\"older regularity for mixed local and nonlocal
  $ p$-{L}aplace parabolic equations}, arXiv:2112.08698 (2021).

\end{thebibliography}

 \providecommand{\bysame}{\leavevmode\hbox to3em{\hrulefill}\thinspace}
\providecommand{\MR}{\relax\ifhmode\unskip\space\fi MR }
\providecommand{\MRhref}[2]{%
  \href{http://www.ams.org/mathscinet-getitem?mr=#1}{#2}
}
\providecommand{\href}[2]{#2}

\end{document}